\documentclass[a4paper,10 pt]{amsart}

\usepackage[english]{babel}

\usepackage{amscd}
\usepackage[arrow, matrix, curve]{xy}
\usepackage{color}
\usepackage{hyperref}

\newtheorem{Satz}{Satz}[section]
\newtheorem{Lemma}[Satz]{Lemma}
\newtheorem{Corollary}[Satz]{Corollary}

\theoremstyle{definition}
\newtheorem{Definition}{Definition}
\newtheorem{Remark}{Remark}
\newtheorem{Theorem}[Satz]{Theorem}
\newtheorem{Example}[Satz]{Example}
\newtheorem{proposition}[Satz]{Proposition}

\newtheorem{Fact}[Satz]{Fact}
\newtheorem{Conjecture}[Satz]{Conjecture}

\usepackage{amssymb}
\usepackage{fullpage}

\DeclareMathOperator{\Hom}{Hom}

\DeclareMathOperator{\End}{End}

\DeclareMathOperator{\Lie}{Lie}
\DeclareMathOperator{\Spec}{Spec}
\DeclareMathOperator{\Spa}{Spa}
\DeclareMathOperator{\Spd}{Spd}
\DeclareMathOperator{\Spf}{Spf}
\DeclareMathOperator{\Nilp}{Nilp}

\DeclareMathOperator{\GL}{GL}

\DeclareMathOperator{\Isom}{Isom}
\DeclareMathOperator{\Displ}{Displ}

\DeclareMathOperator{\Acris}{\mathbb{A}_{\cris}}
\DeclareMathOperator{\Ainf}{\mathbb{A}_{\infintesimal}}
\DeclareMathOperator{\Rad}{Rad}
\DeclareMathOperator{\MaxSpec}{MaxSpec}
\DeclareMathOperator{\Fil}{Fil}

\DeclareMathOperator{\Ad}{Ad}

\DeclareMathOperator{\Rep}{Rep}

\DeclareMathOperator{\Bun}{Bun}
\DeclareMathOperator{\G}{\mathcal{G}}

\DeclareMathOperator{\Proj}{Proj}

\DeclareMathOperator{\Z}{\mathbb{Z}}
\DeclareMathOperator{\colim}{colim}

\DeclareMathOperator{\pr}{pr}

\DeclareMathOperator{\Mat}{Mat}
\DeclareMathOperator{\Frac}{Frac}
\DeclareMathOperator{\Perf}{Perf}
\DeclareMathOperator{\Vect}{Vect}
\DeclareMathOperator{\Affd}{Affd}
\DeclareMathOperator{\SmAffd}{SmAffd}
\DeclareMathOperator{\Sh}{Sh}
\DeclareMathOperator{\Ann}{Ann}
\DeclareMathOperator{\Set}{Set}
\DeclareMathOperator{\BP}{BP}
\DeclareMathOperator{\Bcris}{\mathbb{B}^{+}_{\cris}}
\DeclareMathOperator{\BKF}{BKF}
\DeclareMathOperator{\Win}{Win}
\DeclareMathOperator{\aff}{aff}
\DeclareMathOperator{\et}{\acute{e}t}
\DeclareMathOperator{\gr}{gr}
\DeclareMathOperator{\fpqc}{fpqc}
\DeclareMathOperator{\op}{op}
\DeclareMathOperator{\Grp}{Grp}
\DeclareMathOperator{\nilp}{nilp}
\DeclareMathOperator{\id}{id}
\DeclareMathOperator{\triv}{triv}
\DeclareMathOperator{\adic}{adic}
\DeclareMathOperator{\cris}{cris}

\DeclareMathOperator{\banal}{banal}
\DeclareMathOperator{\infintesimal}{inf}

\usepackage{relsize} 
\usepackage[bbgreekl]{mathbbol} 
\usepackage{amsfonts} 
\usepackage[T1]{fontenc}
\DeclareSymbolFontAlphabet{\mathbb}{AMSb} 
\DeclareSymbolFontAlphabet{\mathbbl}{bbold} 
\newcommand{\Prism}{{\mathlarger{\mathbbl{\Delta}}}}

\begin{document}
\pagenumbering{gobble}

\begin{titlepage}

\begin{center}

\title{$\G$-$\mu$-Displays and local Shtuka}
\author{Sebastian Bartling}
\maketitle
\end{center}

Abstract: Two approaches to the construction of integral models of local Shimura-varieties are compared: that of Bültel-Pappas using $\G$-$\mu$-displays and that of Scholze using local mixed-characteristic shtuka. As an application, the representability of the diamond of the generic fiber of the Bültel-Pappas moduli problem is verified and a local representability conjecture of Rapoport-Pappas is proven. The methods apply to all unramified local Shimura-data under a certain condition on slopes and if $p\neq 2;$ thus in particular also to exceptional groups.

\tableofcontents
\end{titlepage}
\newpage

\pagenumbering{arabic}
\setcounter{page}{1}
\section{Introduction}
The theory of displays was introduced by Zink in \cite{ZinkDisplay} to give a Dieudonné-type classification for formal $p$-divisible groups over general $p$-adic rings. This was motivated by an old paper of Norman \cite{NormanAlgorithm} and Zink succeeded in proving that formal $p$-divisible groups are classified by nilpotent displays over a very large class of rings, for example those $p$-adic rings $R,$ such that $R/pR$ is a finitely generated $\mathbb{F}_{p}$-algebra. Later Lau proved in \cite{LauInventiones} that this classification result in fact extends to all $p$-adic rings.
\\
To give a display is to give a filtered, finitely generated projective module over the $p$-typical Witt vectors of the $p$-adic ring, one is working over, together with certain Frobenius-linear operators. This data is required to satisfy certain axioms. Instead of recalling this precisely, I refer the reader to Zink's definition \cite[Definition 1]{ZinkDisplay} and give the following slogans:
\begin{enumerate}
\item[(a)] a display is (locally on the ring) a Frobenius-twisted conjugacy class of matrices,
\item[(b)] a display is a mixed-characteristic analog of Drinfeld$^{\prime}$s shtukas with one leg.
\end{enumerate}
Let me make this more precise: concerning (a), consider the group $\GL_{n}$ over $\mathbb{Z}_{p}$ and let $L^{+}\GL_{n}$ be the Greenberg-transform with points $L^{+}\GL_{n}(R)=\GL_{n}(W(R)).$ Next, let $d\leq n$ be an integer and consider the subgroup scheme $\GL_{n}(\mathcal{W}(.))_{n,d}\subseteq L^{+}\GL_{n}$ of matrices of the form
$$
\begin{bmatrix}
A & B \\
J & C
\end{bmatrix}
$$
where $J$ is a block matrix with entries in $I(R)=\text{Im}(V\colon W(R)\rightarrow W(R))$ and $A$ has size $d\times d$ and $C$ has size $n-d\times n-d.$
Define the Frobenius-linear group homomorphism
$$
\Phi_{n,d}\colon \GL_{n}(\mathcal{W}(R))_{n,d}\rightarrow L^{+}\GL_{n}(R),
$$
by the formula
$$
\Phi_{n,d}\left(\begin{bmatrix}
A & B \\
J & C
\end{bmatrix}\right)=\begin{bmatrix}
F(A) & pF(B) \\
V^{-1}(J) & F(C)
\end{bmatrix}.
$$
Then the precise formulation of slogan (a) is the following statement again due to Zink.
\begin{Fact}(Zink)\label{Displays sind lokal Konjugationsklassen von Matrizen}
Let $R$ be a $p$-adic ring. 
Then the groupoid of displays (locally of height $n$ and dimension $d$) over $R$ is equivalent to the groupoid of $\GL_{n}(\mathcal{W}(.))_{n,d}$-torsors $Q$ together with a sheaf-morphism
$$
\alpha\colon Q\rightarrow L^{+}\GL_{n},
$$
such that
$$
\alpha(q\cdot h)=h^{-1}\alpha(q)\Phi_{n,d}(h),
$$
for all $q\in Q$ and $h\in \GL_{n}(\mathcal{W}(.))_{n,d}.$
\end{Fact}
Note that once one adds the Zink-nilpotency condition on displays, this proposition together with the statement that displays indeed classify formal $p$-divisible groups, gives a very explicit description of formal $p$-divisible groups (of fixed height and dimension) purely in terms of group theory!
\\
The meaning of slogan (b) remains a bit vague over general base rings (e.g. what should play here really the role of the leg?). If one instead considers only perfect rings in characteristic $p,$ then one can explain this slogan by the following easy
\begin{Fact}
Let $R$ be a perfect $\mathbb{F}_{p}$-algebra.
\\
Then the category of displays over $R$ embeds fully faithfully into the category of finitely generated projective $W(R)$-modules $P$ together with a $F$-linear morphism
$$
\varphi_{P}\colon P\rightarrow P,
$$
such that $\varphi_{P}[1/p]\colon P\otimes_{W(R)} W(R)[1/p]\rightarrow P\otimes_{W(R)} W(R)[1/p]$ is a $F$-linear isomorphism.
\end{Fact}
The datum of a pair $(P,\varphi_{P})$ has obvious similiarity with a (local) shtuka with only one leg, but this interpretation works only if $R$ is a perfect ring in characteristic $p.$
\\
Regarding the mixed-characteristic case, Scholze has used perfectoid geometry to introduce $p$-adic shtuka in his Berkeley lectures \cite{Berkeleylectures}, see \cite[Def. 11.4.1.]{Berkeleylectures}. A theorem of Fargues then asserts that in case one is working over a geometric perfectoid point, $p$-divisible groups are indeed classified by local mixed-characteristic shtukas with one leg and a minuscule bound on the singularity at that leg, see \cite[Thm. 14.1.1.]{Berkeleylectures}. If one restricts to formal $p$-divisible groups, one can thus use Scholze$^{\prime}$s mixed-characteristic local shtukas to give a precise meaning to slogan (b) over perfectoid base rings.
\\
One of the aims of this article is to extend this picture to formal $p$-divisible groups carrying additional structures and beyond.
\\
The wish to study mixed-characteristic shtuka might have its origins in trying to construct their moduli in the hope of achieving geometric realizations of local Langlands correspondences.  Propagated by an article of Rapoport-Viehmann \cite{RapoViehmann} and deeply inspired by Deligne$^{\prime}$s group theoretical reformulation (and extension) of the theory of Shimura-varieties \cite{DeligneBourbakiShimura}, an ongoing project right now is therefore to rewrite the theory of Rapoport-Zink spaces and their generic fibers in terms of group theory. In particular, this means to extend the theory beyond the cases of (P)EL-type considered by Rapoport-Zink \cite{RZ}. The part of this program concerning the generic fiber was first achieved by Scholze in his Berkeley lectures \cite{Berkeleylectures}, where he constructed general local Shimura-varieties as rigid analytic spaces. His methods of construcing these spaces used crucially his theory of diamonds, the relative Fargues-Fontaine curve and foundational results of Kedlaya-Liu (\cite{KedlayaLiuFoundations}); it therefore only makes sense on the generic fiber.
\\
An open problem is therefore to investigate the existence of integral models for local Shimura-varieties, beyond the PEL-cases considered by Rapoport-Zink; or more generally beyond the abelian case (where recently a lot of progress has been made by Rapoport-Pappas \cite{RapoPappasShtuka}, \cite{RapoPappasIntegralModels}): these 'exotic' cases are interesting because these spaces are not in an obvious way linked to $p$-divisible groups anymore.\footnote{Compare however with the work of Bültel \cite{bueltel2013pel}, where he is nevertheless still able in specific situations to make a link with $p$-divisible groups in the case of global Shimura varieties - I thank Oliver Bültel for making me aware of this!} Although Scholze proposes an integral model in \cite[Section 25]{Berkeleylectures}, he is only able to make sense of this on the level of $v$-sheaves, which is still strong enough to characterize a conjectural flat and normal formal scheme, whose associated $v$-sheaf should be the moduli problem introduced by Scholze in \cite[Def. 25.1.1]{Berkeleylectures}. In the following, let me concentrate on the case when the level structure is maximal hyperspecial.
\\
This is where the work of Bültel-Pappas \cite{BueltelPappas} ties in: their approach is different than that of Scholze in that it works more classically in the world of formal schemes; and more importantly: not just for integral perfectoid rings. To explain their ideas, one first simply notes that in case one is considering Rapoport-Zink spaces, where one is deforming a formal $p$-divisible group, one can use the classification result of Zink and Lau, to reformulate the Rapoport-Zink moduli problem of $p$-divisible groups completely in terms of nilpotent displays. Therefore, to generalize Rapoport-Zink spaces, one may try to generalize displays. But then one can use the perspective on displays given by fact \ref{Displays sind lokal Konjugationsklassen von Matrizen}, to see that to generalize displays one just has to generalize the group $\GL_{n}(\mathcal{W}(.))_{n,d}$ and the morphism $\Phi_{n,d}.$ This is exactly what Bültel \cite{bueltel2013pel} and later Bültel-Pappas were able to do for unramified reductive groups $\G$ over $\mathbb{Z}_{p}$ and a minuscule cocharacter $\mu.$\footnote{Although their theory is definable in greater generality than that of a reductive group scheme, one cannot expect it to give rise to the right objects. Note for example that Lau in \cite{LauHigherFrames} is also able to define $\G$-$\mu$-displays for non-minuscule cocharacter and general smooth $\mathbb{Z}_{p}$-groups. In cases of ramification one should use the local model to set up the theory of $\G$-$\mu$-displays and in that case one can e.g. compare with the recent work of Pappas in that direction, \cite{PappasIntegralDisplay}.} After having generalized displays to unramified groups, one can introduce moduli problems of these generalized displays. In the case of $\GL_{n}$ and a minusucle cocharacter they give back the raw Rapoport-Zink spaces, where one is deforming a formal $p$-divisible group. Then Bültel-Pappas are able to show that in the Hodge-type case, on locally noetherian test-objects, these moduli problems have indeed the expected shape of a (formally smooth) formal scheme, locally formally of finite type, see \cite[Thm. 5.1.3.]{BueltelPappas}. Nevertheless, they leave open the problem to investigate the relationship between their approach and that of Scholze. This is what is done in this article.
\subsection{Statement of the results}
I will now state more precisely the main results. To do this, one has to introduce a bit of set-up: fix a prime number $p$ and choose an algebraic closure $\overline{\mathbb{Q}}_{p}$ of $\mathbb{Q}_{p},$ which induces an algebraic closure $k=\overline{\mathbb{F}}_{p}.$ Denote by $\breve{\mathbb{Q}}_{p}=W(k)[1/p]$ the maximal unramified extension of $\mathbb{Q}_{p}.$ 
\\
Let $G$ be a connected reductive group over $\mathbb{Q}_{p},$ $\lbrace \mu \rbrace$ a $G(\overline{\mathbb{Q}}_{p})$-conjugacy class of minuscule cocharacters of $G_{\overline{\mathbb{Q}}_{p}}$ and $[b]\in B(G,\mu^{-1}).$ A triple $(G,\lbrace \mu \rbrace, [b])$ is called a local Shimura datum \cite[Def. 24.1.1]{Berkeleylectures}. Let $E=E(G,\lbrace \mu \rbrace)$ be the local reflex field, a finite extension of $\mathbb{Q}_{p}$. If $K\subset G(\mathbb{Q}_{p})$ is some compact open subgroup, Scholze has associated in \cite[Def. 24.1.3.]{Berkeleylectures} to $(G,\lbrace \mu \rbrace, [b], K)$ a uniquely determined smooth and partially proper rigid analytic variety $\Sh_{K}$ over $\Spa(\breve{E}).$ This is the local Shimura variety with level structure $K.$
\\
To make the link with the work of Bültel-Pappas, let me assume that $G$ is unramified and choose a reductive model $\G$ over $\mathbb{Z}_{p}.$ Then one can consider the maximal hyperspecial level structure $K=\G(\mathbb{Z}_{p})$ and the local Shimura variety is uniquely characterized by the fact that the associated diamond is the solution of a moduli problem of local shtuka (\cite[Def. 23.1.1]{Berkeleylectures}). The rigid analytic variety $\Sh_{K}$ is conjectured (c.f. \cite[Conjecture 5.2.]{PappasICM}) to admit a uniquely determined integral model.
\\
The moduli problem of $\G$-$\mu$-displays should give rise to the desired integral model if one adds a further assumption on the slopes of $[b]$. To define this moduli problem, let 
$$
\mu\colon \mathbb{G}_{m,\mathcal{O}_{E}}\rightarrow \G_{\mathcal{O}_{\breve{E}}}
$$
be an integral extension of a representative of the previously considered conjugacy class of geometric cocharacters. Its $\G(\mathcal{O}_{E})$-conjugacy class is well-defined. Assume that $[b]\in B(G,\mu)$ admits a representative $b=u_{0}\mu(p),$ where $u_{0}\in \G(W(k)).$ This element $u_{0}$ gives rise to a $\G$-$\mu$-display $\mathcal{P}_{0}$ over $k.$
\\
Bültel-Pappas introduce (\cite[section 4.2]{BueltelPappas}) a functor
$$
\mathcal{M}^{\text{BP}}\colon \Nilp_{W(k)}^{\text{op}}\rightarrow \Set
$$
classifying $\G$-$\mu$-displays, which modulo $p$ are quasi-isogenous to the base change of $\mathcal{P}_{0}.$ Assume in the following that $-1$ is not a slope of $\Ad(b).$ By an unpublished result of Oliver Bültel, this functor admits no non-trivial automorphism and is therefore a sheaf for the fpqc topology on $\Nilp_{W(k)}.$
\\
Making use of a construction of Scholze-Weinstein \cite{ScholzeWeinstein}, one can define a generic fiber of the functor $\mathcal{M}^{\text{BP}}$ without knowing its representability. This is a sheaf for the analytic topology 
$$
\mathcal{M}^{\text{BP}}_{\eta}\colon \text{CAffd}^{\text{op}}_{\Spa(\breve{E},\mathcal{O}_{\breve{E}})}\rightarrow \Set,
$$
where $\text{CAffd}_{\Spa(\breve{E},\mathcal{O}_{\breve{E}})}$ is the category of complete affinoid Huber-pairs (see subsection \ref{Subsection Recollections on the Bueltel-Pappas moduli problem}). Let $$(\mathcal{M}^{\BP}_{\eta})^{\lozenge}\rightarrow \Spd(\breve{E})$$ be the diamantine version of the previous sheaf (see subsection \ref{subsection Representability of the Diamond}). Then one of the main results of this article is the following
\begin{Theorem}{(Proposition \ref{Vergleich der Diamanten})}
\\
Assume that $p\geq 3.$
There is an isomorphism of analytic sheaves on $\text{AffdPerf}^{\text{op}}_{k}$
$$
\psi\colon (\mathcal{M}^{\BP}_{\eta})^{\lozenge}\simeq (\Sh_{K})^{\lozenge},
$$
over $\Spd(\breve{E}).$
\end{Theorem}
Furthermore, a conjecture of Rapoport-Pappas \cite[Conjecture 3.3.4]{RapoPappasShtuka} is verified: let $\mathcal{M}^{\text{int}}_{\text{Scholze}}$ be the $v$-sheaf introduced by Scholze as in \cite[Def. 25.1.1.]{Berkeleylectures}, which is associated to the data $(\G,\lbrace \mu \rbrace,[b])$ and the flag-variety as the local model. Using the main results of the PhD-thesis of Gleason \cite{GleasonPhD}, for a point $x\in X_{\mu}(b)(k)$ in the affine Deligne-Lustzig variety (\cite[Def. 3.3.1]{RapoPappasShtuka}), one may consider the formal completion
$$
\widehat{\mathcal{M}^{\text{int}}_{\text{Scholze}}}_{/_{x}},
$$
(c.f. \cite{RapoPappasShtuka} (3.3.4.)). Then one shows the following
\begin{Theorem}{(Proposition \ref{Formale Modelle fuer die Tubes})}
\\
Assume that $p\geq 3.$
The $v$-sheaf $
\widehat{\mathcal{M}^{\text{int}}_{\text{Scholze}}}_{/_{x}}
$ is representable by the formal spectrum over $\Spf(\mathcal{O}_{\breve{E}})$ of a regular complete noetherian local ring.
\end{Theorem}
\subsection{Techniques of the proof}
\subsubsection{Link between adjoint nilpotent displays and local shtuka}
To prove the previous statements, one has to relate adjoint nilpotent $\G$-$\mu$-displays to Scholze's local shtuka. One actually compares adjoint nilpotent $\G$-$\mu$-displays to Breuil-Kisin-Fargues modules with extra structure (see Def. \ref{Definition BKF}) and then the passage to shtuka is evident.
Here the main statement is the following
\begin{proposition}(Corollary \ref{Korollar Vergleich Displays und BKF})
\\
Assume $p\geq 3.$ Let $R$ be an integral perfectoid $\mathcal{O}_{\breve{E}}$-algebra.
Then there is a fully faithful functor from the category of adjoint nilpotent $\G$-$\mu^{\sigma}$-displays to the category of $\G$-Breuil Kisin-Fargues modules of type $\mu.$
\end{proposition}
In the case where $R=\mathcal{O}_{C}$ is the ring of integers of a complete and algebraically closed extension of $\breve{E},$ one can describe the essential image in the style of Scholze-Weinstein as modifications of trivial $G$-bundles on the Fargues-Fontaine curve.
\begin{proposition}{(Corollary \ref{Scholze-Weinstein G-displays})}
\\
Assume $p\geq 3$ and let $\mathcal{O}_{C}$ be as before. Consider the pointed adic Fargues-Fontaine curve $(X^{\text{ad}}_{C^{\flat},\mathbb{Q}_{p}},\infty)$ associated to the pair $(C^{\flat},\mathbb{Q}_{p})$ and the given untilt $C$ of $C^{\flat}.$
\\
The following categories are equivalent:
\begin{enumerate}
\item[(a):] Adjoint nilpotent $\G$-$\mu^{\sigma}$-displays over $\mathcal{O}_{C},$
\item[(b):] tuples $(\mathcal{E}_{1},\mathcal{E},\iota,\mathcal{T}),$ where $\mathcal{E}_{1}$ and $\mathcal{E}$ are $G^{\text{adic}}$-torsors on $X^{\text{ad}}_{C^{\flat},\mathbb{Q}_{p}},$ $\mathcal{E}_{1}$ is trivial, $\iota$ is a $\mu$-bounded meromorphic modification of $\mathcal{E}_{1}$ by $\mathcal{E}$ along $\infty,$
$\mathcal{T}$ is a $\G$-torsor on $\Spec(\mathbb{Z}_{p}),$ such that if $\mathcal{V}$ is the $G$-torsor on $\Spec(\mathbb{Q}_{p})$ corresponding to $\mathcal{E}_{1}$ 
one requires that
$$
\mathcal{T}\times_{\Spec(\mathbb{Z}_{p})}\Spec(\mathbb{Q}_{p})\cong \mathcal{V}.
$$
One furthermore requires that the vector bundle $\Ad(\mathcal{E})$ has all HN-slopes $<1;$ here $\Ad(\mathcal{E})$ is the vector bundle on $X^{\text{ad}}_{C^{\flat},\mathbb{Q}_{p}}$ obtained by pushing out along the adjoint representation.
\end{enumerate}
\end{proposition}
I will now explain what goes into the proof of these statements. Here I follow the strategy used by Eike Lau in \cite{LauPerfektoid} where he considers the $\GL_{n}$ case.
\\
Let $R$ be an integral perfectoid $\mathcal{O}_{\breve{E}}$-algebra. The bridge between the rings $\Ainf(R)$ and $W(R)$ is built using the crystalline period ring $\Acris(R):$ the ring $\Ainf(R)$ embeds into $\Acris(R)$ and using the Cartier morphism and Witt functoriality one constructs the rather funny morphism
$$
\chi\colon \Acris(R)\rightarrow W(\Acris(R))\rightarrow W(R).
$$
The whole set-up admits an enhencement to frame morphisms and one can introduce the appropiate notion of $\G$-$\mu$-window categories. Objects in this category are torsors under certain groups and in case these torsors are trivial one calls these windows banal (a term that was introduced in \cite{BueltelPappas} for displays). Then one shows the following (the 'crystalline equivalence'):
\begin{proposition}{(Proposition \ref{Kristalline Aequivalenz})}
\\
Let $R$ be an integral perfectoid $\mathcal{O}_{\breve{E}}$-algebra.
\begin{enumerate}
\item[(a):] The morphism $\chi$ induces an equivalence
$$
\chi_{\bullet}\colon \G\text{-}\mu\text{-}\text{Win}(\mathcal{F}_{\cris}(R))_{\nilp, \banal}\rightarrow \G\text{-}\mu\text{-}\text{Displ}(R)_{\nilp, \banal}.
$$
\item[(b):] Assume furthermore that $R$ is $p$-torsion free, then
$$
\chi_{\bullet}\colon  \G\text{-}\mu\text{-}\text{Win}(\mathcal{F}_{\cris}(R))_{\nilp}\rightarrow  \G\text{-}\mu\text{-}\text{Displ}(R)_{\nilp}
$$
is an equivalence.
\end{enumerate}
\end{proposition}
\begin{Remark}
The restriction in (b) to $p$-torsion free integral perfectoid rings is maybe not necessary, but I was only able to verify descent for $\Acris(.)$ in this set-up c.f. Lemma \ref{Descent fuer Acris}.
\end{Remark}
Here the key ingredient is Lau's unique lifting lemma, which I reproduce in the language used in this article in Proposition \ref{Unique lifting lemma}.
\\
It remains to link $\G$-$\mu$-windows over $\Ainf(R)$ to those over $\Acris(R).$ This is the step where I have to assume $p\geq 3$ unfortunately. Here I give a group theoretic adaptation of the arguments of Cais-Lau \cite{CaisLau} and show the following
\begin{proposition}{(Proposition \ref{Deszent von Acris nach Ainf})}
\\
Assume that $p\geq 3$ and let $R$ be an integral perfectoid $\mathcal{O}_{\breve{E}}$-algebra. Then the inclusion $\Ainf(R)\hookrightarrow \Acris(R)$ induces an equivalence
$$ \G\text{-}\mu\text{-}\text{Win}(\mathcal{F}_{\infintesimal}(R))_{\banal}\rightarrow  \G\text{-}\mu\text{-}\text{Win}(\mathcal{F}_{\cris}(R))_{\banal}.$$
\end{proposition}
\subsubsection{Link between the moduli problems}
Having wrestled with the theory of frames and windows, one can finally turn to the comparison of the moduli problems. To get going, one has to compare the two notions of quasi-isogeny that are used. Unfortunately, this is not a direct consequence of the work done before: in the case of $\GL_{n}$ this translation of the quasi-isogeny is automatic because there one does not only work with groupoids and one can use a good notion of an isogeny. This is one of the most technical parts of this article and I refer to section \ref{Section Translation of the Quasi-isogeny} for the full details; let me just say here that one uses the Tannakian perspective on $\G$-$\mu$-displays introduced by Daniels in \cite{Daniels1} and the key fact that the Frobenius acts $p$-adically very nilpotently on
$$
\ker(\Acris(R/p)\rightarrow W(R/p)).
$$
At this point, one can construct a morphism of $v$-sheaves over $\Spd(\mathcal{O}_{\breve{E}})$
$$
\mathcal{M}^{\text{BP}}_{v}\rightarrow \mathcal{M}^{\text{int}}_{\text{Scholze}},
$$
see for example the proof of Proposition \ref{Formale Modelle fuer die Tubes} and one can verify that this morphism is a bijection on geometric points. Scholze has pioneered a trick which allows one to deduce an isomorphism of $v$-sheaves in this situation: one considers arbitrary products $\prod_{i\in I}\mathcal{O}_{C_{i}},$ where $C_{i}$ are complete algebraically closed fields over $\mathcal{O}_{\breve{E}}$ with perfectoid pseudo uniformizers $\varpi_{i}\in C_{i}.$ Let $R^{+}=\prod_{i\in I}\mathcal{O}_{C_{i}},$ $\varpi=(\varpi_{i})_{i}\in R^{+}$ and $R=R^{+}[1/\varpi].$ Then one has to show that a $\Spa(R,R^{+})$-valued point of $\mathcal{M}^{\text{Scholze}}_{\text{int}}$ factors over $\mathcal{M}^{\text{BP}}_{v}\rightarrow \mathcal{M}^{\text{Scholze}}_{\text{int}}.$ Since this product of point construction does not interact well with the adjoint nilpotency condition, one has to find a different way.
\\
Here I succeeded in the generic fiber to show Proposition \ref{Vergleich der Diamanten}. One technical ingredient is the following statement, which might also be helpful in other contexts:
\begin{Lemma}{(see Lemma \ref{Lemma Bouthier Cesnavicius} and Remark \ref{Input Bhatt zum Ausdehnen})}
\\
Let $(R,R^{+})$ be an affinoid perfectoid pair over $k.$ Let $\mathcal{N}$ be a $\G$-torsor over $$\Spec(W(R^{+}))-V(p,[\varpi]).$$
There exists a covering for the analytic topology of $\Spa(R,R^{+}),$ such that $\mathcal{N}$ extends to the whole spectrum after pullback to this covering.
\end{Lemma}
\subsection{Some loose ends}
Finally, let me briefly explain why I got interested in the problem of showing that the generic fiber of the Bültel-Pappas moduli problem of adjoint nilpotent $\G$-$\mu$-displays is representable by the local Shimura variety: the methods of this article allow to construct a continous and specializing morphism
$$
\text{sp}\colon |\Sh_{K}|\rightarrow |X_{\mu}(b)|,
$$
here $X_{\mu}(b)$  the affine Deligne-Lustzig variety associated to an unramified local Shimura-datum $$(G,\lbrace \mu \rbrace, [b])$$ over $\mathbb{Q}_{p},$ with chosen reductive model $\G$ of $G$ and maximal hyperspecial levelstructure $K=\G(\mathbb{Z}_{p});$ this morphism was also constructed in great generality and studied in depth in the PhD-thesis of Ian Gleason \cite{GleasonPhD}. The idea was to show that the following locally ringed space is in fact a formal scheme locally formally of finite type: $(|X_{\mu}(b)|,\text{sp}_{*}\mathcal{O}^{\circ}_{\Sh_{K}}).$ To glue the underlying reduced scheme one would for example have to show that there is an isomorphism
$$
((\text{sp}_{*}\mathcal{O}^{\circ}_{\Sh_{K}})/(\text{sp}_{*}\mathcal{O}^{\circ \circ}_{\Sh_{K}}))^{\text{perf}}\simeq \mathcal{O}_{X_{\mu}(b)}.
$$
To attack this, it would be helpful to be able to construct an adjoint nilpotent $\G$-$\mu$-display over the ring of power bounded functions over $\text{sp}^{-1}(U),$ where $U$ is open, affine in $X_{\mu}(b)$ To actually get going, a moduli description in terms of $\G$-$\mu$-displays of the local Shimura variety might be helpful. Let me nevertheless admit, that the most difficult thing seems to be to verify that the ring $\mathcal{O}^{\circ}_{\Sh_{K}}(\text{sp}^{-1}(U))$ is $I=\mathcal{O}^{\circ \circ}_{\Sh_{K}}(\text{sp}^{-1}(U))$-adic and formally of finite type; this seems difficult because the specialization morphism is not quasi-compact and I don't know how to control the geometric shape of $\text{sp}^{-1}(U)$ (it should at least be a smooth Stein-space in the sense of rigid analytic geometry). Needless to say, once one knows that the locally ringed space  $(|X_{\mu}(b)|,\text{sp}_{*}\mathcal{O}^{\circ}_{\Sh_{K}})$ is a formal scheme, formally smooth and locally of finite type, one could finally attack the representability conjecture put forward by Bültel-Pappas, see Conjecture \ref{Darstellbarkeitsvermutung BP}.
\subsection{Overview}
Let me briefly explain what is actually done in the sections.
In section \ref{section Frames}, I recall the theory of frames and introduce the main examples that will be relevant later on. Here I also verify descent for the relevant frames and will introduce an h-frame structure in the sense of Lau on the crystalline frame associated to a quasi-regular semiperfect ring. Afterwards, as was already said, in section \ref{Section Gmu windows}, I will introduce the notion of $\G$-$\mu$ windows for the frames that are necessary. Furthermore, I will recall the adjoint nilpotency condition discovered by Bültel-Pappas in \cite{BueltelPappas} and reproduce Lau's unique lifting lemma in the language I am using here, see Prop \ref{Unique lifting lemma}: this will be used repeatedly in the sequel. In section \ref{Section crystalline equivalence}, I will show the crystalline equivalence and in section \ref{Section Descent from Acris to Ainf}, I will explain why $\G$-$\mu$-windows for the frame corresponding to $\Ainf(R)$ are equivalent to $\G$-$\mu$-windows for the frame corresponding to $\Acris(R),$ as long as $p\geq 3.$ This will be put to use in section \ref{Section Connection to local mixed-characteristic shtukas}, when I compare $\G$-$\mu$-displays over integral perfectoid rings with Scholze's mixed characteristic shtukas. The next section \ref{Section Translation of the Quasi-isogeny}, is concerned with the translation between the notion of quasi-isogeny as used by Bültel-Pappas and with that used by Scholze in his moduli problem of local mixed characteristic shtukas. Finally, in the last section \ref{Applications to the Bueltel Pappas}, everything will be put together to first show the representability of diamond of the generic fiber of the Bültel-Pappas moduli problem, then I explain a conjecture on absolute prismatic crystals which should pave the way towards showing that the generic fiber of the Bültel-Pappas moduli problem itself is representable by the local Shimura variety; at least inside the category of complete and smooth affinoid adic spaces over $\Spa(\breve{E}).$ Finally, the existence of integral models of the tubes in the local Shimura variety is verified.

\subsection{Acknowledgments}
It is my pleasure to thank many mathematicians that have helped me a lot throughout working on this project and without whose help this article would not exist. First of all, my advisor Laurent Fargues, who suggested to supervise me when I came to Paris with some vague ideas concerning this project in my head. I thank Eike Lau for a very helpful conversation concerning section \ref{Section Descent from Acris to Ainf}; he suggested to think in the direction of Lemma \ref{Aequivalenz der Fasern unter Reduktion modulo p}. Furthermore, I thank Matthew Morrow for very helpful conversations concerning Lemma \ref{Lemma Bouthier Cesnavicius}: he helped me to correct my previous proof attempt, told me to look into \cite{BouthierCesnavicius} and gave some remarks on a write-up. I thank Oliver Bültel for his interest, comments on an earlier draft and finding a mistake. George Pappas for his interest, comments on an earlier version, a remark that is used in proposition \ref{Formale Modelle fuer die Tubes} and for his work as a thesis referee. I also profited a lot from exchanges and stimulating discussions with Johannes Anschütz, Patrick Daniels, Haoyang Guo, Arthur-César Le Bras, Tobias Kreutz, João Lourenço, Peter Scholze and Benoît Stroh.
\\
During the work on this article the author received financial support from the ERC Advanced Grant 742608 GeoLocLang.
\section{Notations and Conventions}
I will fix throughout a prime $p$. All rings will be assumed to be commutative and to have a 1 element. If $R$ is some ring, I denote by $\Nilp_{R}$ the category of $R$-algebras $S$, such that $p$ is nilpotent in $S$. When talking about the ring of Witt vectors, I always mean the ring of $p$-typical Witt vectors.
A ring $R$ is called $p$-adic, when it is complete and separated in the $p$-adic topology. A surjection of $p$-adic rings $S\rightarrow R$ is called a pd-extension, if $\mathfrak{a}=\ker(S\rightarrow R)$ is equipped with a divided power structure $\lbrace \gamma_{n} \rbrace,$ that I require to be compatible with the canonical pd-structure on $p\mathbb{Z}_{p}$. I define
$
\gamma_{n}(x)=[x]^{(n)}.
$
Groups will always act from the right. If $\G=\Spec(R)$ is some smooth-affine group scheme over $\mathbb{Z}_{p},$ then $\G^{\text{adic}}=\Spa(R,R^{+}),$ where $R^{+}$ is the integral closure of $\mathbb{Z}_{p}$ in $R;$ this definition is made in a way such that if $S$ is some adic space over $\mathbb{Z}_{p},$ then $\G^{\text{adic}}(S)=\G(\Gamma(S,\mathcal{O}_{S})).$ 

\newpage
\section{Frames}\label{section Frames}
As the terminology already suggests, one first has to introduce the concept of frames, to prepare the study of $\G$-windows with $\mu$-structure. A frame is an axiomatization of the structures on the ring of Witt vectors, that are used in developing the theory of displays, as done by Zink \cite{ZinkDisplay}. This section might be a bit dry but the reason it is fruitful to think abstractly about the concept of a frame is that Zink himself (\cite{ZinkWindow}), Lau (\cite{LauPerfektoid}) and Anschütz-Le Bras (\cite{JohannesArthur}) found many situations, where they naturally make an appearance. For this article frame structures arising from constructions in $p$-adic Hodge-theory, that have been discussed by Lau in \cite{LauPerfektoid}, will be particularly relevant.
\subsection{The category of frames}
First of all, the notion of frames that is being used in this article will be quickly introduced and it will be explained what a morphism between them is. The reference for this is \cite[Section 2]{LauFrames}.
\begin{Definition}
A frame $\mathcal{F}$ consists of a $5$-tuple $(S,R,I,\varphi,\dot{\varphi}),$ where $S$ and $R$ are rings, $I\subset S$ is an ideal, such that $R=S/I$, $\varphi\colon S\rightarrow S$ is a ring-endomorphism and $\dot{\varphi}\colon I \rightarrow S$ is a $\varphi$-linear map. This data is required to fulfill the following properties:
\begin{enumerate}
\item[(a)] $\varphi(x)\equiv x^{p} \text{ mod }pS$,
\item[(b)] $\dot{\varphi}(I)$ generates $S$ as an $S$-module,
\item[(c)] $pS + I\subseteq \Rad(S)$.
\end{enumerate}
\end{Definition}
\begin{Remark}
From (b) it follows that there exists a unique element $\zeta_{\mathcal{F}}\in S,$ such that
$$\zeta_{\mathcal{F}}\dot{\varphi}(i)=\varphi(i),$$
for all $i\in I.$ I will call this element the \textit{frame-constant of}$\mathcal{F}.$ \footnote{This terminology is not found in the literature, but it sounds reasonable to me.}
\end{Remark}
Next, let me quickly introduce morphisms of frames.
\begin{Definition}
Let $\mathcal{F}$ and $\mathcal{F}^{\prime}$ be frames.
\\
A $u$-frame morphism is the data of a pair $\lambda$ and $u,$ of a ring-homomorphism
$$
\lambda\colon \mathcal{F}\rightarrow \mathcal{F}^{\prime}
$$
and a unit $u\in (S^{\prime})^{\times},$ such that
\begin{enumerate}
\item[(a):] $\lambda(I)\subseteq I^{\prime},$
\item[(b):] $\varphi^{\prime}\circ \lambda=\lambda\circ \varphi,$
\item[(c):] $(\dot{\varphi}^{\prime}\circ\lambda)(i)=u\cdot (\lambda\circ \dot{\varphi})(i),$ for all $i\in I.$
\end{enumerate}
\end{Definition}
\begin{Remark}
\begin{enumerate}
\item[(i):]
The unit $u\in (S^{\prime})^{\times}$ is uniquely determined, again thanks to Axiom (b) in the definition of a frame.
\item[(ii):] For the Frame-constants $\zeta_{\mathcal{F}}$ and $\zeta_{\mathcal{F}^{\prime}}$ one gets the relation
$$u \zeta_{\mathcal{F}^{\prime}}=\lambda(\zeta_{\mathcal{F}}).$$
\item[(iii):] In case $u=1,$ I will speak of a strict frame-morphism rather than a $1$-morphism.
\end{enumerate}
\end{Remark}
\subsection{Main examples of frames}
\subsubsection{The Witt frame}
Let $R$ be a $p$-adic ring. Let $W(R)$ the ring of Witt vectors and denote by $$I(R)=\ker(w_{0}\colon W(R)\rightarrow R)$$ the image of Verschiebung. Then Zink proved that $W(R)$ is $p$-adic and $I(R)$-adic (and both topologies coincide, if $R\in \Nilp_{\mathbb{Z}_{p}}$), see \cite[Prop. 3]{ZinkDisplay}.  It follows that
$$\mathcal{W}(R)=(W(R),I(R),R,F,V^{-1})$$
is a frame with frame constant equal to $p$. It is functorial in homomorphisms of $p$-adic rings.

\subsubsection{The frame $\mathcal{F}_{\infintesimal}$}
Let me recall the definition of integral perfectoid rings, as in \cite{BMS1}.
\begin{Definition}
An integral perfectoid ring $R$ is a topological ring $R,$ such that
\begin{enumerate}
\item[(a):] $R/p$ is semiperfect,
\item[(b):] $R$ is $p$-adic,
\item[(c):] $\ker(\theta_{R}\colon W(R^{\flat})\rightarrow R)$ is a principal ideal,
\item[(d):] there exists an element $\varpi\in R,$ such that
$$
\varpi^{p}=pu,
$$
where $u\in R^{\times}.$
\end{enumerate}
Here $R^{\flat}=\lim_{\text{Frob}}R/p$ is the inverse limit perfection and $\theta_{R}$ is Fontaine$^{\prime}$s map, which exists for any $p$-adic ring.
\end{Definition}
Traditionally, one writes $\Ainf(R)=W(R^{\flat})$ and checks that any generator $\xi\in\Ainf(R)$ of $\ker(\theta)$ is automatically a non-zero divisor; in fact generators are exactly the so-called distinguished elements: if one writes $\xi=(\xi_{0},\xi_{1},...),$ then $\xi_{0}$ has to be topologically nilpotent and $\xi_{1}$ a unit. Furthermore, $\Ainf(R)$ is $(p,\xi)$-adically complete. I denote by $\varphi$ the Witt vector Frobenius on $\Ainf(R).$
\\
Fix a generator $\xi$ of $\ker(\theta).$ Consider the $\varphi$-linear map
$$
\dot{\varphi}\colon \ker(\theta)=(\xi)\rightarrow \Ainf(R)
$$
given by $\dot{\varphi}(\xi x)=\varphi(x).$ Note that it is well-defined, because $\xi$ is a non-zero divisor, but that it \textit{depends on the choice of}$\xi$.
It follows that
$$
\mathcal{F}_{\infintesimal}(R)=(\Ainf(R),R,\ker(\theta)=(\xi),\varphi,\dot{\varphi})
$$
is a frame with frame-constant $\varphi(\xi),$ which depends on the choice of $\xi.$ It is functorial for homomorphisms of integral perfectoid rings.

\subsubsection{The frame $\mathcal{F}_{\text{cris}}$}
Consider again an integral perfectoid ring $R.$ Recall the ring $\Acris(R),$ which is the universal $p$-complete pd-thickening of the semi-perfect ring $R/p.$ One can construct it as the $p$-adic completion of the pd-hull (over $(\mathbb{Z}_{p},p\mathbb{Z}_{p})$ by the conventions) of $\Ainf(R)$ with respect to $\ker(\theta);$ by continuity one still has the surjection
$$
\theta\colon \Acris(R)\rightarrow R
$$
(whose kernel is now way bigger).
In the following, one considers the ideal $$\Fil(\Acris(R))=\ker(\theta\colon \Acris(R)\rightarrow R),$$ that turns out to be a pd-ideal. A crucial observation made by Lau is that $\Acris(R)$ is $p$-torsion free, see \cite[Prop. 8.11.]{LauPerfektoid} -  Bhatt-Morrow-Scholze give in \cite[Thm. 8.14 (1)]{BMS2} a different proof of a more general fact (i.e. that $\Acris(S)$ is $p$-torsion free for any quasi-regular semiperfect ring $S$). Since on $\Fil(\Acris(R))$ the Frobenius $\varphi$ becomes divisible by $p,$ it follows that 
$$\dot{\varphi}=\frac{\varphi}{p}\colon \Fil(\Acris(R)) \rightarrow \Acris(R)$$ is a well-defined $\varphi$-linear map.
\\
Then
$$\mathcal{F}_{\cris}(R):=(\Acris(R), \Fil(\Acris(R)),R,\varphi,\dot{\varphi})$$
is a frame with frame constant equal to $p$. It is functorial in homomorphisms of integral perfectoid rings.
\begin{Remark}
A $\mathbb{F}_{p}$-algebra is integral perfectoid if and only if it is perfect. Then, for integral perfectoid $\mathbb{F}_{p}$-algebras $R,$ it follows that that $$\mathcal{F}_{\cris}(R)=\mathcal{F}_{\infintesimal}(R)=\mathcal{W}(R).$$
\end{Remark}
\subsubsection{The frame $\mathcal{F}_{\Prism}$}
The frame-structure discussed here was first considered by Anschütz-Le Bras in \cite[Example 4.1.17]{JohannesArthur}, when they classify $p$-divisible groups over quasi-syntomic rings using absolute filtered prismatic $F$-crystals \cite[Thm. 4.6.10]{JohannesArthur}.
\\
To introduce these frames, let me first introduce the notion of a quasi-regular semi-perfectoid ring; this concept was introduced by Bhatt-Morrow-Scholze in \cite[Def. 4.19 and Rem. 4.21]{BMS2}.
\begin{Definition}
\begin{enumerate}
\item[(a):] A ring $S$ is called quasi-syntomic, if it has bounded $p^{\infty}$-torsion, is $p$-adic, and the contangent complex $L_{S/\mathbb{Z}_{p}}$ has $p$-complete Tor-amplitude in $[-1,0],$ i.e. for all $S/p$-modules $N,$ it is true that
$$
(L_{S/\mathbb{Z}_{p}}\otimes_{S}^{L}S/p)\otimes_{S}^{L}N\in D^{[-1,0]}(R/p).
$$
A morphism of rings with bounded $p^{\infty}$-torsion $S\rightarrow S^{\prime}$ is called quasi-syntomic, if it is $p$-completely flat and the contangent complex $L_{S^{\prime}/S}\in D(S^{\prime})$ has $p$-complete Tor-amplitude in $[-1,0].$ It is called a quasi-syntomic cover if in addition $S\rightarrow S^{\prime}$ is $p$-completely faithfully flat. 
\item[(b):]  A quasi-syntomic ring $S$ is called quasi-regular semi-perfectoid if it is quasi-syntomic and admits a surjection $R\rightarrow S,$ where $R$ is integral perfectoid.
\item[(c):] Let $S$ be a quasi-syntomic ring, then the big quasi-syntomic site is the opposite category of the category that has objects given by $S$-algebras and morphisms are just $S$-algebra homomorphisms and the coverings are generated by quasi-syntomic covers. This site is denoted by $S_{\text{QSYN}}.$ The small quasi-syntomic site is the opposite of the category of quasi-syntomic $S$-algebras with covers generated by quasi-syntomic covers. This site is denoted by $S_{\text{qsyn}}.$
\end{enumerate}
\end{Definition}
\begin{Remark}
One has to argue that the category defined in $(c)$ forms indeed a site, this is done in \cite[Lemma 4.16.]{BMS2}. Recall furthermore the important fact, that locally for the quasi-syntomic topology, a quasi-syntomic ring will be quasi-regular semiperfectoid, see \cite[Lemma 4.27.]{BMS2}; this will be put to use later.
\end{Remark}
Let $S$ be quasi-regular semi-perfectoid, admitting a surjection $A/I\rightarrow S,$ where $(A,I)$ is a perfect prism. Then $I$ is a principal ideal, generated by a regular element; let $d$ be a generator. It is known that the absolute prismatic site $(S)_{\Prism}$ admits an initial object $(\Prism_{S},I\Prism_{S}),$\footnote{As the notation suggests the ring $\Prism_{S}=\Prism_{S/R}$ is in fact independent of $R.$} which is a prism over $(A,I).$ In particular, it is classically $(p,I)$-adically complete and $d$-torsion free. Define the first step of the Nygaard-filtration to be
$$
\mathcal{N}^{\geq 1}(\Prism_{S})=\lbrace x\in \Prism_{S}\colon \varphi(x)\in I \rbrace.
$$
Then one has that $\Prism_{S}/\mathcal{N}^{\geq 1}(\Prism_{S})\simeq S$ by \cite{bhatt2022prisms}, Thm. 12.2. As $\Prism_{S}$ is $d$-torsion free, the divided Frobenius
$$
\dot{\varphi}\colon \mathcal{N}^{\geq 1}(\Prism_{S})\rightarrow \Prism_{S},
$$ 
given by $\dot{\varphi}(i)=\varphi(i)/d$ is well-defined. In total, one can give the following 
\begin{Definition}\label{Prismatischer Frame}
Let $S$ be a quasi-regular semiperfectoid ring, which is a quotient of $A/I,$ where $(A,I)$ is a perfect prism. Then the prismatic frame associated to $S$ is
$$
\mathcal{F}_{\Prism}(S)=(\Prism_{S},\mathcal{N}^{\geq 1}(\Prism_{S}),S,\varphi,\dot{\varphi})
$$
\end{Definition}
\begin{Remark}
This is a frame with frame-constant (contrary to the $\mathcal{F}_{\infintesimal}$-frame) given by $d.$
\end{Remark}
\begin{Remark}
In this remark I will discuss the relation between the frame $\mathcal{F}_{\Prism}$ and the previously introduced frames.
\begin{enumerate}
\item[(a):] If $S=R$ is integral perfectoid, then $\Prism_{S}=\Ainf(R).$ Then $\mathcal{F}_{\Prism}(R)$ differs from the frame $\mathcal{F}_{\infintesimal}(R)$ as follows:
$$
\mathcal{F}_{\Prism}(R)=(\Ainf(R),\varphi^{-1}((\xi)),R, \varphi,\dot{\varphi}),
$$
where $\dot{\varphi}:=\varphi/\xi;$ this frame has frame constant $\xi,$ and
$$
\mathcal{F}_{\infintesimal}(R)=(\Ainf(R),(\xi),R,\varphi,\dot{\varphi}),
$$ 
where $\dot{\varphi}(\xi\cdot x)=\varphi(x).$ It has frame constant $\varphi(\xi).$ Compare with the discussion in \cite[Example 4.1.17]{JohannesArthur}.
\item[(b):] Let $S$ be quasi-regular semi-perfect, i.e. $S$ is a $\mathbb{F}_{p}$-algebra and the quotient of an integral perfectoid ring $R.$ Then one has an isomorphism $$\Prism_{S}\simeq \Acris(S),$$ which on $\Ainf(R)$ induces the Frobenius (see \cite[ Lemma 3.4.3.]{JohannesArthur}) - this identification will also be later used when verifying descent for $\Acris(.)$). In this case also $\mathcal{F}_{\Prism}(S)$ is a Frobenius-twist of $\mathcal{F}_{\cris}(S).$
\end{enumerate}
\end{Remark}
\subsubsection{Some important frame-morphisms}
\begin{enumerate}\label{Die zwei Framemorphismsen}
\item[(i):] Let $R$ be an integral perfectoid ring.
Using the Cartier-morphism\footnote{Recall that for a $p$-torsionfree ring $A,$ which carries a Frobnius lift $\varphi\colon A\rightarrow A,$ the Cartier-morphism is the homomorphism $\delta\colon A\rightarrow W(A)$ determined by $w_{n}(\delta(a))=\varphi^{n}(a),$ for all $n\geq 0$ and $a\in A.$ Here $w_{n}(.)$ are the Witt polynomials. one has $\delta(\varphi(a))=F(\delta(a))$ for all $a\in A.$ This is for example explained in \cite[Lemma 2.38]{ZinkVorlesung}.} $\delta\colon \Acris(R)\rightarrow W(\Acris(R)),$ one can construct a strict frame-morphism
$$
\chi\colon \mathcal{F}_{\cris}(R)\rightarrow \mathcal{W}(R),
$$
see Lemma \ref{Framemorphismus von Acris nach W(R)}.
\item[(ii):] Let $R$ be an integral perfectoid ring. Then the natural inclusion $\Ainf(R)\rightarrow \Acris(R)$ induces a $u=\frac{\varphi(\xi)}{p}$-frame morphism (for a choice of a generator $\xi$)
$$\lambda\colon \mathcal{F}_{\infintesimal}(R)\rightarrow \mathcal{F}_{\cris}(R).\footnote{One has to check that $u$ is indeed a unit. For completeness, I recall the argument: write $\xi=(a_{0},a_{1},...)$ in Teichmueller-coordinates. Then $u\equiv \frac{[a_{0}]}{p}+[a_{1}]^{p}$ modulo $p\Acris(R),$ thus $u\equiv [a_{1}]^{p}$ modulo $\text{Fil}(\Acris(R))+p\Acris(R)$. But $\text{Fil}(\Acris(R))+p\Acris(R)$ is in the radical of $\Acris(R).$ As $a_{1}$ is a unit.}$$
\end{enumerate}
In the following it will be shown that $\chi$ induces an equivalence between $\G$-$\mu$-windows over the respective frames, once one adds a nilpotency condition discovered by Bültel-Pappas (Prop. \ref{Kristalline Aequivalenz}) and that $\lambda$ induces an equivalence of $\G$-$\mu$-windows if one assumes $p\geq 3$ (Prop. \ref{Deszent von Acris nach Ainf}).
\subsection{Descent for frames}
As all the previously introduced examples of frames satisfied some functoriality, one can investigate the question, whether there are topologies for which one will get sheaf-properties. To prevent confusion, let me fix the following vocabulary: One says that a contravariant functor from some category equipped with a site-structure to the category of frames is a sheaf, if each entry in the datum of a frame gives rise to a sheafy version of the respective structure.
\subsubsection{Descent for the Witt frame:}
Let $R$ be a ring, in which $p$ is nilpotent. Then it follows for example from Zink$^{\prime}$s Witt descent (\cite{ZinkDisplay}, Lemma 30), that the functor
$$ (R\rightarrow R^{\prime})\mapsto W(R^{\prime})$$
is a sheaf for the fpqc-topology on the category of affine $R$-schemes.
From this one deduces the following
\begin{Lemma}
Let $R\in \Nilp.$
Then the functor
$$\mathcal{W}()\colon \Spec(R)_{\fpqc}^{\aff, \op} \rightarrow \Set,$$
$$(R\rightarrow R^{\prime}) \mapsto \mathcal{W}(R^{\prime})$$
is a sheaf of frames.
\end{Lemma}

\subsubsection{Descent for the frame $\mathcal{F}_{\infintesimal}$:}
Let $R$ be an integral perfectoid ring. Then I want to explain how to get a sheafy version of the frame $\mathcal{F}_{\infintesimal}(R)$ for the affine étale topology on $\Spec(R/p)$ (i.e. the formal affine $p$-adic étale topology of $\Spf(R)$). This construction rests on observations due to Lau made in \cite{LauPerfektoid}.
\\
In the following fix the integral perfectoid ring $R$ from above and also fix a generator $\xi\in \Ainf(R)$ of the kernel of $\theta\colon \Ainf(R)\rightarrow R.$ One first observes that then for any integral perfectoid $R$-algebra $R^{\prime},$ one has that
$$ \ker(\theta^{\prime}\colon \Ainf(R^{\prime})\rightarrow R^{\prime})=\xi\Ainf(R^{\prime}),$$
as the property of being a distinguished element passes through ring-homomorphisms.
Then Lau proves the following
\begin{Lemma}\label{Eindeutige integral perfektoide algebra}\cite[Lemma 8.10.]{LauPerfektoid}
Let $R$ be an integral perfectoid ring. Let $B:=R/p$ and $B\rightarrow B^{\prime}$ be an étale ring-homomorphism. Then there exists a unique integral perfectoid $R$-algebra $R^{\prime},$ such that (i) $R^{\prime}/p=B^{\prime}$ and (ii) all $R/p^{n}\rightarrow R^{\prime}/p^{n}$ are étale, for $n\geq 1$.
\end{Lemma}
Here the integral perfectoid $R$-algebra is constructed as $R^{\prime}:=W((B^{\prime})^{\text{perf}})/\xi,$ where $$(B^{\prime})^{\flat}:=(B^{\prime})^{\text{perf}}=\lim_{\text{Frob}}B^{\prime}$$ is the inverse-limit perfection. Using this, one can construct the following pre-sheaves:
\\
Let $\Spec(R/p)^{\text{aff}}_{\text{ét}}$ be the category of affine schemes, that are étale over $\Spec(R/p)$. This category carries the structure of a site by declaring jointly surjective morphisms as covers. Then consider
$$\mathcal{R}\colon \Spec(R/p)^{\text{aff}, \op}_{\text{ét}}\rightarrow \Set,$$
$$\mathcal{R}(B^{\prime})=R^{\prime},$$
where $R^{\prime}$ is the uniquely determined integral perfectoid $R$-algebra from Lemma \ref{Eindeutige integral perfektoide algebra} above. Furthermore, consider 
$$\mathcal{A}_{inf}\colon \Spec(R/p)^{\text{aff}, \op}_{\text{ét}}\rightarrow \Set,$$
$$\mathcal{A}_{inf}(B^{\prime})=\Ainf(\mathcal{R}(B^{\prime})).$$ This gives two ring pre-sheaves. Lau proves then the following:
\begin{Lemma}\cite[Lemma 10.9]{LauPerfektoid}\label{Lau perfektoide Garbe}
 The pre-sheaves $\mathcal{R}$ and $\mathcal{A}_{inf}$ are sheaves.
\end{Lemma}
\begin{Remark}
This can also be explained from the prismatic perspective, using that the prismatic structure (pre)-sheaf is in fact a sheaf for the $p$-completely étale topology by \cite[Cor. 3.12.]{bhatt2022prisms} and then restricting to the perfect prismatic site.
\end{Remark}
From this one immediately deduces that 
$$\mathcal{F}_{\infintesimal}(\mathcal{R})\colon \Spec(R/p)^{\text{aff}, \op}_{\text{ét}}\rightarrow \Set$$ is a sheaf of frames.
\subsubsection{Descent for the frame $\mathcal{F}_{\text{cris}}$}
Consider still an integral perfectoid ring $R.$ Note that for an étale $R/p$-algebra $B^{\prime},$ one has that $B^{\prime}$ is still semi-perfect, as in fact $B^{\prime}=\mathcal{R}(B^{\prime})/p.$ In the next statement I will use prismatic techniques and therefore one really will want to know that $B$ is quasi-syntomic, for which I will have to suppose that $R$ is in fact $p$-torsion free. Recall that $\Acris(B^{\prime})=\Acris(\mathcal{R}(B^{\prime})).$ Then one observes the following
\begin{Lemma}\label{Descent fuer Acris}
Let $R$ be a $p$-torsion free integral perfectoid ring.
Then the presheaf
$$
\mathcal{A}_{\cris}(\mathcal{R})\colon \Spec(R/p)^{\aff, \op}_{\et}\rightarrow \Set
$$
is a sheaf.
\end{Lemma}
\begin{proof}
First, let me give a prismatic proof: write as before $B=R/p;$ first observe that $B$ is a quasi-regular semiperfect ring. By definition, one has to see that $B$ admits a surjection from a perfectoid ring (trivial) and that $B$ is itself quasi-syntomic. This follows from the assumption that $p$ was regular in $R.$ \footnote{Here quickly the argument: To see that $B$ is indeed quasi-syntomic, look at the morphisms $\mathbb{Z}_{p}\rightarrow R \rightarrow R/p,$ which gives the following triangle for cotangent complexes:
$$
\xymatrix{
L_{R/\mathbb{Z}_{p}}\otimes^{L}_{R}R/p \ar[r] & L_{R/p/\mathbb{Z}_{p}} \ar[r] & L_{(R/p)/R} \ar[r]^{+1} &.
}
$$
Since $R$ is perfectoid, the term $L_{R/\mathbb{Z}_{p}}\otimes^{L}_{R}R/p=(\ker(\theta))/(\ker(\theta))^{2}[1]\otimes_{R}^{\mathbb{L}}R/p$ has $p$-completed Tor-amplitude in $[-1,0]$ and since $p$ is regular in $R,$ it is true that $L_{(R/p)/R}\simeq (pR)/(pR)^{2}[1],$ so that it follows that also $L_{R/p/\mathbb{Z}_{p}}$ has $p$-completed Tor-amplitude in $[-1,0].$
}
Now consider the small quasi-syntomic site of $B,$ denoted by $B_{\text{qsyn}}.$ It has a basis for the topology given by objects $B\rightarrow B^{\prime},$ such that $B^{\prime}$ is quasi-regular semiperfect. Then let $B_{\text{qrsp}}$ be the subcategory of $B_{\text{qsyn}},$ with objects $B\rightarrow B^{\prime},$ where $B^{\prime}$ is quasi-regular semiperfect; it is still endowed with the quasi-syntomic topology and forms a site by \cite[Lemma 4.26.]{BMS2} Then the association 
$$
\Prism_{\bullet}\colon B_{\text{qrsp}}^{\op}\rightarrow \Set,
$$
given by $(B\rightarrow B^{\prime})\mapsto \Prism_{B^{\prime}}$ is a sheaf by \cite[Cor. 3.12]{bhatt2022prisms}. But recall  from \cite[Lemma 3.4.3]{JohannesArthur}, that one has a canonical, $\varphi$-equivariant identification
$$
\Prism_{B^{\prime}}\simeq \Acris(B^{\prime}).
$$
The point being here that on the one hand, $\Acris(B^{\prime})$ is a $\delta$-ring, mapping down to $B^{\prime},$ which gives a map in one direction and on the other hand that $\ker(\Prism_{B^{\prime}}\rightarrow B^{\prime})$ admits divided powers.
This now implies the statement one is after: if $B\rightarrow B^{\prime}$ is an étale cover, then $B^{\prime}$ is again quasi-regular semiperfect. Namely, the result of Lau that I recalled before, Lemma \ref{Eindeutige integral perfektoide algebra}, one may write $B^{\prime}=R^{\prime}/p,$ where $R^{\prime}$ is an integral perfectoid ring. It thus suffices to see that $B^{\prime}$ is quasi-syntomic, for which one has to convince oneself of the fact that $p$ is regular in $R^{\prime}.$ But $\mathbb{Z}_{p}\rightarrow R$ is flat and $R\rightarrow R^{\prime}$ is $p$-completely étale. By Elkik, there exists an étale $R$-algebra $R^{\prime}_{0},$ such that $R^{\prime}$ is the $p$-adic completion of $R^{\prime}_{0}.$ It follows that $p$ is regular in $R^{\prime}_{0}$ and therefore also in $R^{\prime}.$ Since a $p$-completely étale covering is in particular a quasi-syntomic covering, it follows that one may consider $B_{\text{ét}}$ as a sub-site of $B_{\text{qrsp}}$ and now the lemma follows.
\\
A different proof can be given by observing first, that it suffices to show that $(B\rightarrow B^{\prime})\mapsto \Acris(B^{\prime})/p$ is a sheaf, because all the values $\Acris(B^{\prime})$ are by construction $p$-adic and and an inverse limit of sheaves is still a sheaf. Then one may use the following heavy input: by \cite[Prop. 8.12]{BMS2}, there is a natural (for morphisms of quasi-regular semiperfect rings) isomorphism
$$
\Acris(B^{\prime})/p\simeq L\Omega_{B^{\prime}/\mathbb{F}_{p}};
$$
here one uses again that $B^{\prime}$ will be quasi-regular semiperfect to verify that the hypothesis of the statement of Bhatt-Morrow-Scholze are verified and $L\Omega_{./\mathbb{F}_{p}}$ denotes the derived de Rham complex of an $\mathbb{F}_{p}$-algebra. By comparing conjugate filtrations and using that $\mathbb{L}_{B^{\prime}/B}\simeq 0$ for any étale $B\rightarrow B^{\prime},$ it follows that
$$
L\Omega_{B^{\prime}/\mathbb{F}_{p}}\simeq L\Omega_{B/\mathbb{F}_{p}}\otimes_{B}B^{\prime}.
$$
By usual faithfully flat descent for quasi-coherent sheaves (as $L\Omega_{B/\mathbb{F}_{p}}$ is concentrated in degree $0$), it is then deduced that $\Acris(.)/p$ is a sheaf, as desired.
\end{proof}
\begin{Remark}
This lemma of course begs for a more down to earth proof and should probably be true for integral perfectoids with $p$-torsion. Here I fall however on the following problem: it is not clear to me that the PD-hull commutes with $I$-complete flat basechange. In fact, the formula one would want to verify is the following: let $B=R/p\rightarrow R^{\prime}/p$ be a $p$-completely faithfully flat étale map, then $B^{\flat}\rightarrow (B^{\prime})^{\flat}$ is $\xi_{0}$-completely faithfully flat étale and one would like to have that
$$
\Acris(B^{\prime})/p^{n}\simeq (\Acris(B)\otimes_{W(B^{\flat})}W(B^{\prime \flat}))/p^{n}(\Acris(B)\otimes_{W(B^{\flat})}W(B^{\prime \flat})),
$$
which would follow from the universal properties of $\Acris(B)$ resp. $\Acris(B^{\prime}),$ if one would know that PD-Hulls extend along $(p,\xi_{0})$-completely flat morphisms. This however is unfortunately not clear to me, so that I had to stick to the more abstract proof given above.
\end{Remark}
In total, one may deduce the following: let $R$ be a $p$-torsion free integral perfectoid ring. Then the pre-sheaf of frames
$$
\mathcal{F}_{\cris}(\mathcal{R}(.))\colon \Spec(R/p)^{\text{aff}, \op}_{\text{ét}}\rightarrow \text{Frames}
$$
is indeed a sheaf of frames. Indeed, this follows from the previous descent statement for $\Acris(.)$ and the known descent for $\mathcal{R}(.),$ by making use of the fact that a kernel of morphisms of sheaves is again a sheaf; to take care of the $\text{Fil}_{\cris}(.)$ component of the gadgets that make up a frame.
\subsubsection{Descent for the frame $\mathcal{F}_{\Prism}$}
Let $R$ be a quasi-syntomic ring and consider the quasi-syntomic site $R_{\text{qsyn}}$ of $R.$ Since $R$ is quasi-syntomic, this site admits as a basis quasi-regular semiperfectoid $R$-algebras $A.$ Denote by $R_{\text{qrsp}}$ the category of quasi-regular semiperfectoid $R$-algebras, equipped with the quasi-syntomic topology. Then one can consider the following pre-sheaf of frames
$$
\mathcal{F}_{\Prism}\colon R_{\text{qrsp}}^{\op}\rightarrow \text{Frames},
$$
given by sending a quasi-regular semiperfectoid $R$-algebra $A$ towards the frame $$\mathcal{F}_{\Prism}(A)=(\Prism_{A},\mathcal{N}^{\geq 1}(\Prism_{A}),A,\varphi,\dot{\varphi}),$$ as in Def. \ref{Prismatischer Frame}. By \cite{bhatt2022prisms}, Cor. 3.12. the association
$$
\Prism_{\bullet}\colon R_{\text{qrsp}}^{\op}\rightarrow \Set,
$$
sending a quasi-regular semiperfectoid $R$-algebra $A$ towards $\Prism_{A}$
defines a sheaf for the quasi-syntomic topology. Furthermore, the association $\mathcal{O}_{\text{qsyn}}\colon R_{\text{qrsp}}^{\op}\rightarrow \Set,$ given by sending a quasi-regular semiperfectoid $R$-algebra $A$ towards $A$ itself defines a quasi-syntomic sheaf. It follows that in total $\mathcal{F}_{\Prism}$ is a sheaf of frames for the quasi-syntomic topology.
\subsection{h-frame structure on prismatic frame for quasi-regular semiperfects}\label{hframe auf prismatischem Frame}
Here I will record an h-frame structure on the frame $\mathcal{F}_{\Prism}(A),$ where $A$ is as before a quasi-regular semiperfect ring.
\\
First, let me recall the definition of an h-frame, which is due to Lau:
\begin{Definition}{\cite[Def. 2.0.1.]{LauHigherFrames}}
\\
A triple $(\bigoplus_{n\geq 0}S_{n},\sigma, \lbrace t_{n} \rbrace_{n\geq 0}),$ where $\bigoplus_{n\geq 0}S_{n}$ is an $\mathbb{Z}_{n\geq 0}$-graded ring, $\sigma\colon S_{\geq 0}\rightarrow S_{0}$ is a ring homomorphism and $t_{n}\colon S_{n+1}\rightarrow S_{n}$ are maps of sets, is called an h-frame if
\begin{enumerate}
\item[(a):] For all $n\geq 0,$ the map $t_{n}\colon S_{n+1}\rightarrow S_{n}$ are $S_{\geq 0}$-module homomorphisms,
\item[(b):] $\sigma_{0}\colon S_{0}\rightarrow S_{0}$ is a Frobenius lift and $\sigma_{n}(t_{n}(a))=p\sigma_{n+1}(a),$ for all $a\in S_{n+1},$
\item[(c):] $p\in \Rad(S_{0}).$
\end{enumerate}
One can consult Lau's article \cite[Examples 2.1.1-2.1.14]{LauHigherFrames} for several examples from nature; for example he explains a natural h-frame structure on the Witt frame \cite[Example 2.1.3]{LauHigherFrames}.
\\
Now fix a quasi-regular semiperfect ring $A,$ let $(\Prism_{A},(p))$ be the associated prism. Recall that the Nygaard-filtration was defined as follows (\cite[Def. 12.1]{bhatt2022prisms}):
$$
\mathcal{N}^{\geq n}(\Prism_{A})=\lbrace x\in \Prism_{A}\colon \varphi(x)\in p^{n}\Prism_{A} \rbrace.
$$
\end{Definition}
This is a decreasing and multiplicative filtration indexed by the natural numbers. 
After this preparation, one can construct the desired h-frame structure: let
$$
\bigoplus_{n\geq 0}S_{n}:=\bigoplus_{n\geq 0}\mathcal{N}^{\geq n}(\Prism_{A}),
$$
which is a $\mathbb{Z}_{\geq 0}$-graded ring. Then let $\sigma_{n}\colon \mathcal{N}^{\geq n}\rightarrow \Prism_{A}$ be the $p^{n}$-divided Frobenius. Next, define
$$
t_{n}\colon \mathcal{N}^{\geq n+1}(\Prism_{A}) \rightarrow \mathcal{N}^{\geq n}(\Prism_{A})
$$
just to be the inclusion. Then the conditions (a)-(c) are satisfied. This is an h-frame associated to a $p$-torsion free ring with a multiplicative filtration with Frobenius-divisibility as in \cite[Example 2.1.1]{LauHigherFrames}. 
\begin{Remark}
Note that one is here not taking the filtration of $\Acris(A)$ by the divided powers of $\Fil(\Acris(A))$ (which would correspond to the Hodge-Filtration under the comparison with the derived de Rham cohomology). The worry with this filtration is that the Frobenius-divisibility fails. Note further that $\mathcal{N}^{\geq 1}(\Prism_{A})$ coincides with $\Fil(\Acris(A)),$ so that the above h-frame extends the 1-frame structure that was previously put on $\Acris(A).$
\end{Remark}
The reason I discussed this h-frame structure is that it will be later useful to have a Tannakian perspective on the notion of a crystalline $G$-quasi-isogeny, see section \ref{Section Translation of the Quasi-isogeny} and more precisely Lemma \ref{Tannakische Interpretation einer Quasi-Isogenie}.
\section{$\G$-$\mu$-windows}\label{Section Gmu windows}
The aim of this section is to quickly introduce $\G$-$\mu$-windows for the frames discussed in the previous section. To set up this theory one has two at least possibilities: either follow Lau's elegant approach in \cite{LauHigherFrames} using $h$-frame structures on the frames just discussed, or copy pase the definition given by Bültel-Pappas. In total, I sticked to the latter possibility since later I have to consider a frame for which I was not able to find a $h$-frame structure, c.f. Remark \ref{Erklaerung weshalb Problem mit h-frame Struktur}. To get going, in the first subsection the reader finds a recollection on the inputs from group theory that will be relevant. Then the category of $\G$-$\mu$-windows for the different frames is introduced and lastly I will recall the adjoint nilpotency condition as introduced by Bültel-Pappas, to be able to restate Lau's unique lifting lemma in the language used here.
\subsection{Reminder on some group-theory}
Let $\G\rightarrow \Spec(\mathbb{Z}_{p})$ be a smooth affine group scheme. Furthermore, let $k_{0}$ be a finite extension of $\mathbb{F}_{p}$ and
$$\mu\colon \mathbb{G}_{m,W(k_{0})}\rightarrow \G_{W(k_{0})}$$
be a cocharacter. Following Lau \cite{LauHigherFrames}, the data $(\G,\mu)$ will be refered to as a window datum. 
\\
To this set-up one can associate the following weight subgroups
\begin{enumerate}
\item[(i):] The closed subgroups $P^{\pm}\hookrightarrow \G_{W(k_{0})},$ that have points in $W(k_{0})$-algebras $R$ given as
$$ P^{\pm}(R)=\lbrace g\in \G_{W(k_{0})}(R)\colon \lim_{t^{\pm}\rightarrow 0} \mu(t)^{-1}g\mu(t)\text{ exists}\rbrace.$$ Here for example the expression $^{\prime} \lim_{t\rightarrow 0} \mu(t)^{-1}g\mu(t)\text{ exists} ^{\prime}$ means that the orbit map associated to $g:$
$$\mathbb{G}_{m,R}\rightarrow \G_{R},$$
$$ t\mapsto \mu(t)^{-1}g\mu(t) $$
extends (necessarily uniquely) to a morphism
$$\mathbb{A}_{R}^{1}\rightarrow \G_{R}.\footnote{Note the difference in  sign to \cite{CGP}, because I consistently consider right-actions.}$$
\item[(ii):] The closed normal subgroups $U^{\pm}\hookrightarrow P^{\pm},$ that have points in $W(k_{0})$-algebras $R$ given as
$$U^{\pm}(R)=\lbrace g\in \G_{W(k_{0})}(R)\colon \lim_{t^{\pm}\rightarrow 0} \mu(t)^{-1}g\mu(t)\text{ exists and equals the unit-section} \rbrace.$$
\end{enumerate}
Let me recall a couple of well-known facts about these groups, that will be used in the following.
\begin{Fact}\label{Zusammenfassung zu Gewichtsgruppen}
\item[(i):] The subgroups $P^{\pm}$ and $U^{\pm}$ are smooth over $W(k_{0}).$ Furthermore, if $\G$ has connected fibers, then $P^{\pm}$ and $U^{\pm}$ also have connected fibers and $U^{\pm}$ have unipotent fibers. In case $\G$ is a reductive group scheme, $P^{\pm}$ are parabolic subgroup schemes. See \cite[Lemma 2.1.4, Lemma 2.1.5]{CGP}.
\item[(ii):] Multiplication in $\G_{W(k_{0})}$ induces an open immersion
$$U^{+} \times_{W(k_{0})} P^{-} \rightarrow \G_{W(k_{0})}.$$
See \cite[Proposition 2.1.8. (3)]{CGP}.
\item[(iii):] Using the adjoint representation of $\G_{W(k_{0})},$ one gets a $\mathbb{G}_{m,W(k_{0})}$ representation
$$\Ad(\mu^{-1})\colon \mathbb{G}_{m,W(k_{0})} \rightarrow \GL(\Lie(\G)_{W(k_{0})}),$$
which in turn gives the weight decomposition
$$\Lie(\G)_{W(k_{0})}=\bigoplus_{n\in \mathbb{Z}} (\Lie(\G)_{W(k_{0})})_{n}.$$
Under this decomposition of the Lie-algebra, one has
$$
 \Lie(P^{+})= \bigoplus_{n\geq 0} (\Lie(\G)_{W(k_{0})})_{n} \text{ resp. } \Lie(P^{-})= \bigoplus_{n\leq 0} (\Lie(\G)_{W(k_{0})})_{n},
$$
and
$$
\Lie(U^{+})= \bigoplus_{n > 0} (\Lie(\G)_{W(k_{0})})_{n} \text{ resp. } \Lie(U^{-})= \bigoplus_{n < 0} (\Lie(\G)_{W(k_{0})})_{n}.
$$
See \cite[Proposition 2.1.8.(1)]{CGP}.
\item[(iv):] There exists $\mathbb{G}_{m,W(k_{0})}$-equivariant morphisms of schemes
$$ \log^{\pm}\colon U^{\pm}\rightarrow V(\Lie(U^{\pm})),$$
that induce the identity on Lie-algebras. They are necessarily isomorphisms of schemes. See \cite[Lemma 6.1.1]{LauHigherFrames}.
\\
Assume in addition that $\Lie(U^{+})=(\Lie(\G)_{W(k_{0})})_{1},$ then the morphism
$$ \log^{+}\colon U^{+}\rightarrow V(\Lie(U^{+}))$$
is uniquely characterized by the requirement to be $\mathbb{G}_{m}$-equivariant and to induce the identity on the Lie-algebras. Furthermore, it is a group scheme isomorphism. See \cite[Lemma 6.3.2]{LauHigherFrames}.
\item[(v):] Assume that $\G$ is a reductive group scheme, then the condition $\Lie(U^{+})=(\Lie(\G)_{W(k_{0})})_{1}$ is equivalent to $\mu$ being a minuscule cocharacter. This follows from the fact that the set of roots is preserved under $\eta\mapsto -\eta.$
\end{Fact}
Following Lau, one defines the following concept:
\begin{Definition} Let $(\G,\mu)$ be a window datum, then
\begin{enumerate}
\item[(i):] it is called 1-bound, if
$
\Lie(U^{+})=(\Lie(\G)_{W(k_{0})})_{1}
$,
\item[(ii):] it is called reductive, if $\G$ is a reductive group scheme over $\mathbb{Z}_{p}.$
\end{enumerate}
\end{Definition}
For the main applications towards moduli of $\G$-$\mu$-displays and their relations to local Shimura varieties, reductive and 1-bound display-data (which then automatically implies that $\mu$ is minuscule by the above) will be most relevant, since the concept of $\G$-$\mu$-windows one is using here is otherwise not the right one as was already pointed out in the introduction.
\subsection{Quotient groupoid of banal $\G$-$\mu$-windows}
The statements in this subsection are basically easy modifications of the results of Bültel-Pappas \cite[section 3.1.]{BueltelPappas}.
\\
Let $(\G,\mu)$ be a window datum, in particular one has fixed a finite extension $k_{0}$ of $\mathbb{F}_{p}.$
\begin{Definition}
Let $\mathcal{F}=(S,I,R,\varphi,\dot{\varphi})$ be a frame. One says that $\mathcal{F}$ is a frame over $W(k_{0}),$ if $S$ is a $W(k_{0})$-algebra and $\varphi$ extends the Witt vector Frobenius on $W(k_{0}).$
\end{Definition}
\begin{Example}
Let $R\in \Nilp_{W(k_{0})}$ (or a $p$-adic $W(k_{0})$-algebra). Using the Cartier-morphism
$$\bigtriangleup\colon W(k_{0})\rightarrow W(W(k_{0})),$$
which is characterized by the formula
$$
W_{n}(\bigtriangleup(x))=F^{n}(x),
$$
one sees that the Witt frame $\mathcal{W}(R)$ is a frame over $W(k_{0}).$ A similiar remark applies to the relative Witt frame.
\\
Furthermore, if $R$ is an integral perfectoid $W(k_{0})$-algebra, then both $\mathcal{F}_{\infintesimal}(R)$ and $\mathcal{F}_{\cris}(R)$ are frames over $W(k_{0}).$
\end{Example}
\begin{Definition}
Let $\mathcal{F}=(S,I,R,\varphi,\dot{\varphi})$ be a frame over $W(k_{0}).$ The window group of the window datum $(\G,\mu)$ is the group
$$\G(\mathcal{F})_{\mu}=\lbrace g\in \G_{W(k_{0})}(S)\colon g (\text{ mod }I)\in P^{-}(R) \rbrace.$$
\end{Definition}
Let me describe the structure of the window group.
\begin{Lemma}\label{Zerlegungslemma fuer die Window-gruppe}
Multiplication in $\G_{W(k_{0})}(S)$ induces a bijection
$$m\colon U^{+}(I) \times P^{-}(S) \rightarrow \G(\mathcal{F})_{\mu}.$$
\end{Lemma}
\begin{proof}
First, note that the above map is an injection. Indeed, by Fact \ref{Zusammenfassung zu Gewichtsgruppen} (ii) multiplication is an open immersion and thus a monomorphism.
\\
It remains to see the surjectivity. Take a point $g\in \G(\mathcal{F})_{\mu}$ corresponding to a morphism $$g\colon \Spec(S) \rightarrow \G_{W(k_{0})}(S),$$ that modulo $I$ factors through $P^{-}.$ By the axioms of a frame, $I\subseteq \Rad(S),$ and therefore there is a bijection of maximal prime ideals $\MaxSpec(R) \longleftrightarrow \MaxSpec(S).$ Now consider the following commutative diagram
$$
\xymatrix{
\Spec(S)\ar[r]^{g} & \G_{W(k_{0})} \\
\Spec(R)\ar[u]^{f^{\sharp}}\ar[r] & U^{+}\times_{W(k_{0})}P^{-},\ar[u]^{m}
}
$$
where the lower horizontal map is given by $(e,g\circ f^{\sharp}),$ where $e\in U^{+}(R)$ is the unit section. If $x_{0}\in \Spec(S)$ is a closed point, there is a unique closed point $x_{0}^{\prime}\in \Spec(R),$ such that $f^{\sharp}(x_{0}^{\prime})=x_{0}.$ It follows that $g(x_{0})=g(f^{\sharp}(x_{0}^{\prime}))\in V,$ where $V\subseteq \G_{W(k_{0})}$ is the open image of the mutliplication $m.$ Here one has used again Fact \ref{Zusammenfassung zu Gewichtsgruppen} (ii) above. Now let $x\in\Spec(S)$ be an arbitrary point. One finds a closed point $x_{0}$ lying in the closure of $x.$ It follows that $g(x_{0})\in g(\overline{\lbrace x \rbrace})\subseteq \overline{\lbrace g(x) \rbrace},$ as $g$ is continous. As already observed, $g(x_{0})\in V$ and $V$ is open, thus closed under generalization, so that $g(x)\in V.$ This concludes the verification of the surjectivity.
\end{proof}
The next aim is to construct the divided Frobenius
$$\Phi_{\G,\mu,\mathcal{F}}\colon \G(\mathcal{F})_{\mu}\rightarrow \G_{W(k_{0})}(S)$$
for window data $(\G,\mu),$ such that $\Lie(U^{+})=(\Lie(\G)_{W(k_{0})})_{1}.$
In general it might only be a map of sets. I will explain, why it is a group homomorphism for all frames that will be relevant.
\\
Let $h\in \G(\mathcal{F})_{\mu}$ and using Lemma \ref{Zerlegungslemma fuer die Window-gruppe} above, one writes it uniquely as a product $h=u\cdot p,$ where $u\in U^{+}(I)$ and $p\in P^{-}(S).$ By Fact \ref{Zusammenfassung zu Gewichtsgruppen} (iv), one can define
$$ \Phi_{\G,\mu,\mathcal{F}}\colon U^{+}(I) \rightarrow U^{+}(S)$$ to be the composition, as in the following commutative diagram
$$
\xymatrix{
U^{+}(I) \ar[d]^{\log^{+}} \ar[r]^{\Phi_{\G,\mu,\mathcal{F}}} & U^{+}(S) \\
\Lie(U^{+}) \otimes I \ar[r]^{\id\otimes \dot{\varphi}} & \Lie(U^{+}) \otimes S \ar[u]^{(\text{log}^{+})^{-1}}.
}
$$
Next, note that the very description of the functor of points of $P^{-}$ makes it possible to give the following definition: set $\Phi_{\G,\mu,\mathcal{F}}\colon P^{-}(S) \rightarrow P^{-}(S)$ to be $$\Phi_{\G,\mu,\mathcal{F}}=\mu(\zeta_{\mathcal{F}})\varphi(h)\mu(\zeta_{\mathcal{F}})^{-1}.$$ Recall that $\zeta_{\mathcal{F}}\in S$ was the frame-constant, i.e. the unique element such that $\varphi(i)=\zeta_{\mathcal{F}}\dot{\varphi(i)}.$
\\
The next statment will explain, why $\Phi_{\G,\mu,\mathcal{F}}$ is a group homomorphism for example for the frames $\mathcal{F}_{\infintesimal}(R)$ and $\mathcal{F}_{\cris}(R).$
\begin{Lemma}
Let $\mathcal{F}=(S,I,R,\varphi,\dot{\varphi})$ be a frame over $W(k_{0}),$ such that $S$ is $\zeta_{\mathcal{F}}$-torsion free. Then the divided Frobenius
$$\Phi_{\G,\mu,\mathcal{F}}\colon \G(\mathcal{F})_{\mu}\rightarrow \G_{W(k_{0})}(S)
$$
is a group homomorphism.
\end{Lemma}
\begin{proof}
Note that inside $S[1/\zeta_{\mathcal{F}}]$ the relation $\dot{\varphi}(i)=\frac{\varphi(i)}{\zeta_{\mathcal{
F}}}$ is true for all $i\in I.$ As $\log^{+}$ is $\mathbb{G}_{m,W(k_{0})}$-equivariant and $\Ad(\mu^{-1})((\zeta_{\mathcal{F}})^{-1})$ acts by \textit{division} by $\zeta_{\mathcal{F}}$ on $\Lie(U)^{+},$ one deduces that
$$\Phi_{\G,\mu,\mathcal{F}}(u)=\mu(\zeta_{\mathcal{F}})\varphi(u)\mu(\zeta_{\mathcal{F}})^{-1}\in \G_{W(k_{0})}(S[1/\zeta_{\mathcal{F}}])
,$$ for all $u\in U^{+}(I).$
Thus, one has in total for all $h\in \G(\mathcal{F})_{\mu}$
\begin{equation}\label{Beschreibung geteilter Frobenius nach invertieren}
\Phi_{\G,\mu,\mathcal{F}}(h)=\mu(\zeta_{\mathcal{F}})\varphi(h)\mu(\zeta_{\mathcal{F}})^{-1}\in \G_{W(k_{0})}(S[1/\zeta_{\mathcal{F}}]).
\end{equation}
This is clearly a group homomorphism and since $S$ is $\zeta_{\mathcal{F}}$-torsion free the map into the localization is injective. This concludes the easy verification.
\end{proof}
\begin{Remark}\label{Erklaerung weshalb Problem mit h-frame Struktur}
Note that the ring of Witt vectors $W(R)$ for a non-reduced ring in characteristic $p$  is  certainly not $p$-torsion free. Thus the above argument does not apply. Still the divided Frobenius for the Witt frame is a group homomorphism. In fact, one can prove this either using the argument given by Bültel-Pappas in \cite[Prop. 3.1.2.]{BueltelPappas}, or using the h-frame structure on the Witt frame introduced by Lau \cite[Example 2.1.3]{LauHigherFrames}. In general, one can say that the divided Frobenius $\Phi_{\mathcal{F}}$ is a group homomorphism either in case one can put a h-frame structure on $\mathcal{F}$ or in case $S$ is $\zeta_{\mathcal{F}}$-torsion free. One can in fact construct h-frame structures on $\mathcal{F}_{\infintesimal}(R)$ and $\mathcal{F}_{\cris}(R),$ which gives another reason, why the divided Frobenius is a group homomorphism in these cases. I decided to stick to the construction of $\G$-$\mu$-windows given here, as one will later need an auxillary frame, for which I was not able to give a h-frame structure.\footnote{This is the frame $\underline{A_{0}}$ from Lemma \ref{Framemorphismus von Acris mod p zu A0}.}
\end{Remark}
\begin{Definition}
Let $\mathcal{F}$ be a frame over $W(k_{0}),$ such that the divided Frobenius $\Phi_{\mathcal{F}}$ is a group homomorphism. Then the groupoid of banal \footnote{This 'banal' refers to the triviality of torsors. This usage was introduced in \cite{BueltelPappas}.} $\G$-$\mu$-windows over $\mathcal{F}$ is the quotient groupoid
$$
\G-\mu-\text{Win}(\mathcal{F})_{\banal}=[\G_{W(k_{0})}(S)/_{\Phi_{\G,\mu,\mathcal{F}}} \G(\mathcal{F})_{\mu}],
$$ 
for the action $g\ast h=h^{-1}g\Phi_{\G,\mu,\mathcal{F}}(h)$ for all $g\in \G_{W(k_{0})}(S)$ and $h\in \G(\mathcal{F})_{\mu}.$
\end{Definition}
\begin{Remark}
One has two kinds of functorialities for these objects: Let $$\mathcal{F}\in \lbrace \mathcal{W}(R),\mathcal{W}(S/R),\mathcal{F}_{\infintesimal}(R),\mathcal{F}_{\cris}(R) \rbrace.$$
\begin{enumerate}
\item[(i)]\textit{Functoriality in minuscule window-data:}
\\
A morphism of minuscule window-data $f\colon (\G_{1},\mu_{1})\rightarrow (\G_{2},\mu_{2})$ is a $\mathbb{Z}_{p}$-group scheme homomorphism $f\colon \G_{1}\rightarrow \G_{2},$ such that $f_{W(k_{0})}\circ \mu_{1}=\mu_{2}.$ It induces a functor
$$f\colon \G_{1}-\mu_{1}-\text{Win}(\mathcal{F})_{\banal}\rightarrow \G_{2}-\mu_{2}-\text{Win}(\mathcal{F})_{\banal}.$$
\item[(ii):]\textit{Functoriality in frame morphisms:}
\\
Consider a $u$-frame homomorphism
$\lambda\colon \mathcal{F}\rightarrow \mathcal{F}^{\prime},$ where $\mathcal{F},\mathcal{F}^{\prime}$ are frames, such that banal $\G$-$\mu$-windows are defined. Then one gets a base-change morphism
$$
\lambda_{\bullet}\colon \G-\mu-\text{Win}(\mathcal{F})_{\banal} \rightarrow \G-\mu-\text{Win}(\mathcal{F}^{\prime})_{\banal},
$$
induced by the map $\G(S)\rightarrow \G(S^{\prime}),$ $g\mapsto \lambda(g)\mu(u),$
in general only in case the frame-constant $\zeta_{\mathcal{F}^{\prime}}$ does not lead to torsion in $S^{\prime},$ or if $u=1.$ Let me remark here already that in Lemma \ref{chi ist kristalline!} one is faced with an $u$-frame morphism, where the above hypothesis does not apply. Nevertheless, one can check that this morphism induces a base-change functor, because in that example the frame-constants of both the source and the target are $0.$ 
\end{enumerate}
\end{Remark}
\subsection{Construction of $\G$-$\mu$-windows for some frames}
\subsubsection{$\G$-$\mu$-displays}
Here I will briefly recall the construction of $\G$-$\mu$-displays of Bültel-Pappas (\cite[Def. 3.2.1]{BueltelPappas}).
\\
Recall the following result due to Greenberg:
\begin{proposition}
Let $R$ be a ring and $X\rightarrow W_{n}(R)$ be an affine scheme of finite type (resp. finite presentation).
\\
Then the fpqc-sheaf
$$F_{n}(X)\colon R-\text{algebras}\rightarrow \Set,$$
$$F_{n}(X)(R^{\prime})=X_{W_{n}(R)}(W_{n}(R^{\prime}))$$
is representable by an affine scheme over $R$ of finite type (resp. of finite presentation).
\end{proposition}
In the situation of the proposition, one calls the $R$-scheme $F_{n}(X)$ the Greenberg-Transformation of X. It commutes with fiber-products. Thus, in case $X$ is a group scheme, one deduces that $F_{n}(X)$ is a $R$-group scheme.
\\
Let $(\G,\mu)$ be a window datum. Using the Cartier-morphism, one can consider the $W_{n}(W(k_{0}))$-group scheme $\G_{W(k_{0})}\times_{W(k_{0})} \Spec(W_{n}(W(k_{0}))).$ Then
$$L^{+}(\G_{W(k_{0})})=\lim_{n\geq 1} F_{n}(\G_{W(k_{0})}\times_{W(k_{0})} \Spec(W_{n}(W(k_{0}))))$$
is an affine $W(k_{0})$-group scheme with points in $W(k_{0})$-algebras $R$ given as
$$
L^{+}(\G_{W(k_{0})})(R)=\G_{W(k_{0})}(W(R)).
$$ One has to be careful, because it will not be of finite type or finite presentation anymore. By \cite[Prop. 2.2.1.,(d)]{BueltelPappas} it is flat and formally smooth over $W(k_{0}).$ One can consider the window group for the Witt frame as the sections of a closed subgroup scheme
$$\G(\mathcal{W})_{\mu}\hookrightarrow L^{+}(\G_{W(k_{0})}).$$ In case $(\G,\mu)$ is a minuscule window datum, one checks, see \cite[Prop. 3.1.2.]{BueltelPappas}, that the divided Frobenius extends to a group scheme homomorphism
$$
\Phi_{\G,\mu,\mathcal{W}}\colon \G(\mathcal{W})_{\mu} \rightarrow L^{+}(\G_{W(k_{0})}).
$$
Hence the following definition makes sense.
\begin{Definition}
Let $R\in \Nilp_{W(k_{0})}$ (or a $p$-adic $W(k_{0})$-algebra) and $(\G,\mu)$ a minuscule window datum. 
\\
A $\G$-$\mu$-displays over $R$ is a pair $(Q,\alpha),$ where $Q$ is a fpqc $\G(\mathcal{W})_{\mu}$-torsor on $\Spec(R)$ and 
$$\alpha\colon Q \rightarrow L^{+}(\G_{W(k_{0})})_{R}$$
is a map of fpqc-sheaves on $\Spec(R),$ such that 
$$\alpha(q\cdot h)=h^{-1}\alpha(q)\Phi_{\G,\mu,\mathcal{W}}(h)$$
for all $h\in \G(\mathcal{W})_{\mu}$ and $q\in Q.$
\end{Definition}
\begin{Remark}\label{Remark zur Hodge-Filtration}
Let $R$ be a $W(k_{0})$-algebra, such that $p$ is nilpotent.
In case $\G$ is assumed to be a reductive group scheme, then the datum of a $\G(\mathcal{W})_{\mu}$-torsor over $\Spec(R)$ is equivalent to the datum of a $\G$-torsor $\mathcal{E}$ on $\Spec(W(R))$ and a section of the proper morphism $\bar{\mathcal{E}}/P^{-}\rightarrow \Spec(R)$. Here one denotes by $\bar{\mathcal{E}}=P\times_{\Spec(W(R))} \Spec(R).$ See \cite[section 3.2.4.]{BueltelPappas}. In the $\GL_{n}$-case this just means that a display has an underlying finite projective module over the Witt vectors together with a filtration. Thus, I call the above section the \textit{Hodge-filtration}.
\end{Remark}
The next statement appears as \cite[Prop. 3.2.11.]{BueltelPappas} under a noetherian hypothesis.
\begin{Lemma}\label{Displays ueber adischen Ringen}
 Fix a minuscule and reductive window datum $(\G,\mu).$
Let $A$ be a $W(k_{0})$-algebra, that is $I$-adically complete and separated for an ideal $I,$ that contains a power of $p.$ Then there exists a natural equivalence
$$\G-\mu-\Displ(\Spec(A)) \simeq 2-\lim_{n\geq 1} \G-\mu-\Displ(\Spec(A/I^{n})).$$
\end{Lemma}
\begin{proof}
Given the arguments in the proof of \cite[Prop. 3.2.11.]{BueltelPappas}, one only has to explain, how to construct from a compatible system of $\G(\mathcal{W})_{\mu}$-torsors $Q_{n}$ over $\Spec(A/I^{n})$ for all $n\geq 1$ a $\G(\mathcal{W})_{\mu}$-torsors over $\Spec(A).$ By Remark \ref{Remark zur Hodge-Filtration} above, one has to show that a compatible system of $L^{+}(\G_{W(k_{0})})$-torsors $\mathcal{E}_{n}$ over $\Spec(A/I^{n})$ together with a compatible system of sections of $\mathcal{E}_{n}/\G(\mathcal{W})_{\mu}\rightarrow \Spec(A/I^{n})$ gives a $L^{+}(\G_{W(k_{0})})$-torsor $\mathcal{E}$ over $\Spec(A)$ together with a section of $\mathcal{E}/\G(\mathcal{W})_{\mu}\rightarrow \Spec(A).$ The $L^{+}(\G_{W(k_{0})})$-torsor $\mathcal{E}$ is constructed in the proof of the cited proposition \footnote{Where they do not make use of the noetherianess assumption!} and for the existence of the section one can use \cite[Remark 4.6.]{BhattTannaka}, which avoids the use of Grothendieck$^{\prime}$s algebraization theorem. In fact, the theorem of Bhatt gives that
$$(\mathcal{E}/\G(\mathcal{W})_{\mu})(\Spec(A))=\lim_{n\geq 1} (\mathcal{E}/\G(\mathcal{W})_{\mu})(\Spec(A/I^{n})).$$
But one has that $(\mathcal{E}/\G(\mathcal{W})_{\mu})\times_{\Spec(A)} \Spec(A/I^{n})\cong (\mathcal{E}_{n}/\G(\mathcal{W})_{\mu})$ and by assumption, there are sections $s_{n}\in (\mathcal{E}_{n}/\G(\mathcal{W})_{\mu})(\Spec(A/I^{n}).$
\end{proof}
\subsubsection{$\G$-$\mu$-windows for the frame $\mathcal{F}_{\infintesimal}$}
Let $(\mathcal{G},\mu)$ be a minuscule window datum. Fix an integral perfectoid $W(k_{0})$-algebra $R.$
\begin{Lemma}\label{Garbeneigenschaft Window Gruppe Inf-Frame}
\begin{enumerate}
\item[(a):] The pre-sheaf
$$\G(\mathcal{A}_{\infintesimal}(\mathcal{R}))\colon \Spec(R/p)^{\aff, \op}_{\et}\rightarrow \Grp$$ is a sheaf. The same is true for $\G(\mathcal{R}).$
\item[(b):] The pre-sheaf
$$\G(\mathcal{F}_{\infintesimal}(\mathcal{R}))_{\mu}\colon \Spec(R/p)^{\aff, \op}_{\et}\rightarrow \Grp$$ is a sheaf.
\end{enumerate}
\end{Lemma}
\begin{proof}
For (a), one just observes that as $\G/\mathbb{Z}_{p}$ is of finite type, one can consider $\G$ as the equalizer of maps 
$$
\xymatrix{
\mathbb{A}^{n}_{\mathbb{Z}_{p}} \ar@<1ex>[r] \ar@<-1ex>[r] & \mathbb{A}^{m}_{\mathbb{Z}_{p}}
}
$$
in the category of étale sheaves. Obviously, $(\mathcal{A}_{inf})^{k}(B^{\prime})=(\Ainf(R^{\prime}))^{k}$ for $k\geq 1$ are sheaves and therefore the claim follows from the fact that equalizer of étale sheaves are again étale sheaves.\footnote{The forgetful functor from sheaves to presheaf is a right adjoint, thus preserverse limits.} The same argument shows that also $\G(\mathcal{R})$ is a sheaf. For (b), one first notes that Fontaine$^{\prime}$s map induces a morphism of sheaves
$$\theta\colon \G(\mathcal{A}_{inf}(\mathcal{R})) \rightarrow \G_{W(k_{0})}(\mathcal{R}).$$
Consider the monomorphism of sheaves $\iota\colon P^{-}(\mathcal{R})\hookrightarrow \G_{W(k_{0})}(\mathcal{R})$ induced by the closed immersion $$P^{-}\subset \G_{W(k_{0})}$$ of group schemes. Then one can write the presheaf $\G(\mathcal{F}_{inf}(\mathcal{R}))_{\mu}$ as the fiber-product
$$\G(\mathcal{A}_{inf}(\mathcal{R})) \times_{\theta,\G_{W(k_{0})}(\mathcal{R}),\iota} P^{-}(\mathcal{R})$$ and the Lemma is proven, as a fiber product of étale sheaves is again an étale sheaf (same argument as in the footnote just made in this proof).
\end{proof}
As the frame-constant $\varphi(\xi)=\zeta_{\mathcal{F}_{inf}(\mathcal{R}(B^{\prime}))}$ for varying étale $B=R/p$-algebras $B^{\prime}$ does not lead to torsion in $\Ainf(\mathcal{R}(B^{\prime})),$ one has a group homomorphism
$$\Phi_{inf}\colon \G(\mathcal{F}_{inf}(\mathcal{R}))_{\mu} \rightarrow \G(\mathcal{A}_{inf}(\mathcal{R})).$$
Therefore one can make the following
\begin{Definition}\label{Definition von Windows fuer Ainf}
Let $R$ be an integral perfectoid $W(k_{0})$-algebra. A $\G$-$\mu$-window for the frame $\mathcal{F}_{inf}$ over $R$ is a tuple $(Q,\alpha),$
where
\begin{enumerate}
\item[(a):] $Q$ is a $\G(\mathcal{F}_{inf}(\mathcal{R}))_{\mu}$-torsor on $\Spec(R/p)_{\text{ét}}^{\text{aff}}$,
\item[(b):] $\alpha\colon Q\rightarrow \G(\mathcal{A}_{\infintesimal}(\mathcal{R}))$ is a morphism of sheaves, such that
$$\alpha(q\cdot h)=h^{-1}\alpha(q)\Phi_{\infintesimal}(h).$$
\end{enumerate}
\end{Definition}
In the following I will denote by $\G$-$\mu$-$\text{Win}(\mathcal{F}_{\infintesimal}(\mathcal{R}))$ the corresponding stack of $\G$-$\mu$-windows over $R.$ In other words, this is just the quotient stack
$$[\G(\mathcal{A}_{\infintesimal}(\mathcal{R}))/_{\Phi_{\infintesimal}}\G(\mathcal{F}_{\infintesimal}(\mathcal{R}))_{\mu}]$$ over $\Spec(R/p)_{\text{ét}}^{\text{aff}}$.
\begin{Remark}
One similiarly sees, that one can construct $\G$-$\mu$-windows for prisms giving rise to frames.
\end{Remark}
For later use, let me record a description of banal $\G$-$\mu^{\sigma}$-windows for $\mathcal{F}_{\infintesimal}(R),$ whose proof is of course very similiar to \cite[Prop. 3.2.15.]{BueltelPappas}.
\begin{Lemma}\label{Windows ueber Ainf und BKF lokales Statement}
Let $R$ be an integral perfectoid $W(k_{0})$-algebra as before. There exists an equivalence of groupoids
$$[\G(\Ainf(R))/_{\Phi_{\infintesimal},\G,\mu^{\sigma}}\G(\mathcal{F}_{\infintesimal}(R))_{\mu}] \simeq [\G(\Ainf(R))\mu^{\sigma}(\varphi(\xi))\G(\Ainf(R))/_{\varphi}\G(\Ainf(R))].$$
\end{Lemma}
\begin{proof}
Note at first that $\mu$ being minuscule, implies that\footnote{In fact, one inclusion is obvious and the other one follows from the description of the window group: multiplication induces an isomorphism $U^{+}((\xi))\times P^{-}(\Ainf(R))\rightarrow \G(\mathcal{F}_{\infintesimal}(R))_{\mu}.$ Let $u\in U^{+}((\xi)),p\in P^{-}(\Ainf(R)),$ then $\mu(\xi)^{-1}u\mu(\xi)$ is the unit section and $\mu(\xi)p\mu(\xi)^{-1}\in \G(\Ainf).$ It follows that $u\cdot p\in \G(\Ainf)\cap \mu(\xi)^{-1}\G(\Ainf(R))\mu(\xi).$}
\begin{equation}\label{Beschreibung der Windowgruppe fuer Ainf}
\G(\mathcal{F}_{\infintesimal}(R))_{\mu}=\mu(\xi)^{-1}\G(\Ainf(R))\mu(\xi)\cap \G(\Ainf(R)).
\end{equation}
One deduces the equivalence
$$
[\G(\Ainf(R))/_{\Phi_{\infintesimal,\G,\mu^{\sigma}}}\G(\mathcal{F}_{\infintesimal}(R))_{\mu}]\rightarrow [\G(\Ainf(R))\varphi(\mu(\xi))/_{\varphi}\mu(\xi)^{-1}\G(\Ainf(R))\mu(\xi)\cap \G(\Ainf(R))],
$$
given by sending $g\mapsto g\varphi(\mu(\xi)).$ But using (\ref{Beschreibung der Windowgruppe fuer Ainf}) above, one deduces that the inclusion induces an equivalence between 
$$ 
[\G(\Ainf(R))\varphi(\mu(\xi))/_{\varphi}\mu(\xi)^{-1}\G(\Ainf(R))\mu(\xi)\cap \G(\Ainf(R))]$$ and  $$[\G(\Ainf(R))\mu^{\sigma}(\varphi(\xi))\G(\Ainf(R))/_{\varphi}\G(\Ainf(R))].
$$
Indeed, to show fully faithfullness, assume that
$$g_{1}\varphi(\mu(\xi))=h^{-1}g_{2}\varphi(\mu(\xi)) \varphi(h).$$
This implies that
$$\varphi(h)=\varphi(\mu(\xi))^{-1}g_{2}^{-1}hg_{1}\varphi(\mu(\xi)).$$
Then (as $\varphi$ is an automorphism of $\Ainf$) one deduces by (\ref{Beschreibung der Windowgruppe fuer Ainf}) above that $$h\in \mu(\xi)\G(\Ainf(R))\mu(\xi)^{-1}\cap \G(\Ainf(R)).$$
\\
It is essentially surjective: take $g_{1}\varphi(\mu(\xi))g_{2}.$ Then just write $g_{2}=\varphi(h)^{-1},$ for $h\in\G(\Ainf(R)).$ It follows that
$$h^{-1}g_{1}\varphi(\mu(\xi))g_{2}\varphi(h)\in \G(\Ainf(R))\varphi(\mu(\xi)).$$
\end{proof}
\begin{Remark}
Let $R$ be a perfect $\mathbb{F}_{p}$-algebra. Then $\G$-$\mu$-windows for $\mathcal{F}_{inf}(R)$ are the same as $\G$-$\mu$-displays over $R.$ Thus the previous lemma recovers and extends \cite[Prop. 3.2.15.]{BueltelPappas}.
\end{Remark}
\subsubsection{$\G$-$\mu$-windows for the frame $\mathcal{F}_{\text{cris}}$} The construction for these windows works exactly the same way as for the frame $\mathcal{F}_{\infintesimal}.$ 
\\
Fix a $p$-torsion free integral perfectoid $W(k_{0})$-algebra $R.$ Then $\G(\mathcal{A}_{\cris}(\mathcal{R}))$ and $\G(\mathcal{F}_{\text{cris}}(\mathcal{R}))_{\mu}$ are still sheaves for $\Spec(R/p)^{\text{aff}, \text{op}}_{\text{ét}}$ and as all rings $\Acris(\mathcal{R})$ are $p$-torsion free, one has a group homomorphism
$$
\Phi_{\text{cris}}\colon \G(\mathcal{F}_{\text{cris}}(\mathcal{R}))_{\mu} \rightarrow \G(\mathcal{A}_{\text{cris}}(\mathcal{R})).
$$
Thus, one can copy-paste the previous Definition \ref{Definition von Windows fuer Ainf}.
\subsection{Adjoint nilpotency condition}
In this section, I briefly want to recall the adjoint nilpotency condition for $\G$-$\mu$-displays, as formulated by Bültel-Pappas in \cite[Def. 3.4.2]{BueltelPappas}. As in the theory of classical Zink-displays, this nilpotency condition is needed to develop the deformation theory - but one should be careful that under the equivalence of stacks between classical Zink-displays and $(\GL_{n},\mu_{n,d})$-displays, the nilpotency conditions do not coincide.
\\
Let $k$ be an algebraically closed field of characteristic $p,$ $W=W(k)$ and $L=W(k)[1/p]$ the fraction field. Let $\sigma$ be the Frobenius-automorphism of $L.$ Furthermore, let $(\G,\mu)$ be a minuscule and reductive window datum and denote by $G=\G_{\mathbb{Q}_{p}}$ the generic fiber. If $b\in G(L),$ Kottwitz \cite{KottwitzIsocrystal1} associates a morphism of $L$-groups
$$
\nu_{b}\colon \mathbb{D}_{L}\rightarrow G_{L},
$$
characterized by the property, that for any $(V,\rho)\in \text{Rep}_{\mathbb{Q}_{p}}(G),$ the $\mathbb{Q}$-graduation on $V_{L}$ given by $\rho\circ \nu_{b}$ coincides with the decomposition in isotopic components of the isocristal $(V_{L},\rho(b)\circ (1\otimes \sigma)).$ Thus one can say that $\rho(b)$ has a slope at $\lambda\in \mathbb{Q}.$ A trivial observation is that this requirement is independent of the choice of a representative $b\in [b]\in B(G).$ 
By Lemma \ref{Windows ueber Ainf und BKF lokales Statement}, one deduces
\begin{Corollary}
Let $(Q,\alpha)$ a $\G$-$\mu$-display over $k.$ After trivialization, one may represent it by a section $g\in \G(W).$ Then the requirement that $\Ad(g\mu^{\sigma}(p))$ has all slopes greater than $-1$ is independent of the choice of the trivialization.
\end{Corollary}

Thus, the next definition makes sense.
\begin{Definition}
Let $R\in \Nilp_{W(k_{0})}$ and $(Q,\alpha)$ a $\G$-$\mu$-display over $\Spec(R).$ Then $(Q,\alpha)$ is called adjoint nilpotent, if for all geometric points
$$
\bar{x}_{k}\colon \Spec(\overline{\kappa(x)})\rightarrow \Spec(R),
$$
the pull-back $(\bar{x}_{k})^{*}(Q,\alpha)$ fulfills the requirement of the last corollary.
\end{Definition}
Recall the following translation of the adjoint nilpotency condition in the case of a trivialized $\G$-$\mu$-display given in \cite[3.4.4.]{BueltelPappas}.
\begin{proposition}
Let $R$ be $k_{0}$-algebra and $(Q,\alpha)$ a trivialized $\G$-$\mu$-display over $R,$ given by a section $g\in L^{+}(\G)(R).$ Then $(Q,\alpha)$ is adjoint nilpotent, if and only if the endomorphism
$$
\Ad(w_{0}(g))\circ (\text{Frob}_{R}\otimes \text{id}) \circ (\text{id}_{R}\otimes \pi)\colon \Lie(U^{+})_{R}\rightarrow \Lie(U^{+})_{R}
$$
is nilpotent. Here I denote by $\pi\colon \Lie(\G)\rightarrow \Lie(U^{+})$ the projection onto $\Lie(U^{+}),$ killing $\Lie(P^{+}).$
\end{proposition}
\begin{Remark}
Let me briefly comment on the relation between this nilpotency condition and Zink$^{\prime}$s nilpotency condition in the case of the linear group. Let $R\in \Nilp$ and $\mathcal{P}$ be a classical Zink-display over $R.$ Then $\mathcal{P}$ is adjoint nilpotent if and only if there are radical ideals $I_{nil}\subseteq R$ and $I_{uni}\subseteq R,$ with $I_{nil}\cap I_{uni}=\sqrt{pR},$ such that $\mathcal{P}_{R/I_{nil}}$ and $(\mathcal{P}/I_{uni})^{t}$ are Zink-nilpotent. Here $()^{t}$ means the dual display. This is a reproduction of \cite[Remark 3.4.5]{BueltelPappas}.
\end{Remark}
This is the motiviation to give the following
\begin{Definition}\label{Definition adjoint nilpotent fuer kristalline Windows}
Let $R$ be a $p$-torsion free integral perfectoid $W(k_{0})$-algebra. Let $$\mathcal{F}\in \lbrace \mathcal{F}_{\cris}(R),\mathcal{F}_{\cris}(R/pR) \rbrace.$$ Then a $\G$-$\mu$-window $\mathcal{P}=(Q,\alpha)$ over $\mathcal{F}$ is said to fulfill the adjoint nilpotency condition, if étale-locally on $\Spec(R/pR)$, one can represent $\mathcal{P}$ by a structure matrix $g\in\G(\Acris(R)),$ such that the endomorphism
$$
\psi_{g}\colon \Lie(U^{+})_{\Acris(R)}\rightarrow \Lie(U^{+})_{\Acris(R)},
$$
where $\psi_{g}=(\id_{Lie(U^{+})}\otimes\varphi)\circ \pr_{2} \circ \Ad(g),$ is nilpotent modulo $\text{Fil}(\Acris(R))+p\Acris(R).$
\end{Definition}
The same definition works for not necessarily $p$-torsion free integral perfectoid $W(k_{0})$-algebras for banal $\G$-$\mu$-windows for the frames $\mathcal{F}_{\cris}(R)$ resp. $\mathcal{F}_{\cris}(R/pR)$ (remember that the worries with the $p$-torsion always came from the necessity of having descent statements as in Lemma \ref{Descent fuer Acris}).
\begin{Remark}
Let me briefly point out why this condition is independent of the choice of a structure matrix $g\in \G(\Acris(R)).$ In fact, let me reformulate the condition that the above $\varphi$-linear endomorphism $\psi_{g}=(\id_{\Lie(U^{+})}\otimes\varphi)\circ \pr_{2} \circ \Ad(g)$ is nilpotent modulo $\text{Fil}(\Acris(R))+p\Acris(R)$ as follows: first, let $\overline{g}\in \G(R/pR)$ be the reduction of $g$ modulo $\text{Fil}(\Acris(R))+p\Acris(R).$ The condition that $\psi_{\overline{g}}$ is nilpotent is equivalent to the condition that its base change towards $R/\sqrt{pR}$ is nilpotent; since $R$ is integral perfectoid, this is a perfect ring in characteristic $p.$\footnote{Indeed, if $\varpi\in R$ is some perfectoid pseudo-uniformizer, then one has that $\sqrt{pR}=\bigcup_{n} \varpi^{1/p^{n}}.$} Then one has the following commutative diagram
$$
\xymatrix{
\Acris(R) \ar[r] \ar[d] & \Acris(R/\sqrt{pR})=W(R/\sqrt{pR}) \ar[d] \\
R/pR \ar[r] & R/\sqrt{pR}.
}
$$
Denoting by $g_{\text{red}}\in \G(W(R/\sqrt{pR}))$ the image of $g$ under $\Acris(R)\rightarrow W(R/\sqrt{pR}),$ the condition that $\psi_{\overline{g}}$ is nilpotent is therefore equivalent to the condition that $\psi_{g_{\text{red}}}$ is nilpotent modulo $p.$ By \cite[Lemma 3.4.4]{BueltelPappas}, this last condition just depends on the $\G$-$\mu$-display that $g_{\text{red}}$ represents, which concludes the verification.
\end{Remark}
\subsection{Lau's unique lifting lemma}
Here I want to reproduce the unique lifting lemma, as proven by Lau in \cite[Proposition 7.1.5.]{LauHigherFrames}. Since I am using a slightly different language than him and since this a central techniqual tool, I decided to give a detailed explanation of the proof.
\begin{Definition}{(Lau)} 
Let $\lambda\colon \mathcal{F}\rightarrow \mathcal{F}^{\prime}$ be a $u$-frame morphism, such that $\G$-$\mu$-windows over $\mathcal{F}$ resp. $\mathcal{F}^{\prime}$ are defined. 
Then one calls $\lambda$ nil-crystalline, if the base-change functor $\lambda_{\bullet}$ induces an equivalence of the corresponding adjoint nilpotent $\G$-$\mu$-window categories. It is called crystalline, if it is an equivalence even without putting the adjoint nilpotency condition.
\end{Definition}
The term 'crystalline' was introduced by Lau in \cite[Def. 3.1.]{LauFrames}.
\\
Now I want to explain the key criterium (the 'unique lifting lemma') for when the base change along a frame-morphism is fully faithful; the question of essential surjectivity then comes down to the question whether one can lift a section $g^{\prime}\in \G(S^{\prime})$ towards $g\in \G(S),$ which in the applications will be ok since then $S^{\prime}\rightarrow S$ is surjective and $S^{\prime}$ is henselian along $\ker(S^{\prime}\rightarrow S).$
\\
Let $\mathcal{F}$ and $\mathcal{F}^{\prime}$ two frames over $W(k_{0})$.
One makes the following assumptions:
\begin{enumerate}
\item[(a):] There exists a strict frame morphism
$$\lambda\colon \mathcal{F}\rightarrow \mathcal{F}^{\prime},$$
given by a surjective ring homomorphism $\lambda\colon S\rightarrow S^{\prime},$ that induces an isomorphism $R\simeq R^{\prime},$
\item[(b):] let $K=\ker(\lambda),$ an ideal that by assumption (a) lies in $I.$ Then one requires that $K$ is $p$-adically complete and separated,
\item[(c):] finally, one requires that $\dot{\varphi}(K)\subseteq K.$
\end{enumerate}
\begin{proposition}{(Lau)}\label{Unique lifting lemma}
Let $\mathcal{F}$ and $\mathcal{F}^{\prime}$ two frames over $W(k_{0})$ together with a strict morphism $$\lambda\colon \mathcal{F}\rightarrow \mathcal{F}^{\prime},$$ fulfilling the above hypotheses $(a)$,$(b)$ and $(c)$. Furthermore let $g_{1},g_{2}\in \G_{W(k)}(S).$ Now assume either
\begin{enumerate}
\item[$\bullet$]  that the endomorphisms of $\Lie(U^{+})\otimes_{W(k)} S$ given by $$(\id_{\Lie(\G)}\otimes \varphi)\circ \pi \circ \Ad(g_{i})$$ are nilpotent, where $i=1,2$ (recall that here $\pi$ was the projection from $\Lie(\G)$ towards $\Lie(U^{+})$),
\item[$\bullet$] or that the action of $\dot{\varphi}$ on $K$ is topologically nilpotent.
\end{enumerate}
Then, if $\lambda(g_{1})=\lambda(g_{2}),$ there exists a unique $h\in \G_{W(k)}(\mathcal{F})_{\mu},$ such that
\begin{equation}
g_{2}=h^{-1}g_{1}\Phi_{\mathcal{F}}(h).
\end{equation}
\end{proposition}
\begin{proof}
One first notes that the analog of \cite[Lemma 7.1.4.]{LauHigherFrames} is of cours true in this set-up and proved word by word the same way. Let me state it explictly for the reader's convenience
\begin{Lemma}\label{Vorbereitung unique lifting lemma}
\begin{enumerate}
\item[(a):] Let $\G(K)=\ker(\G(S)\rightarrow \G(S^{\prime}))$ and $\G(K)_{\mu}=\ker(\G(\mathcal{F})_{\mu}\rightarrow \G(\mathcal{F}^{\prime})_{\mu}).$ Then there is a natural identification $$\G(K)\cong\G(K)_{\mu},$$
\item[(b):] $\lambda$ induces a surjective group homomorphism
$$\G(S)\rightarrow \G(S^{\prime}),$$
\item[(c):] $\lambda$ induces a surjective group homomorhism, $$\G(\mathcal{F})_{\mu}\rightarrow \G(\mathcal{F^{\prime}})_{\mu}.$$
\end{enumerate}
\end{Lemma}
Coming back to the proof of Prop. \ref{Unique lifting lemma}, note that this lemma implies that one has to show that the action by $\Phi$-conjugation is simply transitive on the fibers, i.e. one has to show that for all $g\in\G(S)$ and all $h\in\G(K),$ there exists a unique $z\in\G(K)_{\mu},$ such that
\begin{equation}\label{Gleichung fuer transitive Operation}
z^{-1}g\Phi_{\mathcal{F}}(z)=hg.
\end{equation}
Now one reformulates (\ref{Gleichung fuer transitive Operation}) in a different way, which is more tractable to proving: one has an identification $\G(K)\simeq \G(K)_{\mu}$ by part (a) of the previous lemma. Thus for all $g\in \G(S),$ one can consider the following group homomorphism
$$ \mathcal{U}_{g}\colon \G(K) \rightarrow \G(K)$$
$$ z\mapsto g(\Phi_{\mathcal{F}}(z))g^{-1}.$$
Then (\ref{Gleichung fuer transitive Operation}) is equivalent to
\begin{equation}
z^{-1}\mathcal{U}_{g}(z)=h.
\end{equation}
Thus one has to show the following
\\
\textbf{Claim:} The map $$\G(K)\rightarrow \G(K),$$ $$z\mapsto \mathcal{U}_{g}(z)^{-1}z$$
is bijective.
\\
Note that the assumptions that $K$ is $p$-adic and $\G$ is smooth, together imply that
$$\G(K)\cong\lim_{i\geq 1}\G(K/p^{i}K).$$
Furthermore, consider
$$\G(p^{i}K):=\ker(\G(K)\rightarrow \G(K/p^{i}K)).$$
I claim that $\mathcal{U}_{g}$ preserves $\G(p^{i}K),$ then $\mathcal{U}_{g}$ will operate on $\G(K/p^{i}K)\simeq \G(K)/\G(p^{i}K)\footnote{As $\G$ is smooth, the map $\G(K)\rightarrow \G(K/p^{i}K)$ is surjective.}$ and I furthermore claim that it does so pointwise nilpotently. Then the above \textbf{Claim} follows.
In fact, for $z\in\G(K)$ the product $\Pi:=\cdot\cdot\cdot\mathcal{U}^{2}_{g}(z)\mathcal{U}_{g}(z)z$ then makes sense and it satisfies $\mathcal{U}_{g}(\Pi)^{-1}\Pi=z.$ The injectivity follows, because in case $\mathcal{U}_{g}(z)^{-1}z=\mathcal{U}_{g}(z^{\prime})^{-1}z^{\prime},$ the product $z^{\prime}z^{-1}$ is fixed by $\mathcal{U}_{g}$. But $\mathcal{U}_{g}$ is pointwise nilpotent and thus has as a unique fixpoint the unit section. Therefore, to finish the proof, one has to show the two above properties.
\\
For this, one first shows that $\Phi_{\mathcal{F}}$ respects the decomposition $\G(K)=U^{+}(K)\times P^{-}(K)$ coming from the (omitted) proof of the previous lemma.
\\
On $U^{+}(K)\cong\Lie(U^{+})\otimes_{W(k)} K,$ one has by definition $\Phi_{\mathcal{F}}=\id \otimes \dot{\varphi}$ and because by assumption $\dot{\varphi}(K)\subseteq K$ this is ok.
\\
Take a section $g\in P^{-}(K)$. This corresponds to a morphism of $W(k)$-algebras
$$g^{\sharp}\colon A/(A_{>0})\rightarrow S,$$ such that $\lambda(g^{\sharp}(f))=e^{\sharp}_{S^{\prime}}(f),$ where $e_{S^{\prime}}\in P^{-}(S^{\prime})$ the unit section. In terms of the weight-decomposition 
$$A=\bigoplus_{n\in \mathbb{Z}} A_{n},$$
one knows that $e^{\sharp}_{S^{\prime}}(f_{n})=0,$ for all $n\neq 0.$\footnote{This follows for example because $P^{+}$ and $P^{-}$ are both subgroup schemes of $\G$.} Thus, writing $f=\sum_{n\geq 0} f_{-n}\in A/(A_{>0}),$ the assumption that $g\in P^{-}(K)$ implies $g^{\sharp}(f_{-n})\in K$ for $n\geq 1$ and that $\lambda(g^{\sharp}(f_{0}))=e_{S^{\prime}}^{\sharp}(f_{0}).$ Recall that one defined
$$\Phi_{\mathcal{F}}\colon P^{-}(S)\rightarrow P^{-}(S)$$ by $g\mapsto \mu(p)\varphi(g)\mu(p)^{-1}$. This means that
$$(\Phi(g))^{\sharp}(f_{-n})=p^{n}\varphi(g^{\sharp}(f_{-n})),$$
for $n\geq 0.$ Thus, it follows for $n\geq 1,$ that $(\Phi(g))^{\sharp}(f_{-n})\in K,$ because $\varphi$ also lets $K$ stable \footnote{One has $\lambda(\varphi(x))=\varphi^{\prime}(\lambda(x))=0,$ for $x\in K.$}. Furthermore, the following identities hold: $$\lambda((\Phi(g))^{\sharp}(f_{-0}))=\lambda(\varphi(g^{\sharp}(f_{0}))=\varphi^{\prime}(\lambda(g^{\sharp}(f_{0})))=\varphi^{\prime}(e^{\sharp}_{S^{\prime}}(f_{0}))=e^{\sharp}_{S^{\prime}}(f_{0}).$$ Here one uses that $\varphi^{\prime}$ induces a group endomorphism on $P^{-}(S^{\prime})$ and thus sends the unit section to the unit section. In total, it follows that
$$
\lambda((\Phi(g))^{\sharp}(\sum_{n\geq 1}f_{-n})))=\lambda(\sum_{n\geq 1} p^{n}\varphi(g^{\sharp}(f_{-n})))=\lambda(\varphi((f_{-0})))=e^{\sharp}_{S^{\prime}}(f).
$$
Note that it also follows that $\Phi$ stabilzes $U^{+}(p^{i}K)$ and also $P^{-}(p^{i}K)$. One deduces that also $\mathcal{U}_{g}$ stabilizes $U^{+}(p^{i}K)$ and $P^{-}(p^{i}K)$ and the first claim above is proven. Now let me turn to the second:
\\
By induction it is sufficient to show that $\mathcal{U}_{g}$ acts pointwise nilpotently on all $$\G(p^{i}K/p^{i+1}K)\cong \G(p^{i}K)/\G(p^{i+1}K)$$ to conclude that $\mathcal{U}_{g}$ is pointwise nilpotent on all $\G(K/p^{i}K),$ for all $i\geq 0.$
\\
Now one shows the claimed pointwise nilpotency of $\mathcal{U}_{g}$ on $\G(p^{i}K/p^{i+1}K).$
For this, one first claims that
$$\Phi_{\mathcal{F}}\colon P^{-}(p^{i}K/p^{i+1}K)\rightarrow P^{-}(p^{i}K/p^{i+1}K)$$ is the zero map. In fact, observe that $P^{-}(p^{i}K/p^{i+1}K)\simeq \Lie(P^{-})\otimes p^{i}K/p^{i+1}K.$ Under this identification $\Phi_{\mathcal{F}}$ corresponds to $\cdot p^{-m}\varphi$ on weight $m$-components, where $m\leq 0.$ As $\varphi=p\dot{\varphi}$ on $K,$ the claim is ok.
\\
Therefore, one has the following factorization of the endormorphism $\mathcal{U}_{g}:$
$$\xymatrix{
\G(K_{i}) \ar[r]^{pr_{2}} & U^{+}(K_{i}) \ar[r]^{\id\otimes \dot{\varphi}} & U^{+}(K_{i}) \ar[r]^{z\mapsto gzg^{-1}} & \G(K_{i}),
}$$
here I wrote $K_{i}:=p^{i}K/p^{i+1}K$ to not exceed the margin. This is pointwise nilpotent, if the permutation
$$
\xymatrix{
U^{+}(K_{i}) \ar[r]^{z\mapsto gzg^{-1}} & \G(K_{i}) \ar[r]^{pr_{2}} & U^{+}(K_{i}) \ar[r]^{\id\otimes \dot{\varphi}} & U^{+}(K_{i})
}
$$
is pointwise nilpotent.
But the morphism of pointed sets $U^{+}(K_{i}) \rightarrow U^{+}(K_{i}),$ defined by $z\mapsto pr_{2}(zgz^{-1})$ is described by some power series, whose higher degree terms are annihaled by $\id\otimes \dot{\varphi},$ because one has $\dot{\varphi}(ab)=p\dot{\varphi}(a)\dot{\varphi}(b),$ for all $a,b\in K.$ Therefore,
$$\mathcal{U}_{g}\colon \Lie(U^{+})_{S}\otimes_{S} K_{i} \rightarrow \Lie(U^{+})_{S}\otimes_{S} K_{i}$$
is given by $\lambda \otimes \dot{\varphi},$ where $\lambda$ is the following $\varphi$-linear endomorphism of $\Lie(U^{+})\otimes_{W(k)} S$:
$$(1\otimes \varphi)\circ \pi \circ \Ad(g).$$
Now using either the assumption that this $\varphi$-linear endomorphism is nilpotent or the assumption that $\dot{\varphi}$ acts topologically nilpotent on $K,$ one may conclude the proof of the statement.
\end{proof}
\section{Crystalline equivalence}\label{Section crystalline equivalence}
Fix as usual a 1-bound window datum $(\G,\mu)$.
Let $R$ be a $p$-torsion free integral perfectoid $W(k)$-algebra. In this section I will first construct a strict framemorphism
$$
\chi\colon \mathcal{F}_{\text{cris}}(R)\rightarrow \mathcal{W}(R)
$$
and then proceed to show the following statement
\begin{proposition}(Crystalline equivalence)\label{Kristalline Aequivalenz}
\\
Let $R$ be a $p$-torsion free integral perfectoid $W(k)$-algebra. Then the base change functor
$$
\chi_{\bullet}\colon \G\text{-}\mu\text{-}\Win(
\mathcal{F}_{\text{cris}}(R/p))_{\nilp}\rightarrow \G\text{-}\mu\text{-}\Displ(R/p)_{\nilp}
$$
is an equivalence (i.e. $\chi$ is nil-crystalline).
\end{proposition}
Recall that Zink proves in \cite{ZinkWindow} the $\GL_{n}$-case of the crystalline equivalence (in fact, he proves a more general statement). He writes down a quasi-inverse of the base-change functor $\chi_{\bullet},$ by evaluating the crystal of a nilpotent display at the (topological) pd-thickening $\Acris(R)\rightarrow R$ to get the finite projective module over $\Acris(R)$ and then gets the Frobenius-structure by applying functoriality of the crystal to the Frobenius-isogeny of the display. As one has to work here with torsors, this route is not available. Instead, the idea is to first prove the statement $^{\prime}$ \textit{modulo }$p$ $^{\prime}$ directly by applying Lau's unique lifting lemma, Proposition \ref{Unique lifting lemma}, and then study the relation between lifts and lifts of the Hodge-filtration in the present set-up, to deduce the full statement. In this second step the adjoint nilpotency condition will be crucial.
\subsection{Crystalline equivalence modulo $p$}
Consider still $R$ a $p$-torsion free integral perfectoid $W(k)$-algebra. Then one has the frame $
\mathcal{F}_{\text{cris}}(R/p)=(\Acris(R), \Fil(\Acris(R))+p\Acris(R),R/pR,\varphi,\dot{\varphi}),
$
that of course still has frame-constant $p.$
By Lemma \ref{Framemorphismus von Acris nach W(R)} below,
one has a strict frame-morphism $$\chi/ p\colon \mathcal{F}_{\text{cris}}(R/p)\rightarrow \mathcal{W}(R/p),$$ and the aim of this section is to show that
$$
(\chi/ p)_{\bullet}\colon \G\text{-}\mu\text{-}\text{Win}(
\mathcal{F}_{\text{cris}}(R/p))_{\text{nilp}}\rightarrow \G\text{-}\mu\text{-}\text{Displ}(R/p)_{\text{nilp}}
$$
is an equivalence of groupoids. By Lemma,\ref{Displays sind ein Stack fuer Laus Garbe}, even further below, one sees that the fibered category on $\Spec(R/p)^{\text{aff}}_{\text{ét}},$ given by sending $B=R/p\rightarrow B^{\prime}=R^{\prime}/p,$ $\mathcal{R}(B^{\prime})=R^{\prime},$ towards the groupoid of $\G$-$\mu$-displays over $R^{\prime}$ is a stack; furthermore, for $\G$-$\mu$-Win($\mathcal{F}_{\text{cris}}(R)$) one has built this descent-behaviour into the very definition. Therefore one is immediately reduced to show that $(\chi/ p)$ induces an equivalence on the corresponding quotient-groupoids of banal and adjoint nilpotent windows resp. displays. Let me state the statement explicitely (dropping as usual in the statements that concern the banal case the assumption that the integral perfectoid ring is $p$-torsion free).
\begin{Lemma}
Let $R$ be an integral perfectoid $W(k_{0})$-algebra. Then 
$$
(\chi/ p)_{\bullet}\colon \G\text{-}\mu\text{-}\Win(
\mathcal{F}_{\text{cris}}(R/p))_{\nilp, \banal}\rightarrow \G\text{-}\mu\text{-}\Displ(R/p)_{\nilp, \banal}
$$
is an equivalence.
\end{Lemma}
Let me now operate in slightly greater generality. Consider a $p$-adic and $p$-torsion free ring $A,$ a $p$-adic ring $R,$ that is the quotient of $A$ by a pd-ideal $\mathfrak{a}\subseteq A.$ This means that one requires here that for all $m\geq 0$ and all $a\in \mathfrak{a}$ one has that $\gamma_{m}(a)\in \mathfrak{a}.$ Assume that there exists a Frobenius-Lift
$$
\varphi\colon A\rightarrow A.
$$ 
Since $A$ is assumed to be $p$-torsion free, it follows that one can consider $\dot{\varphi}=\frac{\varphi}{p}\colon \mathfrak{a}\rightarrow A$ and one has a frame 
$$
\underline{A}=(A,\mathfrak{a},R,\varphi,\dot{\varphi}).
$$
\begin{Lemma}\label{Framemorphismus von Acris nach W(R)}
There exists a strict frame-morphism
$$
\chi\colon \underline{A}\rightarrow \mathcal{W}(R).
$$
\end{Lemma}
\begin{proof}
This is contained in \cite[Corollary 2.40.]{ZinkVorlesung}. The frame-morphism is constructed as the composition of the Cartier-morphism
$$
\delta\colon A\rightarrow W(A),
$$
with the homomorphism $W(A)\rightarrow W(R).$\footnote{I am really using $W(\Acris(R))$ here!}
\end{proof}
In this setting, the groupoid of banal $\G$-$\mu$-windows for the frame $\underline{A}$ is given by
$$
[\G(A)/_{\Phi \underline{A}}\G(\underline{A})_{\mu}]
$$
(since the frame-constant is here $p,$ which by assumption is a regular element in $A$ so that the divided Frobenius $\Phi_{\underline{A}}$ is indeed a group homomorphism). An element $U\in \G(A)$ is said to satisfy the adjoint nilpotency condition, if the endomorphism
$$
\psi_{U}\colon \Lie(U^{+})_{A}\rightarrow \Lie(U^{+})_{A},
$$
given by $\psi_{U}=(\id_{\Lie(U^{+})}\otimes\varphi)\circ \pi \circ \Ad(U)$ is nilpotent modulo $\mathfrak{a}+pA.$
Now one may deduce the crystalline equivalence $^{\prime}$ \text{modulo }$p$ $^{\prime}$:
\begin{Lemma}\label{Kristalline aequivalenz modulo p}
Keep the assumptions in Lemma \ref{Framemorphismus von Acris nach W(R)}, but add the assumption that $R$ is a semi-perfect $\mathbb{F}_{p}$-algebra. Then for the strict frame-morphism $\chi\colon \underline{A}\rightarrow \mathcal{W}(R)$  the base-change functor 
$$\chi_{\bullet}\colon \G\text{-}\mu\text{-}\Win(\underline{A})_{\banal,\nilp}\rightarrow \G\text{-}\mu\text{-}\Displ(R)_{\banal,\nilp} $$
is fully faithful.
\end{Lemma}
\begin{proof}
Write $K=\ker(\chi\colon A\rightarrow W(R)).$
I want to apply the unique lifting lemma Proposition \ref{Unique lifting lemma}. Thus one first has to show that $\chi$ satisfies the following hypotheses:
\begin{enumerate}
\item[(a):] $\chi\colon A\rightarrow W(R)$ is a surjection,
\item[(b):] $\ker(\chi)$ is $p$-adically complete,
\item[(c):] $\dot{\varphi}$ stabilizes $K,$ i.e. $\dot{\varphi}(K)\subseteq K.$
\end{enumerate}
First of all, by definition both $\underline{A}$ and $\mathcal{W}(R)$ are frames over the same ring $R.$ Furthermore, as $R$ is semiperfect and $VF=FV=p$ in $W(R),$ it follows that $I(R)=pW(R).$ As both rings $A$ and $W(R)$ are $p$-adically complete, it suffices to show that $\chi$ is surjective after reduction modulo $p.$ This follows because $p\in \mathfrak{a},$ thus $A/pA$ surjects onto $A/\mathfrak{a},$ which is isomorphic to $W(R)/pW(R)\simeq R.$
\\ Part (b) is automatic. Let me show (c). For this, one computes $K=\ker(\chi).$ Let $a\in A$ and write $\delta(a)=(\delta(a)_{0},\delta(a)_{1},...)\in W(A).$ Then $a\in K,$ iff all $\delta(a)_{i}\in \mathfrak{a}.$ I claim that this is the case iff $\varphi^{n}(a)\in p^{n}\mathfrak{a}.$
\\
Let $a\in K.$ Then certainly $\delta(a)_{0}=a\in \mathfrak{a}.$ By the construction of the Cartier-morphism, one has that
\begin{equation}\label{Geistkomponenten von Cartiermorphismus}
W_{n}(\delta(a))=a^{p^{n}}+p(\delta(a)_{1}^{p^{n-1}})+...+p^{n}\delta(a)_{n}=\varphi^{n}(a).
\end{equation}
From (\ref{Geistkomponenten von Cartiermorphismus}), one deduces after division by $p^{n}$ that
\begin{equation}
\dot{\varphi}^{n}(a)=(p^{n}-1)!\gamma_{p^{n}}(a)+(p^{n-1}-1)!\gamma_{p^{n-1}}(\delta(a)_{1})+...+\delta(a)_{n}\in \mathfrak{a}.
\end{equation}
Here I used that all $\delta(a)_{i}\in \mathfrak{a}$ so that one can apply $\gamma$ and that $\mathfrak{a}$ is a pd-ideal, so that $\gamma_{m}(a)\in \mathfrak{a}.$ It follows that $\varphi^{n}(a)\in p^{n}\mathfrak{a}.$
\\
For the other inclusion, I argue by induction: one has that $\delta(a)_{0}=\varphi^{0}(a)=a\in\mathfrak{a}.$ Assume one has already shown for $i\leq n-1$ that $\delta(a)_{i}\in \mathfrak{a}.$ One may write, as $\varphi^{n}(a)\in p^{n}\mathfrak{a},$
$$\delta(a)_{n}=\frac{\varphi^{n}(a)-(\delta(a)_{0})^{p^{n}}-p(\delta(a)_{1})^{p^{n-1}}-...-p^{(n-1)}(\delta(a)_{n-1})^{p}}{p^{n}}
$$
$$
= \dot{\varphi}^{n}(a)-(p^{n}-1)!\gamma_{p^{n}}(\delta(a)_{0})-...-(p-1)!\gamma_{p}(\delta(a)_{n-1})\in \mathfrak{a}.
$$
From this it follows that
$\dot{\varphi}(K)\subseteq K:$ let $a\in K,$ one has that
$$\varphi^{n}(\dot{\varphi}(a))=\frac{\varphi^{n+1}(a)}{p}\in p^{n}\mathfrak{a},$$
because $\varphi^{n+1}(a)\in p^{n+1}\mathfrak{a}.$ Now the above description of the kernel shows that also $\dot{\varphi}(a)\in K.$
\\
Let me write again $K_{i}=p^{i}K/p^{i+1}K.$ Then, to conclude the proof as in Proposition \ref{Unique lifting lemma}, one still has to check that the $\varphi$-linear endomorphism
$$
\psi_{g}\otimes \dot{\varphi}\colon \Lie(U^{+})_{A} \otimes_{A} K_{i} \rightarrow  \Lie(U^{+})_{A}\otimes_{A} K_{i},
$$
is really pointwise nilpotent.
Recall that here $\psi_{g}$ was the $\varphi$-linear endomorphism of $\Lie(U^{+})_{A}$ given by $(1\otimes \varphi)\circ \pr_{2} \circ \Ad(g).$ But this follows from the adjoint nilpotency condition on $g\in \G(A)$. In fact, to make the argument clearer, let me choose a basis of $\Lie(U^{+}),$ i.e. an isomorphism $\Lie(U^{+})\simeq (W(k))^{n_{1}}.$ Then one has $(\psi_{g}\otimes \dot{\varphi})(x)=N\cdot \dot{\varphi}(\underline{x}),$ where $N\in\Mat_{n_{1}}(A)$ and $\underline{x}\in (K_{i})^{n_{1}}$ is the column-vector corresponding to $x$. Inductively, one gets
$$
(\psi_{g}\otimes \dot{\varphi})^{n}(x)=N\cdot \varphi(N)\cdot \varphi^{2}(N) ...\varphi^{n-1}(N)\cdot \dot{\varphi}^{n}(\underline{x}).
$$
The adjoint nilpotency condition says, that there exists some $c\geq 1,$ such that
$$
M=N\cdot \varphi(N)\cdot \varphi^{2}(N) ...\varphi^{c}(N)
$$
has coefficients in $\mathfrak{a}.$ Thus, one gets that $$(\psi_{g}\otimes \dot{\varphi})^{c+1}(x)=N\cdot \varphi(M)\cdot \dot{\varphi}^{c+1}(\underline{x})=N\cdot\dot{\varphi}(M\cdot \dot{\varphi}^{c}(\underline{x}))=N\cdot \dot{\varphi}(M)\cdot\varphi(\dot{\varphi}(\underline{x}))=0.$$
Here I used in the third equation, that $\dot{\varphi}$ is $\varphi$-linear, in the fourth equation, that $M$ has coefficients in $\mathfrak{a}$ and in the final equation, that $\varphi(K)\subseteq pK.$ This concludes the proof.
\end{proof}
Return to the assumption that $R$ is a an integral perfectoid $W(k)$-algebra. Recall that then $R/p$ is semi-perfect and $\Fil_{\cris}(R/p)=\ker(\Acris(R/p)\rightarrow R/p)$ is stabilized by the divided powers.\footnote{As this was critically used above, let me briefly recall the argument. Let $\Acris(R)^{\circ}$ be the pd-envelope of $\ker(\theta)$ over $(p\mathbb{Z}_{p},\mathbb{Z}_{p}),$ so that $\Acris(R)$ is the $p$-adic completion thereof. Let $x=\sum_{n}x_{n}\gamma_{n}(\xi)\in \Acris(R)^{\circ}.$ Then one has that
$$
\frac{x^{m}}{m!}=\sum_{i_{n},s.t. \sum i_{n}=m}\prod_{n}x_{n}\frac{(ni_{n})!}{(n!)^{i_{n}}(i_{n})!}\gamma_{n\cdot i_{n}}(\xi)\in \Acris(R)^{\circ}
$$
(one has to check that $\frac{(ni_{n})!}{(n!)^{i_{n}}(i_{n})!}$ is really a natural number). From this it follows that $\theta(\frac{x^{m}}{m!})=0.$ The case of $\Acris(R)$ now follows from continuity.
} Before I go on, let me give another proof of the fully faithfulness of the base-change along $\chi/p$ in this situation, which works without (!) the adjoint nilpotency condition:
\begin{Remark}
The claim here is that the operator $\dot{\varphi}$ operates topologically nilpotently on 
$$
\ker(\Acris(R/p)\rightarrow W(R/p)).
$$
In fact, this kernel will be generated by
$$
[x]^{(n)},
$$
for all $n\geq 1$ and $x\in \ker(R^{\flat}\rightarrow R/p).$ Then one has that
$$
\varphi^{m}([x]^{(n)})=\frac{(p^{m}n)!)}{n!}[x]^{(np^{m})}.
$$
It follows that $\dot{\varphi}=\varphi/p$ stabilizes $
\ker(\Acris(R/p)\rightarrow W(R/p))
$ and this action is topologically nilpotent; the Lemma \ref{Unique lifting lemma} of Lau above then implies the desired fully-faithfulness. Let me already point out here that the fact that the Frobenius acting on this kernel is very (!) topologically nilpotent will be crucially used later on again, when one compares $G$-quasi-isogenies of $\G$-$\mu$-displays over integral perfectoid rings and the corresponding crystalline $G$-quasi-isogeny of $\G$-Breuil-Kisin-Fargues modules of type $\mu,$ see section \ref{Section Translation of the Quasi-isogeny}.
\end{Remark}\label{Bemerkung zur kristallinen Aequivalenz modulo p}
Either by lemma \ref{Kristalline aequivalenz modulo p} above, or the remark I just made, one deduces that base-change along the strict frame-morphism
$$
\chi/p\colon \mathcal{F}_{\text{cris}}(R/p) \rightarrow \mathcal{W}(R/p)
$$
is fully faithful. It is however also essentially surjective in this context, because $\Acris(R/p)$ is henselian along the ideal $\ker(\Acris(R/p)\rightarrow W(R/p)).$ Indeed, it is enough to check this modulo $p,$ because $\Acris(R/p)$ is $p$-adic. But then $\ker(\Acris(R/p)\rightarrow W(R/p))/p$ is in fact a nil-ideal (to the exponent $p$), so that the pair $(\Acris(R)/p, \ker(\Acris(R/p)\rightarrow W(R/p))/p)$ is indeed henselian; using now that the group $\G$ is always assumed to be smooth over $\mathbb{Z}_{p},$ one may deduce the essential surjectivity by lifting a structure matrix.
\subsection{Proof of the crystalline equivalence}
To deduce that the morphism $\chi\colon \mathcal{F}_{\text{cris}}(R)\rightarrow \mathcal{W}(R)$ is nil-crystalline, one follows \cite[Prop. 9.7.]{LauPerfektoid}.
\\
One has a strict frame-morphism $j\colon \mathcal{F}_{\text{cris}}(R)\rightarrow \mathcal{F}_{\text{cris}}(R/p).$ Note that one gets an injective group homomorphism
$$
j\colon\G(\mathcal{F}_{\text{cris}}(R))_{\mu}\rightarrow \G(\mathcal{F}_{\text{cris}}(R/p))_{\mu}.
$$ Consider the category $\mathcal{C}$ with objects $\G(\Acris(R))\times \G(\mathcal{F}_{\text{cris}}(R/p))_{\mu}/ \G(\mathcal{F}_{\text{cris}}(R))_{\mu}$ and morphism given by $\G(\mathcal{F}_{\text{cris}}(R/p))_{\mu},$ with the action: $(g,\bar{x})\cdot h=(h^{-1}g\Phi_{\mathcal{F}_{\text{cris}}(R/p)}(h),\bar{hx}).$
\\
Then it is straightforward to see, that one gets an equivalence
$$
[\G(\Acris(R))/_{\Phi_{\mathcal{F}_{\text{cris}}(R)}} \G(\mathcal{F}_{\text{cris}}(R))_{\mu}]\simeq \mathcal{C},
$$
given by sending $g\mapsto (g,\bar{e}).$
Oberserve that under the natural maps
$$
\G(\mathcal{F}_{\text{cris}}(R/p))_{\mu}/ \G(\mathcal{F}_{\text{cris}}(R))_{\mu}\hookrightarrow \G(\Acris(R))/ \G(\mathcal{F}_{\text{cris}}(R))_{\mu} \rightarrow (\G/P^{-})(R),
$$
the coset $\G(\mathcal{F}_{\text{cris}}(R/p))_{\mu}/ \G(\mathcal{F}_{\text{cris}}(R))_{\mu}$ gets identified with the sections in the flag-variety, whose reductions modulo $p$ are the identity-coset. Thus lifts under $j$ correspond to lifts of the Hodge-Filtration. On the other hand, by section 3.5.7. in \cite{BueltelPappas}, one sees that lifts of banal and adjoint nilpotent $\G$-$\mu$-displays over $R/pR$ towards $R$ correspond to the coset-space $\G(W(pR))/\G(\mathcal{W}(pR))_{\mu},$ which is again identified with the sections in the flag-variety, whose reductions modulo $p$ are the identity-coset.
\\
Now consider the commutative diagram
$$
\xymatrix{
\G-\mu-\text{Win}(\mathcal{F}_{\text{cris}}(R))_{\nilp, \banal}\ar[r]^{\chi} \ar[d]^{j} & \G-\mu-\text{Displ}(R)_{\text{nilp}, \banal} \ar[d]^{\pi} \\
\G-\mu-\text{Win}(\mathcal{F}_{\text{cris}}(R/pR))_{\nilp, \banal} \ar[r]^{\chi/p} & \G-\mu-\text{Displ}(R/pR)_{\text{nilp}, \banal}.
}
$$
Then I just explained that lifts under $j$ and lifts under $\pi$ correspond to lifts of the Hodge-filtration in the same way. One deduces that this diagram of groupoids is cartesian - because $\chi/p$ is already known to induce an equivalence, it follows that $\chi$ induces an equivalence.
\section{Descent from $\Acris$ to $\Ainf$ }\label{Section Descent from Acris to Ainf}
In this section I will give a group theoretical generalization of the descent of classical windows over $\Acris$ to those over $\Ainf,$ as proven by Cais-Lau in Sections 2.2 and 2.3 of \cite{CaisLau}. Because I follow their methods, I have unfortunately to restrict to the case that $p\geq 3.$
\\
Let as usual $(\G,\mu)$ a 1-bound window datum, $R$ a $p$-torsion free integral perfectoid $W(k)$-algebra. Recall the $u=\frac{\varphi(\xi)}{p}$-frame morphism
$$
\lambda\colon \mathcal{F}_{\infintesimal}(R)\rightarrow \mathcal{F}_{\text{cris}}(R).
$$
Then I want to show the following
\begin{proposition}\label{Deszent von Acris nach Ainf}
The base-change functor
$$
\lambda_{\bullet}\colon \G-\mu-\Win(\mathcal{F}_{\infintesimal}(R))_{\banal}\rightarrow \G-\mu-\Win(\mathcal{F}_{\cris}(R))_{\banal}
$$
is crystalline for $p\geq 3.$
\end{proposition}
Consider the frames $$\mathcal{F}_{\infintesimal,n}(R)=\mathcal{F}_{\infintesimal}(R)\otimes \mathbb{Z}/p^{n}\mathbb{Z} = (\Ainf(R)/p^{n}\Ainf(R),(\xi)/p^{n}(\xi),R/p^{n}R,\varphi (\text{ mod }p^{n}), \dot{\varphi} (\text{ mod }p^{n}))$$ for $n\geq 1$ resp. the same with $\mathcal{F}_{\cris}(R).$ One thus obtains by assumptions that $\mathcal{F}_{\infintesimal}(R)=\lim_{n\geq 1}\mathcal{F}_{\infintesimal,n}(R)$ and $\mathcal{F}_{\cris}(R)=\lim_{n\geq 1} \mathcal{F}_{\cris,n}(R),$ in the sense of \cite[Section 2]{LauFrames}. Due to general problems with torsion, let me first check the following
\begin{Lemma}
Let $\mathcal{F}=(S,I,R,\varphi,\dot{\varphi})\in \lbrace \mathcal{F}_{\infintesimal}(R), \mathcal{F}_{\cris}(R)\rbrace$ and $\mathcal{F}_{n}=(S/p^{n}S,I/p^{n}I,R/p^{n}R,\varphi,\dot{\varphi})\in \lbrace \mathcal{F}_{\infintesimal,n}(R), \mathcal{F}_{\cris,n}(R)\rbrace.$ Then the divided Frobenius-map
$$
\Phi_{\mathcal{F}_{n}}\colon \G(\mathcal{F}_{n})_{\mu}\rightarrow \G(S_{n})
$$
is a group homomorphism.
\end{Lemma}
\begin{proof}
The proof rests on the following 
\\
\textbf{Claim}: The homomorphisms $G(S)\rightarrow \G(S/p^{n}S)$ and $\G(\mathcal{F})_{\mu}\rightarrow \G(\mathcal{F}_{n})_{\mu}$ are surjective.
\\
Let me explain this: one has that $(S,\ker(\pi_{n})=p^{n}S)$ is a henselian pair by \cite[Tag 0CT7]{stacks}. Then one can use that as $\G$ is smooth, the map
$$\pi_{n}\colon \G(S)\rightarrow \G(S/\ker(\pi_{n}))=\G(S/p^{n}S)$$ is surjective. This implies, as $P^{-}$ is also smooth, that likewise the map
$$\pi_{n}\colon P^{-}(S) \rightarrow P^{-}(S/p^{n}S)$$ is surjective. Now the map $I\rightarrow I/p^{n}I$ is surjective. Thus one can use the decomposition $$\G(\mathcal{F})_{\mu}\cong P^{-}(S) \times \Lie(U^{+})\otimes I$$ to conclude the proof of the \textbf{Claim}.
\\
Let me now prove the lemma. As $\pi_{n}\colon \mathcal{F}\rightarrow \mathcal{F}_{n}$ is a \textit{strict} frame-homomorphism, one has the following commutative diagram
$$
\xymatrix{
\G(\mathcal{F})_{\mu} \ar[r]^{\Phi_{\mathcal{F}}} \ar[d]^{\pi_{n}} & \G(S) \ar[d]^{\pi_{n}} \\
 \G(\mathcal{F}_{n})_{\mu} \ar[r]^{\Phi_{\mathcal{F}_{n}}} & \G(S/p^{n}S).
}
$$
Let $h_{n},h_{n}^{\prime}\in \G(\mathcal{F}_{n})_{\mu} ,$ one wants to show the equation
\begin{equation}\label{Gleichung geteilter Frobenius gruppenhomomorphismus modulo p}
\Phi_{\mathcal{F}_{n}}(h_{n}\cdot h_{n}^{\prime})=\Phi_{\mathcal{F}_{n}}(h_{n})\cdot \Phi_{\mathcal{F}_{n}}(h_{n}^{\prime}).
\end{equation}
Let $h,h^{\prime}\in \G(\mathcal{F})_{\mu}$ be pre-images of $h_{n},h_{n}^{\prime}$ under $\pi_{n}.$ Then one has the equation
\begin{equation}\label{Gleichung geteilter Frobenius gruppenhomomorphismus}
\Phi_{\mathcal{F}}(h\cdot h^{\prime})=\Phi_{\mathcal{F}}(h)\cdot \Phi_{\mathcal{F}}(h^{\prime}),
\end{equation}
by the assumption that $\Phi_{\mathcal{F}}$ is a group homomorphism. Now one can apply $\pi_{n}$ to both sides of equation (\ref{Gleichung geteilter Frobenius gruppenhomomorphismus}) and by the commutativity of the above diagram, it follows that equation (\ref{Gleichung geteilter Frobenius gruppenhomomorphismus modulo p}) holds true.
\end{proof}
Now I consider the induced $u_{n}$-frame homomorphisms
$$
\lambda_{n}\colon \mathcal{F}_{\infintesimal,n}(R)\rightarrow \mathcal{F}_{\cris,n}(R),
$$
for all $n\geq 1.$ Here $u_{n}$ is the reduction of $u=\frac{\varphi(\xi)}{p}\in \Acris(R)$ modulo $p^{n}.$
\\
The key statement is now the following:
\begin{Lemma}\label{Liften von kristallinen Morphismen modulo p}
For all $n\geq 1,$ if $\lambda_{n}$ is crystalline, also $\lambda_{n+1}$ is crystalline.
\end{Lemma}
For the proof, I will need a certain variation of the usual deformation theory arguments in an additive setting. Thus let me start with a little excursion on this topic.
\\
Let $\mathcal{F}$ and $\mathcal{F}^{\prime}$ be two frames over $W(k)$ fulfilling the following properties:
\begin{enumerate}
\item[(a):] The divided Frobenii $\Phi_{\mathcal{F}}$ and $\Phi_{\mathcal{F}^{\prime}}$ are group homomorphisms,
\item[(b):] one has a $u$-frame-morphism 
$$
\lambda\colon \mathcal{F}\rightarrow \mathcal{F}^{\prime},
$$
such that the corresponding ring-homomorphism $\lambda\colon S\rightarrow S^{\prime}$ is surjective and one has $R\simeq R^{\prime}.$
\item[(c):] Denoting by $K=\ker(S\rightarrow S^{\prime}),$ one requires that $\dot{\varphi}(K)\subseteq K$ and
$$\dot{\varphi}\colon K\rightarrow K$$
is elementwise (topologically) nilpotent.
\end{enumerate}
Let me consider the following quotient groupoid, which is a Lie-algebra version of the quotient groupoid of banal $\G$-$\mu$-windows: one denotes by $\mathfrak{g}=\Lie(\G)$ the Lie-algebra, a finite free $\Z_{p}$-module, $\mathfrak{u}^{+}=\Lie(U^{+})$ and $\mathfrak{p}^{-}=\Lie(P^{-}).$ Let $g\in \G(S)$ and $\lambda_{\bullet}(g)=\lambda(g)\mu(u)\in \G(S^{\prime}).$ Furthermore, I will consider the Lie-theoretic analog of the divided Frobenius map
$$
\psi_{\mathcal{F}} \colon \mathfrak{u}^{+}\otimes I \oplus \mathfrak{p}^{-} \otimes S \rightarrow \mathfrak{g}\otimes S,
$$
which is defined to be $\id\otimes \dot{\varphi}$ on $\mathfrak{u}^{+}\otimes I$ and $\zeta_{\mathcal{F}}^{m}\cdot \varphi$ on weight $m$-components. Consider
$$
\mathcal{C}_{g}=[\mathfrak{g}\otimes S/_{g} \mathfrak{u}^{+}\otimes I \oplus \mathfrak{p}^{-}\otimes S],
$$
where the action is by $X.Z=X-\Ad(g)(Z)+\psi_{\mathcal{F}}(Z).$ Likewise, one has
$$
\mathcal{C}_{\lambda_{\bullet}(g)}=[\mathfrak{g}\otimes S^{\prime}/_{\lambda_{\bullet}(g)} \mathfrak{u}^{+}\otimes I^{\prime} \oplus \mathfrak{p}^{-}\otimes S^{\prime}],
$$
where the action is now by $X^{\prime}.Z^{\prime}=X^{\prime}-\Ad(\lambda_{\bullet}(g))(Z^{\prime})+\psi_{\mathcal{F}^{\prime}}(Z^{\prime}).$ Observe that for all $Z\in \mathfrak{u}^{+}\otimes I \oplus \mathfrak{p}^{-}\otimes S,$ the following equation is true
\begin{equation}
\Ad(\mu(u))\lambda(\psi_{\mathcal{F}}(Z))=\psi_{\mathcal{F}^{\prime}}(\lambda(Z)).\footnote{ In fact, it suffices to check this on elements $Z=u^{+}\otimes i + p^{-}\otimes s,$ where $u^{+}\in \mathfrak{u}^{+}$ and $p^{-}\in \mathfrak{p}^{-}.$ By the definition of a $u$-framemorphism, one obtains that $\lambda(\id\otimes \dot{\varphi})(u^{+}\otimes i)=u^{-1}(\id \otimes \dot{\varphi}^{\prime})(u^{+}\otimes \lambda(i))$ and $\lambda(\zeta_{\mathcal{F}}^{m}\varphi(p^{-}\otimes s))=u^{+m}\zeta_{\mathcal{F}^{\prime}}^{m}\varphi^{\prime}(p^{-}\otimes \lambda(s)).$ Now use that $\Ad(\mu(u))$ will act by multiplication with $u$ on the first expression, while by division by $u^{m}$ on the second, to conclude.}
\end{equation}
This implies now that whenever an element $Z\in \mathfrak{u}^{+}\otimes I \oplus \mathfrak{p}^{-} \otimes S$ defines a morphism between $X,Y\in \mathfrak{g}\otimes S,$ then $\lambda(Z)$ will define a morphism between $\Ad(\mu(u))\lambda(X)$ and $\Ad(\mu(u))\lambda(Y).$ In total, one has constructed a functor
$$
\mathcal{C}_{g}\rightarrow \mathcal{C}_{\lambda_{\bullet}(g)}.
$$
\begin{Lemma}\label{Aequivalenz der Fasern unter Reduktion modulo p}
The previously introduced morphism of groupoids
$$
\mathcal{C}_{g}\rightarrow \mathcal{C}_{\lambda_{\bullet}(g)}
$$
is an equivalence.
\end{Lemma}
\begin{proof}
This is an additive version of the usual argument for a frame morphism being crystalline: first one notes that
$$\Ad(\mu(u))\circ (\id_{\mathfrak{g}}\otimes \lambda)\colon \mathfrak{g}\otimes S \rightarrow \mathfrak{g}\otimes S^{\prime}$$ is a surjective homomorphism of $S$-modules with kernel $\mathfrak{g}\otimes K.$ Furthermore, $$\lambda\colon \mathfrak{u}^{+}\otimes I \oplus \mathfrak{p}^{-}\otimes S\rightarrow \mathfrak{u}^{+}\otimes I^{\prime} \oplus \mathfrak{p}^{-}\otimes S^{\prime}$$ is surjective.\footnote{Here one uses the assumption that $R\simeq R^{\prime},$ to be able to deduce that $\lambda$ sends $I$ surjectively onto $I^{\prime}$ by the snake-lemma.} It follows that one has to see that the action of $\mathfrak{u}^{+}\otimes I \oplus \mathfrak{p}^{-}\otimes S$ on the fibers is simply transitive. Spelled out this means: it remains to be shown that for all $X\in \mathfrak{g}\otimes S$ and all $X_{0}\in \mathfrak{g}\otimes K,$ there exists a unique $Z\in \mathfrak{g}\otimes K,$ such that
\begin{equation}
X-\Ad(g)(Z)+\psi_{\mathcal{F}}(Z)=X+X_{0}.
\end{equation}
Note that as $\Ad\colon \G\rightarrow V(\Lie(\G))$ is a group scheme homomorphism, one has the following commutative diagram
$$
\xymatrix{
0 \ar[r] & \mathfrak{g}\otimes K \ar[r] \ar[d]^{\Ad(g)} & \mathfrak{g} \otimes S \ar[r]^{\id \otimes \lambda} \ar[d]^{\Ad(g)} & \mathfrak{g} \otimes S^{\prime} \ar[d]^{\Ad(\lambda(g))} \ar[r] & 0
\\
0 \ar[r] & \mathfrak{g}\otimes K \ar[r] & \mathfrak{g} \otimes S \ar[r]^{\id \otimes \lambda} & \mathfrak{g} \otimes S^{\prime} \ar[r] & 0.
}
$$
It follows that $\Ad(g)$ induces an automorphism of $\mathfrak{g}\otimes K.$ Furthermore, note that $\psi_{\mathcal{F}}$ stabilizes $\mathfrak{g}\otimes K,$ as $K$ is an ideal and $\dot{\varphi}$ does so by assumption.
One thus has to show that the map
$$
(\psi_{\mathcal{F}}-\Ad(g))\colon \mathfrak{g}\otimes K \rightarrow \mathfrak{g} \otimes K
$$
is a bijection. As was noted above that $\Ad(g)$ induces an automorphism of $\mathfrak{g}\otimes K,$ it suffices to show that $(U-\id)$ is a bijection, where $U=\Ad(g)^{-1}\circ \psi_{\mathcal{F}}.$ But observe that as $\varphi=\zeta_{\mathcal{F}}\cdot \dot{\varphi}$ on $K$, one my write
$$
\psi_{\mathcal{F}}=\begin{vmatrix}
E_{n_{1}}& 0 &0 \\
0 & \text{diag}_{n_{2}}(\zeta_{\mathcal{F}})& 0 \\
0 & 0 & \text{diag}_{n_{3}}(\zeta_{\mathcal{F}}^{2})
\end{vmatrix} \circ \dot{\varphi},
$$
where $n_{1}=\text{rg}(\mathfrak{u}^{+}),$ $n_{2}=\text{rg}(\mathfrak{p}^{0})$ and $n_{3}=\text{rg}(\mathfrak{p}^{-1}).$ But it was assumed that $\dot{\varphi}$ is pointwise (topologically) nilpotent on $K$. It follows that $U$ is (topologically) nilpotent and thus that $(U-1)$ is a bijection. This concludes the proof.
\end{proof}
Now let me explain why Lemma \ref{Liften von kristallinen Morphismen modulo p} above is true.
\begin{proof}
First one notes that one has strict frame morphisms
$
\pi_{n+1}\colon \mathcal{F}_{\infintesimal,n+1}(R)\rightarrow \mathcal{F}_{\infintesimal,n}(R),
$ resp. $\pi_{n+1}\colon \mathcal{F}_{\cris,n+1}(R)\rightarrow \mathcal{F}_{\cris,n}(R),$ for all $n\geq 1,$ induced by reduction modulo $p^{n}.$ Consider the following commutative diagram of groupoids
$$
\xymatrix{
\G-\mu-\Win(\mathcal{F}_{\infintesimal,n+1}(R))_{\banal} \ar[r]^{\lambda_{n+1}} \ar[d]^{\pi_{n+1}} & \G-\mu-\Win(\mathcal{F}_{\cris,n+1}(R))_{\banal} \ar[d]^{\pi_{n+1}} \\
\G-\mu-\Win(\mathcal{F}_{\infintesimal,n}(R))_{\banal} \ar[r]^{\lambda_{n}} & \G-\mu-\Win(\mathcal{F}_{\cris,n}(R))_{\banal}.
}
$$ By assumption the lower horizontal functor is an equivalence. The idea is to show that $\lambda_{n+1}$ induces equivalences on the fibers of the reduction modulo $p^{n}$-functors. Then the statement follows.
\\
Let $g_{n}\in \G(\Ainf(R)/p^{n}\Ainf(R)),$ considered as an representative for a banal $\G$-$\mu$-window over $\mathcal{F}_{\infintesimal,n}(R).$ Then I want to understand the fiber over this object of the above frame morphisms as a groupoid. Thus, denote by $\mathcal{L}(g_{n})$ the fiber of $\pi_{n+1,\bullet}$ over $g_{n}.$ This is the quotient groupoid with objects $$\lbrace g\in \G(\Ainf(R)/p^{n+1})\colon g \text{ mod }p^{n}=g_{n} \rbrace$$ and morphisms $$\lbrace h \in \G(\mathcal{F}_{\infintesimal,n+1}(R))_{\mu}\colon h\text{ mod }p^{n}=1 \rbrace$$ and the usual action by $\Phi$-conjugation. Choosing a lift $\tilde{g}_{n+1}\in \G(\Ainf(R)/p^{n+1})$ of $g_{n}$ (remember $\G$ is smooth), one can identify\footnote{The set $\lbrace g\in \G(\Ainf(R)/p^{n+1})\colon g \text{ mod }p^{n}=g_{n} \rbrace$ is a principal homogenous space under $\ker(\G(\Ainf(R)/p^{n+1})\rightarrow \G(\Ainf(R)/p^{n})).$ Sending $g$ towards $g^{\sharp}-e^{\sharp}$ ($e^{\sharp}$ is the unit section) induces an isomorphism $\ker(\G(\Ainf(R)/p^{n+1})\rightarrow \G(\Ainf(R)/p^{n}))\simeq \Hom_{\mathbb{Z}_{p}}(I_{\G}/I_{G}^{2},I),$ where $I_{G}\subseteq \mathcal{O}_{\G}$ is the ideal cutting out the unit section of $\G.$ Since $\G$ is smooth over $\mathbb{Z}_{p}$, one sees that $I_{\G}/I_{\G}^{2}$ is finite free as a $\mathbb{Z}_{p}$-module, so that one gets
$$
\Hom_{\mathbb{Z}_{p}}(I_{\G}/I_{G}^{2},I)\simeq \Hom_{\mathbb{Z}_{p}}(I_{\G}/I_{\G}^{2},\mathbb{Z}_{p})\otimes_{\mathbb{Z}_{p}}I,
$$
where I wrote $I:=p^{n}\Ainf/p^{n+1}\Ainf.$}
$$\lbrace g\in \G(\Ainf(R)/p^{n+1})\colon g \text{ mod }p^{n}=g_{n} \rbrace \simeq \mathfrak{g}\otimes (\Ainf(R)/p).$$
Here I used that $p^{n}\Ainf(R)/p^{n+1}\Ainf(R)\simeq \Ainf(R)/p\Ainf(R),$ by $p$-torsion freeness.
Furthermore, one has a bijection
$$
\lbrace h \in \G(\mathcal{F}_{\infintesimal,n+1}(R))_{\mu}\colon h\text{ mod }p^{n}=1 \rbrace \simeq \mathfrak{u}^{+} \otimes I_{\infintesimal,1} \oplus \mathfrak{p}^{-} \otimes \Ainf(R)/p.
$$
Indeed, recall that multiplication induces a bijection
$$
U^{+}(I_{\infintesimal,n+1})\times P^{-}(\Ainf(R)/p^{n+1})\rightarrow \G(\mathcal{F}_{\infintesimal,n+1}(R))_{\mu},
$$
then let $h\in \lbrace h \in \G(\mathcal{F}_{\infintesimal,n+1}(R))_{\mu}\colon h\text{ mod }p^{n}=1 \rbrace$ and write it as $u^{+}\cdot p^{-}=h,$ with $u^{+}\in U^{+}(I_{inf,n+1}),$ $p^{-}\in P^{-}(\Ainf(R)/p^{n+1}).$ Then one has to see that $u^{+}\in \ker(U^{+}(I_{inf,n+1})\rightarrow U^{+}(I_{inf,n}))$ and $p^{-}\in \ker(P^{-}(\Ainf(R)/p^{n+1})\rightarrow P^{-}(\Ainf(R)/p^{n})).$ But this follows because, modulo $p^{n},$ one then gets $u^{+}=(p^{-})^{-1}$ and since $U^{+}\cap P^{-}=\lbrace e \rbrace.$ Then one uses the isomorphisms
$$
\ker(U^{+}(I_{\infintesimal,n+1})\rightarrow U^{+}(I_{inf,n}))\simeq \mathfrak{u}^{+}\otimes_{W(k_{0})} p^{n}I_{inf}/p^{n+1}I_{inf}
$$
and
$$
\ker(P^{-}(\Ainf(R)/p^{n+1})\rightarrow P^{-}(\Ainf(R)/p^{n})) \simeq \mathfrak{p}^{-}\otimes_{W(k_{0})} p^{n}\Ainf(R)/p^{n+1}\Ainf(R)
$$
to conclude the verification of the above claim.
\\
Now let $g_{1}\in \G(\Ainf(R)/p)$ be the reduction modulo $p$ of $g_{n}.$ Under the above identifications, the action of the groupoid $\mathcal{L}(g_{n})$ gets identified with
$$
X.Z=X-\Ad(g_{1})(Z)+\psi_{\mathcal{F}_{\infintesimal,1}(R)}(Z).
$$
In the above notation, there is an equivalence of groupoids
$
\mathcal{L}(g_{n})\simeq \mathcal{C}_{g_{1}}.\footnote{In particular one observes that this fiber just depends on the reduction modulo $p$!}
$  I claim that there is an equivalence
$$
\lambda_{n+1}\colon \mathcal{L}(g_{n})\simeq \mathcal{L}(\lambda_{n,\bullet}(g_{n})).
$$
But this follows from the above lemma, \ref{Aequivalenz der Fasern unter Reduktion modulo p}, applied to the frame morphisms $\pi\colon\mathcal{F}_{\infintesimal,1}(R)\rightarrow \underline{A_{0}}$ and $\chi\colon\mathcal{F}_{\cris,1}(R)\rightarrow \underline{A_{0}},$ which I shall introduce in the proof of the next Lemma (where I will also check that these frame morphisms will satisfy the hypothesis I imposed in the lemma one is applying here).
\end{proof}
To finish the proof of Proposition \ref{Deszent von Acris nach Ainf}, it remains to show the following
\begin{Lemma}
Assume that $p\geq 3.$ Then the $u$-frame morphism
$$
\lambda_{1}\colon \G-\mu-\Win(\mathcal{F}_{\infintesimal,1}(R))_{\banal}\rightarrow \G-\mu-\Win(\mathcal{F}_{\cris,1}(R))_{\banal}
$$
is crystalline.
\end{Lemma}
This will be shown in two steps. Again, this is a group theoretic adaptation of the argument of Cais-Lau.
\subsubsection{$\pi$ is crystalline!}
Let me introduce the notation $A_{1}:=\Ainf(R)/p\Ainf(R)$ and $A_{0}:=A_{1}/\xi^{p}A_{1}.$
\\
Then there is the following frame-structure on $A_{0}$:
$$\underline{A_{0}}:=(A_{0},\xi A_{0}/\xi^{p}A_{0},R/p,\varphi=\text{Frob}_{A_{0}},\dot{\varphi}_{0}),$$
where $\dot{\varphi}_{0}(\xi a)=\varphi(a),$ for all $a\in A_{0}.$ Sometimes I will write $\text{Fil}(A_{0}):=\xi A_{0}/\xi^{p}A_{0}.$ Then it follows that the quotient homomorphism
$$A_{1}\rightarrow A_{0},$$
induces a strict frame-morphism
$$\pi\colon \mathcal{F}_{\infintesimal}(R) \otimes \Z /p\Z \rightarrow \underline{A_{0}}.$$
\begin{Lemma}
Assume that $p\geq 3,$ then $\pi$ is crystalline.
\end{Lemma}
\begin{proof}
I want to adapt the Unique Lifting Lemma, Proposition \ref{Unique lifting lemma}, to the situation at hand.
\\
Consider
$$K:=\ker(\pi)=\xi^{p}A_{0}.$$
Then it is true that $\dot{\varphi}(\xi^{p}a))=\varphi(\xi)^{p-1}\varphi(a)=\xi^{p(p-1)}a^{p},$ and therefore $\dot{\varphi}$ stabilizes $K.$ Furthermore, one sees that for $p\geq 3$ the restriction
$$\dot{\varphi}\colon K\rightarrow K,$$
is topologically nilpotent.
\\
Recall that in Proposition \ref{Unique lifting lemma} I worked with frames such that the frame-constant is equal to $p.$ Here one is faced with frames, whose frame-constant is $\xi^{p}$. One has to modify the arguments.
\\
First, $\Ainf(R)$ is $(p,\xi)$-adically complete, and thus $A_{1}$ is $\xi$-adic and also complete in the $\xi^{p}$-adic topology (since $(\xi^{p})\subseteq (\xi)$) and $K$ is open, thus closed in this topology. In Lemma \ref{Vorbereitung unique lifting lemma} the only place, where I used that the frame-constant is $p,$ was to show that $K/pK$ is a nilideal. Recall that this was true, because $\varphi$ is a Frobenius-lift. In the case at hand it follows directly that $K/\xi^{p}K$ is a nil-ideal. In total, one also has Lemma \ref{Vorbereitung unique lifting lemma} in the present set-up.
\\
Then in the proof of Proposition \ref{Unique lifting lemma} the same arguments work, using in the end that $\dot{\varphi}$ is pointwise topologically nilpotent on $K.$
\end{proof}
\subsubsection{$\chi$ is crystalline!}
Now I want to connect the frame $\underline{A_{0}}$ from before with the frame $$\mathcal{F}_{\cris}(R)\otimes \mathbb{Z} /p \mathbb{Z}.$$ To shorten notation, in the following one write $A_{1}^{\cris}:=\Acris(R)/p\Acris(R).$
\begin{Lemma}\label{Framemorphismus von Acris mod p zu A0} Assume $p\geq 3.$
There exists a $u^{-1}$-frame homomorphism
$$\chi\colon \mathcal{F}_{\cris}(R)\otimes \Z /p\Z \rightarrow \underline{A_{0}}.$$
\end{Lemma}
\begin{proof}
One can give the ideal $\text{Fil}(A_{0})$ the trivial pd-structure by setting $\gamma_{p}(\xi)=0.$ The universal property of $A_{1}^{\cris}$  gives a unique pd-homomorphism
$$\chi\colon A_{1}^{\cris}\rightarrow A_{0},$$
extending the map $\pi\colon A_{1}\rightarrow A_{0}$ and sending $\Fil^{p}(A_{1}^{\cris})$ to zero. Then the identity of $A_{0}$ factors as
$$A_{0}=A_{1}/\xi^{p}A_{1}\rightarrow A_{1}^{\cris}/\Fil^{p}(A_{1}^{\cris}) \rightarrow A_{0}.$$
Here the first map is surjective and therefore both homomorphisms have to be bijective.
One has the following identity in $\Acris(R):$
\begin{equation}
\dot{\varphi}_{\cris}(\frac{\xi^{n}}{n!})=u^{n}\frac{p^{n-1}}{n!}=x\cdot p^{n-1-\nu_{p}(n!)},
\end{equation}
where still $u=\frac{\varphi(\xi)}{p}$ and $x$ is some unit. From the estimate $\nu_{p}(n!)\leq \frac{n-1}{p-1}$ one deduces for $p\geq 3$ that $n-1-\nu_{p}(n!)$ is positive and thus
$$\dot{\varphi}_{\cris}\colon \Fil^{p}(A_{1}^{\cris})\rightarrow A_{1}^{\cris}$$ is the zero map. Therefore
$$\dot{\varphi}_{\cris}\colon \Fil(A_{1}^{\cris})\rightarrow A_{1}^{\cris}$$ factorizes over $\Fil(A_{0})\rightarrow A_{1}^{\cris}$ and induces a map
$$\overline{\dot{\varphi}}\colon \Fil(A_{0})\rightarrow A_{0}.$$
Because $\lambda$ is a $u$-framehomomorphism, it follows that $\overline{\dot{\varphi}}=u\cdot \dot{\varphi}_{0}.$ It follows that $\chi$ is a $u^{-1}$-frame-homomorphism, as desired.
\end{proof}
\begin{Lemma}\label{chi ist kristalline!}
Assume $p\geq 3.$ Then the $u^{-1}$-frame homomorphism from Lemma \ref{Framemorphismus von Acris mod p zu A0} before is crystalline.
\end{Lemma}
\begin{proof}
Before checking that the $u^{-1}$-frame morphism is crystalline, let me explain, why it induces indeed a functor of the associated banal $\G$-$\mu$-window categories.
Explicitely, this means that one has to show that for $g,g^{\prime}\in \G(A_{1}^{\cris}),$ such that there is a $h\in\G(\mathcal{F}_{\cris,1}(R))_{\mu}$ with
\begin{equation}
g^{\prime}=h^{-1}g\Phi_{\cris,1}(h),
\end{equation}
there exists a $z\in \G(\underline{A_{0}})_{\mu},$ such that
\begin{equation}\label{Gleichung fuer chi is funktoriell 3}
\chi(g^{\prime})\mu(u^{-1})=z^{-1}\chi(g)\mu(u^{-1})\Phi_{\underline{A_{0}}}(z).
\end{equation}
For this, write as usual $h=v\cdot q,$ where $v\in U^{+}(\Fil(A_{1}^{\cris}))$ and $q\in P^{-}(A_{1}^{\cris}).$ Then $$z=\chi(v)\cdot \chi(q)\in\G(\underline{A_{0}})_{\mu}$$ will be the desired element. In fact, by definition one has that $\Phi_{\underline{A_{0}}}(\chi(q))=\mu(0)\varphi(\chi(q))\mu(0)^{-1},$ as the frame constant of $\underline{A_{0}}$ is zero. This implies that $\Phi_{\underline{A_{0}}}(\chi(q))$ lies in the centralisator of $\mu$\footnote{In fact, one has that $P^{-}\cap P^{+}$ is exactly the $\mu$-centralizer, \cite[Lemma 2.1.5.]{CGP}and by definition $\mu(0)\varphi(\chi(q))\mu(0)^{-1}$ has no parts of weight $<0,$ as one killed those parts via multiplication with $0.$} and thus
\begin{equation}\label{Kommutiert mit Cocharacter}
\Phi_{\underline{A_{0}}}(\chi(q))\mu(u^{-1})=\mu(u^{-1})\Phi_{\underline{A_{0}}}(\chi(q)).
\end{equation}
Furthermore, one checks directly the equations
\begin{equation}\label{Gleichung fuer chi ist funktoriell 1}
\mu(u^{-1})\Phi_{\underline{A_{0}}}(\chi(v))\mu(u)=\chi(\Phi_{\mathcal{F}_{\cris,1}}(v))
\end{equation}
and
\begin{equation}\label{Gleichung fuer chi is funktoriell 2}
\Phi_{\underline{A_{0}}}(\chi(q))=\chi(\Phi_{\mathcal{F}_{\cris,1}}(q)).
\end{equation}
Now equations (\ref{Kommutiert mit Cocharacter})-(\ref{Gleichung fuer chi is funktoriell 2}) together readly imply  equation (\ref{Gleichung fuer chi is funktoriell 3}) for $z=\chi(v)\cdot \chi(q)$ as above.
\\
Now one can turn to checking that $\chi$ is indeed crystalline, where one already knows the game: one has to modify the Unique Lifting Lemma, Proposition \ref{Unique lifting lemma}, to apply it in situation here.
\\
Consider $K:=\ker(A_{1}^{\cris}\rightarrow A_{0})=\Fil^{p}(A_{1}^{\cris}).$ This ideal has pd-structure and because $p=0,$ it follows that it is a nil-ideal (to the exponent $p$). Since smooth morphisms, locally of finite presentation satisfy Grothendieck$^{\prime}$s lifting criterium for nilideals and because $K\subseteq \Fil(A_{1}^{\cris})\subseteq \Rad(A_{1}^{\cris}),$ the same arguments as in the omitted proof of Lemma \ref{Vorbereitung unique lifting lemma} apply. I make the following
\\
\textbf{Claim:} The divided Frobenius
$$\Phi_{\mathcal{F}_{\cris}(R)\otimes \Z/p\Z}\colon \G(K) \rightarrow \G(K)$$ is the trivial map.
\\
If one has verified this statement, the proof of the present Lemma follows exactly as in Proposition \ref{Unique lifting lemma}.
\\
Let me prove the claim. For 
$$\Phi_{\mathcal{F}_{\cris}(R)\otimes \Z/p\Z}=\id \otimes \dot{\varphi}\colon \Lie(U^{+})\otimes K \rightarrow \Lie(U^{+}) \otimes K$$ one already knows that it is true, as one computed above that $\dot{\varphi}$ is zero on $K.$
\\
For
$$\Phi_{\mathcal{F}_{\cris}(R)\otimes \Z/p\Z}\colon P^{-}(K)\rightarrow P^{-}(K)$$
the usual argument applies: if $g\in P^{-}(K)$ and $f=\sum_{n\in \Z} f_{n},$ one has that $g^{\sharp}(f_{-n})\in K$ for $n\geq 1.$ As the $\zeta_{\mathcal{F}_{\cris}(R)\otimes \Z/p\Z}=p=0,$ $\Phi_{\mathcal{F}_{\cris}(R)\otimes \Z/p\Z}(g)^{\sharp}(f_{-n})=0.$ Writing $g^{\sharp}(f_{0})=e^{\sharp}_{A_{1}^{\cris}}(f_{0})+k$ for $k\in K,$ one has that $k^{p}=0$ and the Lemma follows, because Frobenius maps the unit section to the unit section.
\end{proof}
\subsection{Extension to the general case}\label{Section extension to the general case}
Now I sketch how to deduce the descent result, Proposition \ref{Deszent von Acris nach Ainf}, for an arbitrary integral perfectoid $W(k)$-algebra, that is not necessarily $p$-torsion free\footnote{As the category of integral perfectoids is closed under products, one can certainly construct integral perfectoids with $p$-torsion, that are not perfect.}. This follows by the same arguments as used by Lau to prove the $\GL_{n}$-case, see \cite[Lemma 9.2., Prop. 9.3 ]{LauPerfektoid}. Therefore I just briefly recall these.
\\
Let $R$ be an integral perfectoid $W(k)$-algebra, let $R_{1}=R/R[\sqrt{pR}],$ $R_{2}=R/\sqrt{pR}$ and $R_{3}=R_{1}/\sqrt{pR_{1}}.$ Then $R_{1}$ is $p$-torsion free integral perfectoid, while $R_{2}$ and $R_{3}$ are perfect. One then has a crucial pullback and pushout square with the canonical maps
$$
\xymatrix{
R \ar[r] \ar[d] & R_{1} \ar[d] \\
R_{2} \ar[r] & R_{3}.
}
$$
It follows that one has a  pushout diagram in the category of frames, presenting $\mathcal{F}_{\infintesimal}(R)$ (resp. $\mathcal{F}_{\cris}(R)$) as the pushout of $\mathcal{F}_{inf}(R_{1})$ and $\mathcal{F}_{\infintesimal}(R_{2})$ over $\mathcal{F}_{\infintesimal}(R_{3})$ (resp. the same for $\mathcal{F}_{\cris}(.)$).
Thus there is also a pushout diagram of groupoids
$$
\xymatrix{
\G-\mu-\text{Win}(\mathcal{F}_{\infintesimal}(R))_{\banal} \ar[r] \ar[d] & \G-\mu-\text{Win}(\mathcal{F}_{\infintesimal}(R_{1}))_{\banal} \ar[d] \\
\G-\mu-\text{Win}(\mathcal{F}_{\infintesimal}(R_{2}))_{\banal} \ar[r] & \G-\mu-\text{Win}(\mathcal{F}_{\infintesimal}(R_{3}))_{\banal},
}
$$
(resp. the same for banal $\G$-$\mu$-windows for $\mathcal{F}_{\cris}(.)$). The point is now that in the descent result, one may reduce to the case that the integral perfectoid $W(k)$-algebra is either a perfect $\mathbb{F}_{p}$-algebra, where it is trivial as the frame morphism $\lambda\colon \mathcal{F}_{\infintesimal}(R)\rightarrow \mathcal{F}_{\cris}(R)$ is just the identity, or $p$-torsion free, which was addressed before.
\section{Connection to local mixed-characteristic shtukas à la Scholze}\label{Section Connection to local mixed-characteristic shtukas}
Fix a reductive and minuscule window datum $(\G,\mu).$ I will now explain, how to go from adjoint nilpotent $\G$-$\mu$-displays to local mixed characteristic $\G$-shtukas. As the real work has been done in the previous sections, I will be rather quick here.
\\
Let me begin with the obvious extension of the notion of a Breuil-Kisin-Fargues module, as in \cite{FarguesQuelquesresultats}, \cite{LauPerfektoid}, that will be used.
\begin{Definition}\label{Definition BKF}
Let $R$ be an integral perfectoid ring with fixed perfectoid pseudo-uniformizer $\varpi\in R.$
\begin{enumerate}
\item[(a):]
 Breuil-Kisin-Fargues module with $\G$-structure over $R$ is a pair $(\mathcal{P},\varphi_{\mathcal{P}}),$ where $\mathcal{P}\rightarrow \Spec(\Ainf(R))$ is a $\G$-torsor and 
$$
\varphi_{\mathcal{P}}\colon (\varphi^{*}\mathcal{P})[1/\varphi(\xi)]\simeq \mathcal{P}[1/\varphi(\xi)],
$$
for some generator $\xi$ of $\ker(\theta)$ and $\mathcal{P}[1/\varphi(\xi)]=\mathcal{P}\times_{\Spec(\Ainf(R))} \Spec(\Ainf(R)[1/\varphi(\xi)]).$
\item[(b):] One says that a $\G$-Breuil-Kisin-Fargues module $(\mathcal{P},\varphi_{\mathcal{P}})$ is of type $\mu,$ if for all maps $R\rightarrow V,$ where $V$ is a $\varpi$-complete rank $1$ valuation ring with algebraically closed fraction field and any representative $U\in \G(\Ainf(V)[1/\varphi(\xi)])$ for the base-change of $(\mathcal{P},\varphi_{\mathcal{P}})$ to $\Spec(\Ainf(V))$\footnote{Here one uses that $V$ is strictly henselian, as an absolutely integrally closed local domain and that then $\Ainf(V)$ will also be strictly henselian, so that any $\G$-torsor over $\Spec(\Ainf(V))$ will be trivial.} one has that
$$
U\in \G(\Ainf(V))\mu(\varphi(\xi))\G(\Ainf(V)).
$$
\end{enumerate}
\end{Definition}
\begin{Remark}
The condition in (b) is phrased in a way that it exactly gives that the $\G$-shtuka one constructs from $(\mathcal{P},\varphi_{\mathcal{P}})$ will have singularities bounded by the cocharacter $\mu.$ It may be checked for one representative (since changing by $\varphi$-conjugation with elements in $\G(\Ainf(V))$ does not change this containement condition).
\end{Remark}
I will write $\G$-$\BKF(R)$ for the corresponding groupoid and  $\G$-$\BKF(\mathcal{R})$ for the fibered category for the affine etale topology on $\Spec(R/p),$ constructed via Lau$^{\prime}$s sheaf. Note that one defined the fibered category of $\G$-$\mu$-windows for the frame $\mathcal{F}_{\infintesimal}(\mathcal{R})$ in such a way that it forms a stacks.
\begin{Lemma}\label{Displays sind ein Stack fuer Laus Garbe}
\begin{enumerate}
\item[(a):] $\G$-$\BKF(\mathcal{R})$ is a stack.
\item[(b):] $\G$-$\mu$-$\Displ(\mathcal{R})$ is a stack.
\end{enumerate}
\end{Lemma} 
\begin{proof}
The reader is referred to \cite[Lemma 10.9.]{LauPerfektoid}, where the corresponding statements for $\G=\GL_{n}$ and minuscule BKFs resp. $p$-divisible groups are proven. Using \cite[Lemma B.0.4.(b)]{BueltelPappas} for $\G$-BKFs and GFGA, Lemma \ref{Displays ueber adischen Ringen}, for $\G$-$\mu$-displays, the same arguments apply.
\end{proof}
\begin{Corollary}\label{Korollar Vergleich G-mu-Windows und G-mu-BKF}
Let $R$ be an integral perfectoid $W(k)$-algebra. Then there exists a fully faithful functor
$
\G\text{-}\mu\text{-}\Win(\mathcal{F}_{\infintesimal}(R))\rightarrow \G\text{-}\BKF(R)
$
with essential image exactly those $\G$-Breuil-Kisin-Fargues modules over $R,$ which are of type $\mu.$
\end{Corollary}
\begin{proof}
At first one has to construct the functor: let $(Q,\alpha)$ be a $\G$-$\mu$-window for the frame $\mathcal{F}_{\infintesimal}(R).$ Let $B=R/p$ and $f\colon B\rightarrow B^{\prime}$ an étale covering, which trivializes $(Q,\alpha).$ Let $R^{\prime}=\mathcal{R}(B^{\prime}).$ Then Lemma \ref{Windows ueber Ainf und BKF lokales Statement} produces from $f^{*}(Q,\alpha)$ a uniquely determined (trivial) Breuil-Kisin-Fargues module over $R^{\prime}$ (which is automatically of type $\mu$). Using uniqueness and the fact that the fibered category $\G$-BKF($\mathcal{R}$) forms a stack on $\Spec(B)^{\text{aff}}_{\text{ét}},$ one may build a uniquely determined $\G$-BKF over $R.$ This association is functorial and one has therefore constructed the desired functor.
\\
The fully faithfullness part follows directly from the local statement in Lemma, \ref{Windows ueber Ainf und BKF lokales Statement}, and the fact that $\G$-$\BKF(\mathcal{R})$ forms a stack.
\\
For the claim about the essential image, I will use $\varpi$-complete arc-descent: for this, fix an integral perfectoid $W(k)$-algebra $R$ as above and also a perfectoid pseudo-uniformizer $\varpi\in R.$ Then one considers the category of integral perfectoid $R$-algebra and equips this category with the $\varpi$-complete arc-topology: covers are generated by morphism of integral perfectoid $R$-algebras $A\rightarrow B,$ such that for all $\varpi$-complete valuation rings $V$ of rank $\leq 1,$ and any map $A\rightarrow V,$ there exists a faithfully flat extension of $\varpi$-complete valuation rings of rank $\leq 1,$ (=just an inclusion) $V\rightarrow V^{\prime}$ and a morphism $B\rightarrow V^{\prime}$ making the following diagram commutative:
$$
\xymatrix{
A \ar[r] \ar[d] & B \ar[d] \\
V \ar[r] & V^{\prime}
}.
$$
Then, by \cite[Lemma 4.2.6.]{CesnaviciusScholzePurity}, the pre-sheaf $R\rightarrow R^{\prime}\mapsto \Ainf(R^{\prime})$ satisfies $\varpi$-complete arc-descent. It follows that the pre-sheaves $R\rightarrow R^{\prime}\mapsto \G(\Ainf(R^{\prime}))$ resp. $R\rightarrow R^{\prime}\mapsto \G(\mathcal{F}_{\infintesimal}(R^{\prime}))_{\mu}$ satisfy $\varpi$-complete arc-descent, by using the same tricks as I used to check étale descent as in Lemma \ref{Garbeneigenschaft Window Gruppe Inf-Frame}. From this one obtains $\varpi$-complete arc-descent for the fibered category of $\G$-$\mu$-windows for the frame $\mathcal{F}_{\infintesimal}(.)$. One can find in the literature the result that also the fibered category sending $R\rightarrow R^{\prime}\mapsto \G-\text{Tors}(\Ainf(R^{\prime}))$ satisfies $\varpi$-complete arc-descent: indeed, here one can use very recent results due to Ito \cite[Corollary 4.2.]{Ito} that imply $\varpi$-complete arc-descent for finite projective modules and the case of $\G$-torsors then follows by making use of the Tannakian formalism. Since the condition to be of type $\mu$ was formulated in a way which is local for the $\varpi$-complete arc-topology, also the fibered category $(R\rightarrow R^{\prime})\mapsto \G\text{-}\mu\text{-}\BKF(R^{\prime})$ satisfies $\varpi$-complete arc-descent. It follows that one may check arc-locally on $R$ that the functor
$$
\G\text{-}\mu\text{-}\Win(\mathcal{F}_{\infintesimal}(R))\rightarrow \G\text{-}\mu\text{-}\BKF(R)
$$
is an equivalence. This allows one to restrict to a situation, where $R^{\prime}=\prod_{i\in I}V_{i}$ with all $V_{i}$ $\varpi_{i}$-adically complete and separated valuation rings of rank $\leq 1$ and which have algebraically closed fraction field. In this situation, any $\G$-torsor over $\Spec(\Ainf(R^{\prime}))$ is trivial and the condition to be of type $\mu$ becomes that for any structure matrix $U\in \G(\Ainf(R^{\prime})[1/\varphi(\xi)]),$ representing this $\G$-Breuil-Kisin-Fargues module over $R^{\prime},$ the image of $U$ under the morphism
$$
\G(\Ainf(R^{\prime})[1/\varphi(\xi)])\rightarrow \G(\Ainf(V_{i})[1/\varphi(\xi_{i})]),
$$
can be written as
$$
U_{i}=g_{i}\mu(\varphi(\xi_{i}))h_{i},
$$
where $g_{i},h_{i}\in \G(\Ainf(V_{i})).$ Taking the product, one obtains element $g,h\in \G(\Ainf(R^{\prime})),$ such that under all morphisms $\G(\Ainf(R^{\prime})[1/\varphi(\xi)])\rightarrow \G(\Ainf(V_{i})[1/\varphi(\xi_{i})]),$ $g\mu(\varphi(\xi))h$ is send towards $U_{i}.$ Using now that the map $\Ainf(R)[1/\varphi(\xi)]\rightarrow \prod_{i\in I}\Ainf(V_{i})[1/\varphi(\xi_{i})]$ is injective, one deduces that
$$
U=g\mu(\varphi(\xi))h.
$$
This allows one to use lemma \ref{Windows ueber Ainf und BKF lokales Statement} to construct the required $\G$-$\mu$-window over $R^{\prime},$ proving the essential surjectivity.
\end{proof}
As a consequence of the previously proven equivalence between banal adjoint nilpotent $\G$-$\mu$ displays over integral perfectoid $W(k)$-algebras and banal adjoint nilpotent $\G$-$\mu$ windows for the frame $\mathcal{F}_{\cris}(R),$ plus the descent result along $\lambda\colon \mathcal{F}_{\infintesimal}(R)\rightarrow \mathcal{F}_{\cris}(R)$ of the last section, one deduces
\begin{Corollary}\label{Korollar Vergleich Displays und BKF}
Let $R$ still be an integral perfectoid $W(k)$-algebra and assume $p\geq 3,$ there exists a fully faithful functor
$\G-\mu^{\sigma}-\Displ(R)_{\nilp}\rightarrow \G-\mu-\BKF(R).$
\end{Corollary}
\begin{Remark}
Although the frame $\mathcal{F}_{\infintesimal}(R)$ depends on a choice of $\xi,$ the above functor does not.
\end{Remark}
Let $S=\Spa(R,R^{+})$ be an affinoid perfectoid space over $\Spa(k)$ and $S^{\sharp}=\Spa(R^{\sharp},R^{\sharp,+})$ an untilt over $\Spa(W(k)).$ Then $R^{\sharp,+}$ is an integral perfectoid $W(k)$-algebra and by restricting from $\Spec(\Ainf(R^{+}))$ to $\Spa(W(R^{+}),W(R^{+}))-\lbrace [\varpi]=0 \rbrace =: S\dot{\times} \Spa(\mathbb{Z}_{p})$ (with the $(p,[\varpi])$-adic topology, where $\varpi\in R$ is of course some pseudo-uniformizer), one is finally able to produce for $p\geq 3$ a local mixed-characteristic $\G$-shtuka over $S$ with a leg at $S^{\sharp}\subset S \dot{\times} \Spa(\mathbb{Z}_{p}),$ as introduced in \cite{Berkeleylectures}.
\subsection{Classification of adjoint nilpotent $\G$-$\mu$-displays over $\mathcal{O}_{C}$}
I will now state a description of the essential image of the previous functor from adjoint nilpotent $\G$-$\mu$-displays to local mixed-characteristic $\G$-shtukas over a geometric perfectoid point in terms of the schematical Fargues-Fontaine curve \cite{FarguesFontaine} - in the spirit of Scholze-Weinstein. Using GAGA one could also work with the adic Fargues-Fontaine curve. Since most of this material is explained sufficiently in the $\GL_{n}$-case in Scholze$^{\prime}$s Berkeley notes \cite{Berkeleylectures}, and the case of general reductive $\G/\Spec(\mathbb{Z}_{p})$ requires only few new ideas, I decided to not give full proofs here.
\\
Let $C$ be a complete non-archimedean, algebraically closed field extension of $W(k)[1/p].$ Then its ring of integers $\mathcal{O}_{C}$ is a strictly henselian and integral perfectoid $W(k)$-algebra. I will abbreviate $\Ainf=\Ainf(\mathcal{O}_{C})$ and $\mathcal{F}_{\infintesimal}=\mathcal{F}_{\infintesimal}(\mathcal{O}_{C}).$ Note that $\Ainf$ is again a strictly henselian $W(k)$-algebra. Let $X=X^{\text{alg}}_{\mathbb{Q}_{p},C^{\flat}}$ be the schematical Fargues-Fontaine curve over $\mathbb{Q}_{p}$ associated to $C^{\flat}.$ The chosen untilt $C$ of $C^{\flat}$ gives a closed point $\infty\in |X|,$ with residue field $C$ and $\widehat{\mathcal{O}_{X,\infty}}\cong B^{+}_{dR}(C)=B^{+}_{dR}.$
\\
The reason why some crucial arguments go through via the Tannakian formalism is the following handy 
\begin{Lemma}\label{Fargues Lemma fuer Tannaka Formalismus}
Let $K$ be a field of characteristic $0,$ $H$ a reductive group over $K,$ $R$ a $K$-algebra and denote by $\Proj_{R}$ the category of finite projective $R$-modules with the usual additive tensor structure. Then a functor
$$
\omega\colon \Rep_{K}(H)\rightarrow \Proj_{R}
$$
is a fiber functor if and only if it is an additive tensor-functor.
\end{Lemma}
For a proof, see Lemma 5.4. in \cite{FarguesGtorseurs}. Since Fargues-Fontaine showed that the categories of trivial vector bundles on the curve and finite dimensional $\mathbb{Q}_{p}$-vector spaces are equivalent as tensor-categories (compatible with direct sums), one deduces from the previous Lemma that the groupoid of trivial $G$-torsors on $X$ and $G$-torsors on $\Spec(\mathbb{Q}_{p})$ are equivalent.
\\ 
Let $G=\G_{\mathbb{Q}_{p}}$ be the associated generic fiber, an unramified reductive group. Consider a $G$-torsor $\mathcal{E}\rightarrow X.$ Under a modification of $\mathcal{E}\rightarrow X$ at the chosen point $\infty,$ one understands the datum of a $G$-torsor $\mathcal{E}^{\prime}\rightarrow X,$ together with an isomorphism
$$
\iota\colon \mathcal{E}^{\prime}\!\mid_{X-\infty}\simeq \mathcal{E}\!\mid_{X-\infty}.
$$
For a (Galois-orbit of a) co-character $\chi\in X_{*}(T)^{+}/\Gamma,$ let me recall, what it means that a modification $(\mathcal{E}^{\prime},\iota)$ is of type $\chi:$ the given isomorphism $\iota$ corresponds to a double coset, that one gets after pullback of $\mathcal{E}$ and $\mathcal{E}^{\prime}$ along $\Spec(B_{dR}^{+})\rightarrow X$ and choosing respective trivializations, and forgetting about them again, in
$$
G(B_{dR}^{+})\backslash G(B_{dR})/G(B_{dR}^{+}).
$$
By the Cartan-decomposition, this double quotient is in bijection with $X_{*}(T)^{+}/\Gamma.$ Then one requires that $\iota$ corresponds to $\chi\in X_{*}(T)^{+}/\Gamma.$
\\
Now one can state the following theorem, that in the $\GL_{n}$-case is due to Fargues.
\begin{Theorem}
The following categories are equivalent
\begin{enumerate}
\item[(a):] $\G$-$\mu^{\sigma}$ windows over $\mathcal{F}_{\infintesimal},$ and
\item[(b):] tuples $(\mathcal{E}_{1},\mathcal{E}^{\prime},\iota,\mathcal{T}),$ where $\mathcal{E}_{1}$ is a trivial $G$-torsor on $X,$ $(\mathcal{E}^{\prime},\iota)$ is a modification of type $\mu^{\sigma}$\footnote{This is the same as being of type $\mu$.} at $\infty$ of $\mathcal{E}_{1},$ $\mathcal{T}$ is a $\G$-torsor on $\Spec(\mathbb{Z}_{p}),$ such that if $\mathcal{V}$ is the $G$-torsor on $\Spec(\mathbb{Q}_{p})$ corresponding to $\mathcal{E}_{1},$ one requires that
$$
\mathcal{T}\times_{\Spec(\mathbb{Z}_{p})}\Spec(\mathbb{Q}_{p})\cong \mathcal{V}.
$$
\end{enumerate}
\end{Theorem}

\begin{Remark}\label{Bemerkung zum zweiten Buendel in den Modifikationen}
\begin{enumerate}
\item[(i):] This is a mild generalization of \cite[Theorem 14.1.1.]{Berkeleylectures}.
\item[(ii):] Recall that one has an equivalence between vector bundles on $X$ and $\varphi$-modules on $\Bcris(\mathcal{O}_{C}/p)$ (as $C$ is algebraically closed). Using Lemma \ref{Fargues Lemma fuer Tannaka Formalismus} above, one deduces that $G$-torsors on $X$ and $G$-torsors with $\varphi$-structure on $\Bcris(\mathcal{O}_{C}/p)$ are equivalent. Under this equivalence, the $G$-bundle $\mathcal{E}$ in the theorem corresponds to the base-change of the window over $\mathcal{F}_{\infintesimal}$ to $\Bcris(\mathcal{O}_{C}/p).$
\end{enumerate}
\end{Remark}
\begin{proof}
Let me explain, where one has to depart from Scholzes$^{\prime}$s argument in the proof of \cite[Theorem 14.1.1.]{Berkeleylectures}. First, one checks, using that $\Ainf$ is strictly henselian and the local description from Lemma \ref{Windows ueber Ainf und BKF lokales Statement}, that $\G$-$\mu^{\sigma}$-windows over $\mathcal{F}_{\infintesimal}$ and $\G$-BKFs over $\mathcal{O}_{C}$ of type $\mu^{\sigma}$ are equivalent. Here $\G$-BKFs over $\mathcal{O}_{C}$ of type $\mu^{\sigma}$ are simliar defined as modifications of $G$-bundles on the curve of type $\mu^{\sigma}.$ In fact, I will actually show that $G$-BKFs are equivalent to tuples as in the theorem and under this being of type $\mu^{\sigma}$ coincides. Then one shows, that local mixed characteristic $\G^{\text{adic}}$-shtuka over $\Spa(C^{\flat})$ with a paw at $\varphi^{-1}(x_{C})$ uniquely extend to $\G$-BKFs over $\mathcal{O}_{C}.$ To extend $\varphi$-modules from $\mathcal{Y}_{[r,\infty)}$ to $\mathcal{Y}_{[r,\infty]},$ where one follows Scholze$^{\prime}$s notation and $r$ is some rational number $0\leq r < \infty,$ Lemma \ref{Fargues Lemma fuer Tannaka Formalismus} saves one, as this is \textit{not} an exact operation. Then one uses a GAGA-statement originally due to Kedlaya and reproduced in the form needed here e.g. in \cite[Prop. 5.3.]{JohannesExtending}, to transport $\G^{\text{adic}}$-torsors on $\mathcal{Y}_{[0,\infty]}$ to $\G$-torsors on $\Spec(\Ainf)-\lbrace \mathfrak{m} \rbrace,$ where $\mathfrak{m}$ denotes the unique maximal ideal of $\Ainf.$ It remains to extend over the unique closed point, where one can use an old argument of Colliot-Thélène, again reproduced by Anschuetz, see \cite[Prop. 6.5.]{JohannesExtending}. Using work of Kedlaya-Liu, one easily checks that $\G$-torsors on $\Spec(\mathbb{Z}_{p})$ are equivalent to $\G^{\text{adic}}$-torsors over $\mathcal{Y}_{[0,r)}$ together with an isomorphism to their Frobenius pullback.\footnote{In other words, one just observes that the equivalence in \cite[Prop. 12.3.5.]{Berkeleylectures} respects exact structures.} Then one has all ingredients one needs to follow the argument given in loc. cit. to conclude.
\end{proof}
To conclude this section, one can state the promised classification of adjoint nilpotent $\G$-$\mu^{\sigma}$ displays over $\mathcal{O}_{C}$ in a purely geometric way by using the Fargues-Fontaine curve.
\begin{Corollary}\label{Scholze-Weinstein G-displays}
The following categories are equivalent:
\begin{enumerate}
\item[(a):] Adjoint nilpotent $\G$-$\mu^{\sigma}$-displays over $\mathcal{O}_{C},$
\item[(b):] tuples $(\mathcal{E}_{1},\mathcal{E},\iota,\mathcal{T}),$ as in the previous theorem, such that $\Ad(\mathcal{E})$ has all HN-slopes $<1,$ where $\Ad(\mathcal{E})$ is the vector bundle on $X$ obtained by pushing out along the adjoint representation.
\end{enumerate}
\end{Corollary}
\begin{proof}
Let $\mathcal{P}=(Q,\alpha)$ be an adjoint nilpotent $\G$-$\mu^{\sigma}$-display over $\Spec(\mathcal{O}_{C}).$ Let $\overline{\mathcal{P}}$ be the reduction of $\mathcal{P}$ along the projection of $\mathcal{O}_{C}$ to the residue-field $\kappa.$ Then $\overline{\mathcal{P}}$ is determined by $[b]\in B(G),$ where $b=u\mu^{\sigma}(p),$ for some structure matrix $u\in L^{+}\G(\kappa)$ describing $\overline{\mathcal{P}}.$ Now fix a splitting of the projection $\mathcal{O}_{C}\rightarrow \kappa,$ so that one can associate to $b$ a $G$-bundle $\mathcal{E}_{b}$ on $X.$
\\
On the other hand, one can consider the uniquely determined tuple $(\mathcal{E}_{1},\mathcal{E},\iota,\mathcal{T}),$ as in the previous theorem, associated to $\mathcal{P}.$ It follows from Remark \ref{Bemerkung zum zweiten Buendel in den Modifikationen} (ii), that (non-canonically)
$$
\mathcal{E}\simeq \mathcal{E}_{b}.
$$
Indeed, here one has to use that for a $\G$-Breuil-Kisin-Fargues module $\mathcal{P}$ over $\mathcal{O}_{C}$ with reduction $\overline{P}$ over $k,$ one has a $\varphi$-equivariant isomorphism of $G$-torsors over $\Bcris(\mathcal{O}_{C}),$ lifting the identity,
$$
\mathcal{P}\times_{\Spec(\Ainf(\mathcal{O}_{C}))}\Spec(\Bcris(\mathcal{O}_{C}))\simeq \overline{\mathcal{P}}\times_{\Spec(W(k)),s}\Spec(\Bcris(\mathcal{O}_{C})),
$$
where one uses the chosen section to make the sense of the second base-change. In the case of $\GL_{n}$ this is contained in, \cite[Lemma 4.27.]{BMS1}, and the statement for $G$-torsors works similiarly.
The adjoint nilpotency condition says that the $G$-isocrystal $\Ad(b)$ has all slopes $>-1.$ Passing from $G$-isocrystals to $G$-bundles on the Fargues-Fontaine curve, the HN-slopes swap signs, thus the condition is that $\Ad(\mathcal{E})$ has all slopes $<1.$
\end{proof}
\section{Translation of the quasi-isogeny}\label{Section Translation of the Quasi-isogeny}
First, let me recall/introduce the set-up one is working in in this section:
Let $p\geq 3$ and $k=\overline{\mathbb{F}}_{p}$ be an algebraic closure of $\mathbb{F}_{p}.$ Let $R$ be an integral perfectoid $W(k)$-algebra with perfectoid pseudo-uniformizer $\varpi\in R.$ I will suppose that $R$ is $p$-torsion free, to be able to apply certain descent results, see Lemma \ref{Descent fuer Isokristalle}.
\\
As usual $(\G,\mu)$ is the pair of a reductive group scheme over $\mathbb{Z}_{p}$ and $\mu$ is a minuscule cocharacter that is defined over $W(k_{0}),$ where $k_{0}$ is some finite extension of $\mathbb{F}_{p}$ contained in $k.$ 
\\
Let $u_{0}\in \G(W(k)),$ that fulfills the adjoint nilpotence condition and $\mathcal{P}$ a $\G$-$\mu$-display over $R.$ One gives oneself furthermore a $G$-quasi-isogeny
$$
\rho\colon \mathcal{P}_{0}\times_{\Spec(k)}\Spec(R/p)\dashrightarrow \mathcal{P}\times_{\Spec(R)}\Spec(R/p).
$$
By the crystalline equivalence proven before, Proposition \ref{Kristalline Aequivalenz}, one may lift $\mathcal{P}$ uniquely to an adjoint nilpotent $\G$-$\mu$-window for the frame $\mathcal{F}_{\cris}(R),$ denote it by $(Q_{\cris},\alpha_{\cris}).$  By Prop. \ref{Deszent von Acris nach Ainf} it corresponds to a $\G$-$\mu$-window for the frame $\mathcal{F}_{\infintesimal}(R),$ which then gives rise to a $\G$-Breuil-Kisin-Fargues module $(\mathcal{P},\varphi_{\mathcal{P}})$ over $\Spec(\Ainf(R)),$ which is of type $\mu.$ Finally, let $\mathcal{P}_{\cris}=\mathcal{P}\times_{\Spec(\Ainf(R))}\Spec(\Acris(R))$ be the base-change, with Frobenius-structure $\varphi_{\mathcal{P}_{\cris}}.$ Then the aim is to prove the following statement:
\begin{Lemma}\label{Lemma Vergleich der Quasi-isogenie kristalline Qis, p torsionfreier Fall}
There is a uniquely determined isomorphism of $G$-torsors over $\Spec(\Bcris(R/p))$
$$
\rho_{\cris}\colon \mathcal{P}_{0}\times_{\Spec(W(k))}\Spec(\Bcris(R/p))\simeq \mathcal{P}_{\cris}\times_{\Spec(\Acris(R))}\Spec(\Bcris(R/p)),
$$
which is compatible with the Frobenius-structure and that lifts the given $G$-quasi-isogeny $$\rho\colon \mathcal{P}_{0}\times_{\Spec(k)}\Spec(R/p)\dashrightarrow \mathcal{P}\times_{\Spec(R)}\Spec(R/p).
$$
\end{Lemma}
\begin{Remark}
The statement 'that lifts the given $G$-quasi-isogeny' needs explanation, because one does not know descent for finite projective $W_{\mathbb{Q}}(R/p)$-modules resp. $G$-torsors over $\Spec(W_{\mathbb{Q}}(R/p))$ along faithfully flat étale morphisms $B=R/p\rightarrow B^{\prime}.$ Here is what is meant by this:
recall that for any integral perfectoid ring $R,$ there is a natural ring homomorphism
$$
\chi\colon \Acris(R/p)\rightarrow W(R/p),
$$
which induces
$$
\chi_{\mathbb{Q}}\colon \Bcris(R/p)\rightarrow W_{\mathbb{Q}}(R/p).
$$
Now the datum of a $G$-quasi isogeny
$$
\rho\colon \mathcal{P}_{0}\times_{\Spec(k)}\Spec(R/p)\dashrightarrow \mathcal{P}\times_{\Spec(R)}\Spec(R/p)
$$
comes down to the following data: choose a $p$-completely faithfully flat étale morphism $R\rightarrow R^{\prime}$ of integral perfectoids rings trivializing the $\G$-$\mu$-display $(Q,\alpha)$ over $R;$ choose a structure matrix $U^{\prime}\in \G(W(R^{\prime}))$ representing this base-change. Write $B=R/p,$ $B^{\prime}=R^{\prime}/p,$ $B^{\prime (2)}=R^{\prime}/p \otimes_{R/p}R^{\prime}/p,$ $$B^{\prime (3)}=R^{\prime}/p \otimes_{R/p}R^{\prime}/p \otimes_{R/p}R^{\prime}/p.$$ Then $\rho$ corresponds to an element $g^{\prime}\in G(W_{\mathbb{Q}}(B^{\prime})),$ such that
$$
g^{\prime -1}u_{0}\mu(p)F(g^{\prime})=\overline{U^{\prime}}\mu(p),
$$
such that the two pullbacks $p^{1 *}(g^{\prime}),$ $p^{2 *}(g^{\prime})$ are isomorphic in $G(W_{\mathbb{Q}}(B^{\prime (2)}))$ and satisfy a cocycle condition in $G(W_{\mathbb{Q}}(B^{\prime (3)})).$ Similiarly, one can unpack what it means to give an isomorphism of $G$-torsors over $\Spec(\Bcris(R/p))$ as above, using the upcoming Lemma \ref{Descent fuer Isokristalle} - to apply this I have to assume that $R$ is $p$-torsion free, so let me assume this from now on: to give onself an isomorphism $$
\rho_{\cris}\colon \mathcal{P}_{0}\times_{\Spec(W(k))}\Spec(\Bcris(R/p))\simeq \mathcal{P}_{\cris}\times_{\Spec(\Acris(R))}\Spec(\Bcris(R/p)),
$$
is equivalent to give oneself an element $g^{\prime}_{\cris}\in G(\Bcris(B^{\prime}))$ satisfying
$$
g^{\prime -1}_{\cris}u_{0}\mu(p)\varphi(g^{\prime}_{\cris})=U^{\prime}_{\text{cris}}\mu(p);
$$
here $U^{\prime}_{\cris}$ is a structure matrix representing the banal $\G$-$\mu$-window 
$$
(Q_{\cris},\alpha_{\cris})\times_{\Spec(B)}\Spec(B^{\prime})
$$
for the frame $\mathcal{F}_{\cris}(B^{\prime}).$ As before one requires that the two pullbacks of $g^{\prime}_{\cris}$ towards $\Spec(\Bcris(B^{\prime (2)}))$ become isomorphic and satisfy a cocycle condition over $\Spec(\Bcris(B^{\prime (3)})).$
\\
After this unpacking, one can finally explain what it means that $\rho_{\cris}$ lifts the $G$-quasi isogeny $\rho:$ it just means that
$$
\chi_{\mathbb{Q}}(g^{\prime}_{\cris})=g^{\prime},
$$
compatible with the isomorphism of the pullbacks and the cocycle condition.
\end{Remark}
In the following, I will sometimes refer to the datum of an isomorphism of $G$-torsors $\rho_{\cris}$ as above as a crystalline $G$-quasi-isogeny.
\\
Before I embark on the proof, let me point out here that the principal challenge is that one works with torsors (and thus groupoids) throughout and therefore a notion of 'isogeny' does not make sense; making it impossible to deduce such a statement directly from the crystalline equivalence proven before.
\\
The idea of the proof is simple: first reduce to the case, where the $\G$-$\mu$-display $\mathcal{P}$ is banal, represented by some structure matrix $U\in \G(W(R)),$ and then attack this case by direct calculation (getting inspired by the unique lifting lemmas in their different versions).
\subsection{Reduction to the banal case}
So much for the strategy; as said before, let me first quickly explain how to reduce to the banal case. This will use the following statement:
\begin{Lemma}{(Drinfeld-Matthew)}\label{Descent fuer Isokristalle}
Let $R$ be a $p$-torsion free integral perfectoid ring. Consider the functor on the category of integral perfectoid $R$-algebras, sending $R\rightarrow R^{\prime}$ towards
$$
\Isom_{G}(\mathcal{P}_{0}\times_{\Spec(W(k))}\Spec(\Bcris(R^{\prime}/p)),\mathcal{P}_{\cris}\times_{\Spec(\Acris(R))}\Spec(\Bcris(R^{\prime}/p))).
$$
Then this functor is a sheaf for the $p$-adic étale topology.
\end{Lemma}
To prove this lemma, I will use some heavy machinery; one key input being a result of Drinfeld \cite{drinfeld2018stacky} about descent of vector bundles in 'Banachian modules' and the other being the computation of Bhatt-Morrow-Scholze contained in \cite{BMS2} of $\Acris(S)/p,$ where $S$ is any quasi-regular semiperfect $\mathbb{F}_{p}$-algebra, which was already used before in verifying the sheaf property for $\Acris(.).$ In the following proof one writes 
$$
\Bun_{\mathbb{Q}}(A)=\Bun(A\otimes_{\mathbb{Z}_{p}}\mathbb{Q})
$$
for the groupoid of finite locally free $A\otimes_{\mathbb{Z}_{p}}\mathbb{Q}$-modules for any $\mathbb{Z}_{p}$-flat ring $A.$
\begin{proof}
The key claim is the following:
\\
\textit{Claim}: Let $R$ be a $p$-torsion free integral perfectoid ring. Denote by $B:=R/p$ and let $B\rightarrow B^{\prime}$ be a faithfully flat étale covering; observe that by the standing assumption that $R$ is $p$-torsion free, the ring $B$ is quasi-regular semiperfect and also all $B^{\prime (n)}=B^{\prime}\otimes_{B}...\otimes_{B}B^{\prime}$ are quasi-regular semiperfect - they form a simplicial $\mathbb{F}_{p}$-algebra $B^{\prime}_{\bullet}$. Consider the simplicial ring $\Acris(B^{\prime}_{\bullet})$ with $n$-th component given by $\Acris(B^{\prime (n)}).$ Then I claim that the natural morphism induces an equivalence
$$
\Bun_{\mathbb{Q}}(\Acris(B))\rightarrow \lim_{\Delta} \Bun_{\mathbb{Q}}(\Acris(B^{\prime}_{\bullet})).
$$
This will be deduced from Drinfeld's result mentioned above using the aforementioned calculation of Bhatt-Morrow-Scholze.
\\
Namely, recall that if $L\Omega_{B}:=L\Omega_{B/\mathbb{F}_{p}}$ is the (absolute) derived de Rham complex of $B,$ then they check in \cite[Prop.8.12]{BMS2} that one has an isomorphism
$$
\Acris(B)/p\simeq L\Omega_{B},
$$
which is functorial in morphisms of quasi-regular semiperfect $\mathbb{F}_{p}$-algebras.
Next, note that the étale morphism $B\rightarrow B^{\prime}$ gives the following base-change identity for the derived de Rham complex
$$
L\Omega_{B}\otimes_{B}B^{\prime}\simeq L\Omega_{B^{\prime}}.
$$
In fact, this may be checked as in Lemma \ref{Descent fuer Acris} by comparing the conjugate filtration.
It follows that $\Acris(B)\rightarrow \Acris(B^{\prime})$ is $p$-completely faithfully flat étale, as modulo $p$ it identifies with the base-change $L\Omega_{B}\rightarrow L\Omega_{B}\otimes_{B}B^{\prime}$ and since $\Acris(B)$ and $\Acris(B^{\prime})$ are $p$-torsion free, one may use the result of Drinfeld \cite[Prop. 3.5.4]{drinfeld2018stacky} to obtain that
$$
\Bun_{\mathbb{Q}}(\Acris(B))\simeq \lim_{\Delta}\Bun_{\mathbb{Q}}(\Acris(B^{\prime})_{\bullet}),
$$
where $\Acris(B^{\prime})_{\bullet}$ is the simplicial $p$-complete ring with $n$-th component given by
$$
\Acris(B^{\prime})^{(n)}=\Acris(B^{\prime})\widehat{\otimes}_{\Acris(B)}\Acris(B^{\prime})....\widehat{\otimes}_{\Acris(B)}\Acris(B^{\prime}).
$$
To finish the proof of the above claim, it remains therefore just to identify the simplicial ring $\Acris(B^{\prime})_{\bullet}$ with the simplicial ring $\Acris(B^{\prime}_{\bullet}).$ By functoriality, there is a morphism of simplicial rings
$$
\Acris(B^{\prime})_{\bullet}\rightarrow \Acris(B^{\prime}_{\bullet}).
$$
One has to see that it induces isomorphisms
$$
\Acris(B^{\prime})^{(n)}\simeq \Acris(B^{\prime (n)}).
$$
Note that both sides are $p$-adic and $p$-torsion free. It therefore suffices to see that they are isomorphic after modding out by $p.$ But this can then be deduced from the formulas $
\Acris(B)/p\simeq L\Omega_{B}
$ and $
L\Omega_{B}\otimes_{B}B^{\prime}\simeq L\Omega_{B^{\prime}}
$ above.
\\
In other words, one has now verified that the fibered category $\Bun_{\mathbb{Q}}(\Bcris(\mathcal{R}))$ (with the notation as in Lemma \ref{Lau perfektoide Garbe}) is indeed a stack on $\Spec(R/p)^{\text{aff}}_{\text{ét}}.$ As usual, it is not hard to deduce the same descent result for $G$-torsors, which then implies the statement one is after here. Namely, if one chooses an embedding $G\hookrightarrow \GL_{n},$ then there is a linear representation $\rho\colon\GL_{n}\rightarrow \GL(W),$ where $W$ is a finite dimensional $\mathbb{Q}_{p}$-vector space, such that one may realize $G$ as the fixer of a line $L\subset W.$ Then to give onself a $G$-torsor $\mathcal{P}$ over $\Spec(\Bcris(R))$ is equivalent to give a finite locally free rank $n$ $\Bcris(R)$-module $\mathcal{E}$, together with a locally direct factor line sub-bundle of $\rho_{*}(\mathcal{\mathcal{E}});$ in this way one may reduce the problem of descending $G$-torsors to the linear problem one just dealt with.
\end{proof}
Using this lemma, it is easy to see that one may reduce to the banal case.
\subsection{Proof in the banal case}
From now on, I assume that $\mathcal{P}$ is banal, represented by a structure matrix $U\in \G(W(R));$ let me denote by $\overline{U}\in \G(W(R/p))$ the reduction. The $G$-quasi-isogeny between $\mathcal{P}_{0}$ and $\mathcal{P}$ is then just the datum of an element $g\in G(W_{\mathbb{Q}}(R/p)),$ such that
\begin{equation}
g^{-1}u_{0}\mu(p)F(g)=\overline{U}\mu(p),
\end{equation}
that is
$$
u_{0}\mu(p)F(g)\mu(p)^{-1}\overline{U}^{-1}=g
$$
(which is the form I will use later)
is satisfied. One can lift $U$ to an element $U_{\cris}\in \G(\Acris(R/p))$ that is uniquely determined up to $\Phi_{\cris}$-conjugation and which is a structure matrix for the banal $\G$-$\mu$-window $\mathcal{P}_{\cris}$ for the frame $\mathcal{F}_{\cris}(R).$ I will write $\delta\colon W(k)\rightarrow \Acris(R/p)$ for the ring-homomorphism that gives the $W(k)$-algebra structure on $\Acris(R/p).$ I recall again that $\Acris(R/p)$ is $p$-torsion free. By Stacks project, Tag 05GG, $\Acris(R/p)$ is $p$-adically complete(and thus separated). Recall that $\Bcris(R/p)=\Acris(R/p)[1/p]$ is the localization at $p.$ One gives $\Bcris(R/p)$ the Banach-topology, such that $\Acris(R/p)\subseteq \Bcris(R/p)$ is open. With this topology, the $\mathbb{Q}_{p}$-algebra $\Bcris(R/p)$ has the structure of a Banach-space for the norm
$$
|.|\colon \Bcris(R/p)\rightarrow \mathbb{R}_{\geq 0},
$$
given by
$$
|b|=\inf\lbrace |\lambda|_{\mathbb{Q}_{p}}\colon \lambda\cdot b\in \Acris(R/p)\rbrace.
$$
For later use, let me recall that one has a surjection
$$
\chi\colon \Acris(R/p)\rightarrow W(R/p),
$$ 
(which induces the surjection
$$
\chi_{\mathbb{Q}}\colon \Bcris(R/p)\rightarrow W_{\mathbb{Q}}(R/p) )
$$
whose kernel $K_{\cris}:=\ker(\chi)$ is generated by elements
$$
[x]^{(n)},
$$
for $n\geq 1$ and $x\in \ker(R^{\flat}\rightarrow R/p),$ c.f. \cite[Prop. 5.3.5.]{JohannesArthur}. It follows from (\cite[Tag 00CT]{stacks})
$$
W_{\mathbb{Q}}(R/p)=\frac{\Bcris(R/p)}{K_{\cris}[1/p]},
$$
that the kernel of $\chi_{\mathbb{Q}}$ is $K_{\cris}[1/p].$ It identifies with $K_{\cris}\otimes_{\mathbb{Z}_{p}}\mathbb{Q}_{p}$ as an $\Bcris(R/p)=\Acris(R/p)\otimes_{\mathbb{Z}_{p}}\mathbb{Q}_{p}$-module. In total, one sees that $\ker(\chi_{\mathbb{Q}})$ is generated by elements of the form
$$
[x]^{(n)}\otimes \lambda,
$$
with $n\geq 1,$ $x$ as above and $\lambda\in\mathbb{Q}_{p}.$
\\
Before I go on, let me discuss a bit topological pathologies for the ring $W_{\mathbb{Q}}(R/p),$ which come from the fact that $R/p$ has tons of nilpotents, which leads to the ring $W(R/p)$ having a priori unbounded $p^{\infty}$-torsion, so one gets problems with $W_{\mathbb{Q}}(R/p).$ First, one gives the ring $W_{\mathbb{Q}}(R/p)$ the topology given by declaring as a basis of open neighborhoods of zero
$$
\lbrace \text{Im}((p^{m}W(R/p))\rightarrow W_{\mathbb{Q}}(R/p))\rbrace_{m \geq 0}.
$$
Here one has to be careful, that this topology is not given by ideals, so that it is not $W(R/p)$-linear.
Note that although $W(R/p)$ is $p$-adic, with the above topology on $W_{\mathbb{Q}}(R/p),$ this ring is a priori neither separated nor complete. Nevertheless, with these topologies the morphism
$$
\chi_{\mathbb{Q}}\colon \Bcris(R/p)\rightarrow W_{\mathbb{Q}}(R/p)
$$
is continous. To get around these worries, look at the ring $W(R/p)^{\text{bdd}}=W(R/p)/W(R/p)[p^{\infty}],$ where $W(R/p)[p^{\infty}]=\cup_{n\geq 1} \Ann(p^{n}).$ In the ring $W(R/p)^{\text{bdd}}$ the element $p$ is regular. Note that passing to the localization at $p,$ one gets the identity
$$
W_{\mathbb{Q}}(R/p)=W(R/p)^{\text{bdd}}[1/p].
$$
I give $W(R/p)^{\text{bdd}}$ the $p$-adic topology and $W(R/p)^{\text{bdd}}[1/p]$ the topology with basis of open neighborhoods of zero given by
$$
\lbrace \text{Im}((p^{m}W(R/p)^{\text{bdd}})\rightarrow W(R/p)^{\text{bdd}}[1/p]) \rbrace_{m\geq 0}.
$$
With this topology, the identity morphism $W_{\mathbb{Q}}(R/p)=W(R/p)^{\text{bdd}}[1/p]$ is continous and the advantage is now that the topological ring $W(R/p)^{\text{bdd}}[1/p]$ is separated and complete in the given topology. In fact, one has to explain why this ring stays indeed complete; for this it is enough to explain why $W(R/p)^{\text{bdd}}$ is $p$-adic. Here one can use a result of Bouthier-Česnavičius, \cite[Lemma 2.1.11.]{BouthierCesnavicius}, which implies that
$$
\widehat{(W(R/p)^{\text{bdd}})}\simeq (\widehat{W(R/p)})^{\text{bdd}}=W(R/p)^{\text{bdd}},
$$
where all completions are $p$-adic and in the equality I am using that $W(R/p)$ is $p$-adic because $R/p$ is semi-perfect. In Bothier-Česnavičius' notation, one is looking at the Gabber-Romero triple $$(W(R/p),t=p,I=(1))$$ and what I denote by $(.)^{\text{bdd}}$ they denote by $\overline{(.)}.$
Later this will be used as follows: suppose one is given a converging sequence $a_{n}\in \Bcris(R/p)$ with $a=\lim_{n\rightarrow \infty}a_{n}$ and such that $\chi(a_{n})=c$ is constant for all $n\geq 1.$ Then I want to deduce that $\chi(a)=c.$ Since the ring $W_{\mathbb{Q}}(R/p)$ is not separated, so uniqueness of limits fails, this is a priori not clear; but the identity $\chi(a)=c$ can be checked in $W(R/p)^{\text{bdd}}[1/p].$ Since $\chi_{\mathbb{Q}}$ and $W_{\mathbb{Q}}(R/p)\rightarrow W(R/p)^{\text{bdd}}[1/p]$ are continous and since the topology on $W(R/p)^{\text{bdd}}[1/p]$ is separated, the image of $\chi(a)$ in $W(R/p)^{\text{bdd}}[1/p]$ agrees with the image of $c,$ so that $\chi(a)=c$ also in $W_{\mathbb{Q}}(R/p).$ This will be later used to see that one really lifted the given quasi-isogeny.
\\
Now I can finally come to the statement that one wants to prove in the banal case:
\begin{Lemma}
There is is a uniquely determined element $g_{\cris}\in G(\Bcris(R/p)),$ such that
\begin{equation}
g_{\cris}^{-1}\delta(u_{0})\mu(p)\varphi(g_{\cris})=U_{\cris}\mu(p)
\end{equation}
and 
\begin{equation}
\chi_{\mathbb{Q}}(g_{\cris})=g.
\end{equation}
\end{Lemma}
To prepare the proof, I will now apply the Tannakian reconstruction theorem to give the following interpretation of the datum of the $G$-quasi-isogeny $g\in G(W_{\mathbb{Q}}(R/p))$ and of the crystalline $G$-quasi-isogeny $g_{\cris}\in G(\Bcris(R/p))$ one wants to produce:
\begin{Lemma}\label{Tannakische Interpretation einer Quasi-Isogenie}
\begin{enumerate}
\item[(a):] The datum of the $G$-quasi-isogeny $g\in G(W_{\mathbb{Q}}(R/p))$ is equivalent to the following datum: For all $(\rho,T)\in \Rep_{\mathbb{Z}_{p}}(\G),$ an element
$$
g(\rho,T)\in \GL_{r}(T\otimes_{\mathbb{Z}_{p}}W_{\mathbb{Q}}(R/p)),
$$
such that 
\begin{equation}
g((\triv,\mathbb{Z}_{p}))=\id_{W_{\mathbb{Q}}(R/p)},
\end{equation} 
\begin{equation}
g(\rho,T)\otimes_{W_{\mathbb{Q}}(R/p)} g(\rho^{\prime},T^{\prime})=g((\rho,T)\otimes_{\mathbb{Z}_{p}}(\rho^{\prime},T^{\prime}))
\end{equation}
inside of $\GL_{r}(T\otimes_{\mathbb{Z}_{p}}T^{\prime}\otimes_{\mathbb{Z}_{p}}W_{\mathbb{Q}}(R/p)),$ which is furthermore functorial in $(\rho,T)\in \Rep_{\mathbb{Z}_{p}}(\G)$ and such that
\begin{equation}
g(\rho,T)^{-1}\rho(u_{0}\mu(p))F(g(\rho,T))=\rho(\overline{U}\mu(p)).
\end{equation}
\item[(b):] Similiarly everything for $g_{\cris}:$ The datum of the crystalline $G$-quasi-isogeny $g_{\cris}\in G(\Bcris(R/p))$ is equivalent to the following datum: For all $(\rho,T)\in \Rep_{\mathbb{Z}_{p}}(\G),$ an element
$$
g_{\cris}(\rho,T)\in \GL_{r}(T\otimes_{\mathbb{Z}_{p}}\Bcris(R/p)),
$$
such that 
\begin{equation}
g_{\cris}((\triv,\mathbb{Z}_{p}))=\id_{\Bcris(R/p)},
\end{equation} 
\begin{equation}
g_{\cris}(\rho,T)\otimes_{\Bcris(R/p)} g_{\cris}(\rho^{\prime},T^{\prime})=g_{\cris}((\rho,T)\otimes_{\mathbb{Z}_{p}}(\rho^{\prime},T^{\prime})),
\end{equation}
inside of $\GL_{r}(T\otimes_{\mathbb{Z}_{p}}T^{\prime}\otimes_{\mathbb{Z}_{p}}\Bcris(R/p)),$ which is furthermore functorial in $(\rho,T)\in \Rep_{\mathbb{Z}_{p}}(\G)$ and such that
\begin{equation}
g_{\cris}(\rho,T)^{-1}\rho(\delta(u_{0})\mu(p))\varphi(g_{\cris}(\rho,T))=\rho(U_{\cris}\mu(p)).
\end{equation}
\end{enumerate}
\end{Lemma}
\begin{proof}
The item (a) follows from the Tannakian re-interpretation of $\G$-$\mu$-displays that was worked out by Patrick Daniels, see \cite[Lemma 3.27 ]{Daniels1}. Now the item (b) follows similiarly by using the h-frame structure on the frame $\mathcal{F}_{\cris}(R/p)$ that I discussed in section \ref{hframe auf prismatischem Frame}; fortunately one may use here results that were proven by Daniels in his second paper \cite{Daniels2}. Namely, let me denote by $\mathcal{F}_{\cris}^{h}(R/p)$ the h-frame introduced above. Then sending $B=R/p\rightarrow B^{\prime}=R^{\prime}/p$ towards $\mathcal{F}_{\cris}^{h}(B^{\prime})$ defines an étale sheaf of h-frames by descending the Nygaard-filtration; let me momentarily denote by $\underline{S}_{\cris}(B^{\prime})$ the $\mathbb{Z}$-graded ring corresponding to the h-frame structure (i.e. $\underline{S}_{\cris}(B^{\prime})_{\geq 0}=\bigoplus_{n\geq 0}\mathcal{N}^{\geq n}(\Prism_{B^{\prime}})$ and $\underline{S}_{\cris}(B^{\prime})_{\leq 0}=\Prism_{B^{\prime}}[t]).$ Next, note that the fibered category of finitely generated projective graded $\underline{S}_{\cris}(B^{\prime})$-modules satisfies descent for $\Spec((R/p)^{\text{aff}}_{\text{ét}})$ (i.e. it satisfies $p$-completely faithfully flat étale descent): to argue as in the proof \textit{loc.cit.}[Prop. A.18.] (coupled with $p$-completely faithfully flat descent for finite projective modules as provided by the Appendix to \cite{JohannesArthur} and its extension to finitely generated projetive graded (!) modules) one has to check that if $B\rightarrow B^{\prime}$ is $p$-completely faithfully flat étale, then
$$
\underline{S}_{\cris}(B)\rightarrow \underline{S}_{\cris}(B^{\prime})
$$
is $p$-completely faithfully flat étale, which follows because $\Prism_{B}\rightarrow \Prism_{B^{\prime}}$ is $p$-completely faithfully flat étale. Furthermore, one has that
$$
\mathcal{N}^{\geq n}(\Prism_{B})\otimes_{\Prism_{B}}\Prism_{B^{\prime}}\simeq \mathcal{N}^{\geq n}(\Prism_{B^{\prime}}).
$$
Namely, the ring $\Prism_{B^{\prime}}$ has then two $\mathbb{N}$-index and descending filtrations $\mathcal{N}^{\bullet}(\Prism_{B})\otimes_{\Prism_{B}}\Prism_{B^{\prime}}$ and $\mathcal{N}^{\bullet}(\Prism_{B^{\prime}})$ and there is a natural morphism of filtrations $\mathcal{N}^{\bullet}(\Prism_{B})\otimes_{\Prism_{B}}\Prism_{B^{\prime}}\rightarrow \mathcal{N}^{\bullet}(\Prism_{B^{\prime}}).$ To show that this is an isomorphism, it suffices to test this on graded pieces, but then
\begin{align*}
\gr^{i}(\mathcal{N}^{\bullet}(\Prism_{B})\widehat{\otimes}_{\Prism_{B}}\Prism_{B^{\prime}}) & \simeq \gr^{i}(\mathcal{N}^{\bullet}(\Prism_{B})\otimes_{\Prism_{B}}\Prism_{B^{\prime}} \\
 & \simeq \gr^{i}(\overline{\Prism}_{B})\lbrace i \rbrace \otimes_{B}B^{\prime} \\
 & \simeq \gr^{i}(\overline{\Prism}_{B^{\prime}})\lbrace i \rbrace \\
 &  \simeq \gr^{i}(\mathcal{N}^{\bullet}(\Prism_{B^{\prime}})), 
\end{align*}
where in the first isomorphism one uses that $\Prism_{B}\rightarrow \Prism_{B^{\prime}}$ is $p$-completely flat and $\Prism_{B},\Prism_{B^{\prime}}$ are $p$-torsion free (so that $.\widehat{\otimes}_{\Prism_{B}}\Prism_{B^{\prime}}$ is exact), in the second the computation of the graded pieces of the Nygaard-filtration provided by Bhatt-Scholze in \cite[Thm.12.2.]{bhatt2022prisms}, in the third that $B\rightarrow B^{\prime}$ is étale and in the last again [Thm.2.2.]\textit{loc.cit.} In Daniels' words (see \cite[Def. 3.6]{Daniels2}, the étale sheaf of $\mathbb{Z}$-graded rings $\underline{S}_{\cris}(.)$ satisfies descent for modules on $\Spec(R/p)^{\text{aff}}_{\text{ét}}.$ It follows from Lau's computation of the window group \cite[Prop. 6.2.2, Rem. 6.3.3]{LauHigherFrames} that the version of the window group agrees with the one defined using h-frames. Then one may cite \cite[Thm. 3.14]{Daniels2}, to obtain a Tannakian interpretation of $\G$-$\mu$-windows for the frame $\mathcal{F}_{\cris}(R/p).$ Observing that higher windows for the h-frame $\mathcal{F}^{h}_{\cris}(R/p)$ always admit a normal decomposition (which follows because $\Prism_{B}$ is henselian along $\mathcal{N}^{\geq 1}(\Prism_{B}),$ so that one may use \cite[Lemma 3.1.4]{LauHigherFrames}), one may redo section 3.4. in \cite{Daniels1}, to obtain the desired Tannakian interpretation of the crystalline $G$-quasi-isogeny.\footnote{Is this detailed enough?}
\end{proof}
The next observation will be critical for the following arguments to work: denote as usual by $\varphi$ the Frobenius-lift on $\Acris(R/p).$ One then has that for $[x]^{(n)},$ $x,n$ as above,
$$
\varphi([x]^{(n)})=\frac{(pn)!}{n!}[x]^{(np)},
$$
and inductively that
$$
\varphi^{m}([x]^{(n)})=\frac{(p^{m}n)!}{n!}[x]^{(np^{m})}.
$$
The key point is now that the $p$-adic valuation of $\frac{(p^{m}n)!}{n!}$ grows very fast. 
\begin{Lemma}
For $m\geq 1, n\geq 1,$ one has that
$$
\nu_{p}(\frac{(p^{m}n)!}{n!})=n\frac{(p^{m}-1)}{p-1}.
$$ 
\end{Lemma}
\begin{proof}
Recall the following formula: if $n=n_{0}p^{0}+n_{1}p^{1}+...+n_{\ell}p^{\ell}$ is written in base $p,$ let $$s_{p}(n)=\sum_{k=0}^{\ell}n_{k},$$ then
\begin{equation}\label{Legendres Formel}
\nu_{p}(n!)=\frac{n-s_{p}(n)}{p-1}.
\end{equation}
I claim that $\nu_{p}((p^{m}n)!)=n\frac{(p^{m}-1)}{p-1}+\nu_{p}(n!).$ This implies the lemma.
But one has $s_{p}(n)=s_{p}(p^{m} n),$ so that the above formula (\ref{Legendres Formel}) gives
$$
\nu_{p}((p^{m}n)!)=\frac{p^{m}n-s_{p}(p^{m}n)}{p-1}=\frac{p^{m}n-s_{p}(n)}{p-1}=n\frac{p^{m}-1}{p-1}+\frac{n-s_{p}(n)}{p-1}=n\nu_{p}(p^{m}!)+\nu_{p}(n!),
$$
as desired.
\end{proof}
Now one can turn to the proof of the statement in the banal case:
\\
\textit{Uniqueness:} To prepare the proof, take a representation 
$$
\rho\colon \G\rightarrow \GL_{\mathbb{Z}_{p}}(T)
$$
on a finite free rank $r$ $\mathbb{Z}_{p}$-module.
\\
Of course, $\mu^{\prime}=\rho\circ\mu$ is not minuscule anymore, but one has the following fact: if $$M\in \Mat_{r}(p^{-N}\Acris(R/p)),$$ then $\mu^{\prime}(p)M\mu^{\prime}(p)^{-1}\in \Mat_{r}(p^{-(N+c)}\Acris(R/p)),$ i.e. conjugation by $\mu^{\prime}(p)$ makes the denominators worse, but only by a linear amount. This makes it $p$-adically controllable. To convince oneself of this, note that the cocharacter defined over $W(k_{0})[1/p]$
$$
\mu^{\prime}\colon \mathbb{G}_{m,W(k_{0})[1/p]}\rightarrow G_{W(k_{0})[1/p]}\rightarrow \GL_{r,W(k_{0}[1/p])}
$$
is given by $z\mapsto \text{diag}(t^{m_{1}},...,t^{m_{n}}),$ where $m_{i}\in\mathbb{Z}.$ Then one may take for $c=\min_{1\leq i,j \leq n}(m_{i}-m_{j}),$ because if $B$ is a $W(k_{0})[1/p]$-algebra and $M\in \Mat_{r}(B)$, then the $(i,j)$-th entry of $\mu^{\prime}(p)M\mu^{\prime}(p)^{-1}$is
$$
p^{m_{i}-m_{j}}M_{i,j}.
$$
Now assume that $g_{\cris}^{\prime}(\rho,T)\in \GL(T\otimes_{\mathbb{Z}_{p}}\Bcris(R/p)),$ such that $$\chi_{\mathbb{Q}}(g_{\cris}^{\prime}(\rho,T))=\chi_{\mathbb{Q}}(g_{\cris}(\rho,T))=g(\rho,T)$$ and which satisfies the equation of a crystalline quasi-isogeny.  Consider 
$$
M=g_{\cris}^{\prime}(\rho,T)-g_{\cris}(\rho,T)\in \End(T\otimes_{\mathbb{Z}_{p}}\Bcris(R/p)).
$$
The aim is to see that $M=0.$
Since $g_{\cris}(\rho,T)$ and $g_{\cris}^{\prime}(\rho,T)$ lift $g(\rho,T)$ under $\chi_{\mathbb{Q}},$ one sees that $M$ has coefficients in $\ker(\chi_{\mathbb{Q}})=K_{\cris}\otimes_{\mathbb{Z}_{p}}\mathbb{Q}_{p}.$ Now let $N_{0}:=\min_{i,j}(\nu_{p}(M_{i,j})),$ so that $M\in \Mat_{r}(p^{N_{0}}K_{\cris}).$ Furthermore, it is true that
\begin{equation}\label{Gleichung crystalline quasi-isogeny M}
M=\rho(\delta(u_{0}))\mu^{\prime}(p)\varphi(M)\mu^{\prime}(p)^{-1}\rho(U_{\cris}^{-1}),
\end{equation}
since both $g_{\cris}(\rho,T)$ and $g_{\cris}^{\prime}(\rho,T)$ satisfy the equations of a crystalline quasi-isogeny.
I claim that this implies that $M=0.$ Indeed, it is enough to show that the coefficients of $M$ lie in $\cap_{n=1}^{\infty} p^{n}\Acris(R/p),$ because $\Acris(R/p)$ is $p$-adically separated. First notice, that the equation (\ref{Gleichung crystalline quasi-isogeny M}) implies that 
$$
M\in \Mat_{r}(p^{-c}\varphi(p^{N_{0}}K_{\cris})).
$$
Iterating the equation, it follows that 
$$
M=\rho(\delta(u_{0}))\mu^{\prime}(p)\varphi[\rho(\delta(u_{0}))\mu^{\prime}(p)\varphi(M)\mu^{\prime}(p)^{-1}\rho(U_{\cris}^{-1})]\mu^{\prime}(p)^{-1}\rho(U_{\cris}^{-1}),
$$
so that one gets
$$
M\in \Mat_{r}(p^{-c}\varphi(p^{-c}\varphi(p^{N_{0}}K_{\cris})))\subseteq \Mat_{r}(p^{-2c}\varphi^{2}(p^{N_{0}}K_{\cris})).
$$
 It follows by induction that for every $m\geq 1,$ the following containement is satisfied:
$$
M\in \Mat_{r}(p^{-cm}\varphi^{m}(p^{N_{0}}K_{\cris})).
$$
But observe that $\varphi^{m}(K_{\cris})\subseteq p^{\frac{p^{m}-1}{p-1}}\Acris(R/p).$ Indeed, a typical element of $K_{\cris}$ is of the form
$$
[x]^{(n)},
$$
$n\geq 1,$ $x\in \ker(R^{\flat}\rightarrow R/p),$ as explained above. Then $\varphi^{m}([x]^{(n)})=\frac{(np^{m})!}{n!}[x]^{(np^{m})},$ but by the above Lemma,
$$
\nu_{p}(\frac{(np^{m})!}{n!})=\nu_{p}((np^{m})!)-\nu_{p}(n!)=n\frac{p^{m}-1}{p-1}.
$$
It follows that $\varphi^{m}([x]^{(n)})$ is indeed contained in $p^{n\frac{p^{m}-1}{p-1}}\Acris(R/p),$ which itself is then also contained inside $p^{\frac{p^{m}-1}{p-1}}\Acris(R/p).$ Observe now that
$$
p^{-mc+N_{0}}\cdot p^{\frac{p^{m}-1}{p-1}}=p^{\frac{p^{m}-1}{p-1}-cm+N_{0}}\rightarrow \infty,
$$
for $m\rightarrow \infty,$ because $p\geq 2$ and exponential growth beats any linear growth. In some sense, the fact that the action of $\varphi$ on $\ker(\chi)$ is very (!) $p$-adically nilpotent is the key why this argument works.
\\
In total, one deduces the vanishing of $M$ from the fact that $\Acris(R/p)$ is $p$-adically separated. This finishes the uniqueness.
\\
\textit{Functoriality and $\otimes$-compatibility, if a lift exists:} Observe first, that the previously proven uniqueness will also imply functoriality of $g_{\cris}(\rho,T)\in \GL(T\otimes_{\mathbb{Z}_{p}}\Bcris(R/p)),$ if it exists. Indeed, if $f\colon (\rho,T)\rightarrow (\rho^{\prime},T^{\prime})$ is a morphism of representations, then both $f(g_{\cris}(\rho,T)),g_{\cris}(\rho^{\prime},T^{\prime})\in \GL(T\otimes_{\mathbb{Z}_{p}}\Bcris(R/p))$ satisfy the equation of a crystalline quasi-isogeny (use that $\rho^{\prime}=\rho\circ f$) and
$$
\chi_{\mathbb{Q}}(f(g_{\cris}(\rho,T)))=f(\chi_{\mathbb{Q}}(g_{\cris}(\rho,T)))=f(g(\rho,T))=g(\rho^{\prime},T^{\prime})=\chi_{\mathbb{Q}}(g_{\cris}(\rho^{\prime},T^{\prime})),
$$
(use in the first equation that $f\colon \GL_{\mathbb{Z}_{p}}(T)\rightarrow \GL_{\mathbb{Z}_{p}}(T^{\prime})$ is a group scheme homomorphism and in the third equation known functoriality for the $G$-quasi-isogeny)
so that the uniqueness part proven above implies that actually $f(g_{\cris}(\rho,T))=g_{\cris}(\rho^{\prime},T^{\prime}).$ A similiar argument also takes care of the $\otimes$-compatibility of the hypothetical lifts $g_{\cris}(\rho,T).$
\\
\textit{Existence:}
Let a $G$-quasi-isogeny $g\in G(W_{\mathbb{Q}}(R/p))$ be given. Take a representation $(\rho,T)\in \Rep_{\mathbb{Z}_{p}}(\G)$ of $\G$ on a finite free, rank $r$ $\mathbb{Z}_{p}$-module. Consider $\rho(g)\in \Mat_{r}(W_{\mathbb{Q}}(R/p)).$ Because $\chi_{\mathbb{Q}}\colon \Bcris(R/p)\rightarrow W_{\mathbb{Q}}(R/p)$ is surjective, one can lift $\rho(g)$ arbitrarily to an element
$$
M^{\prime}\in \Mat_{r}(\Bcris(R/p)),
$$
simply by lifting all the entries in the matrix $\rho(g).$ Define inductively,
$$
C_{0}=M^{\prime}
$$
and
$$C_{m+1}=\rho(\delta(u_{0}))\mu^{\prime}(p)\varphi(C_{m})(\rho(U_{\cris})\mu^{\prime}(p))^{-1}= \rho(\delta(u_{0}))\mu^{\prime}(p)\varphi(C_{m})\mu^{\prime}(p)^{-1}\rho(U_{\cris}^{-1})\in \Mat_{r}(\Bcris(R/p)).$$
This is a well-defined element, as $U_{\cris}\in \G(\Acris(R/p)),$ so that $\rho(U_{\cris})\in \GL_{r}(\Acris(R/p))$ and $\mu^{\prime}(p)^{-1}\in \GL_{r}(W_{\mathbb{Q}}(k)).$
One finds an integer $N(\rho,T)=N>0,$ (depending on the choice of the representation $(\rho,T)$) such that all $\rho(\delta(u_{0}))\mu^{\prime}(p),\rho(U_{\cris})\mu^{\prime}(p),C_{0}$ are contained inside
$$
\Mat_{r}(p^{-N}\Acris(R/p)).
$$
Using again that conjugation by $\mu^{\prime}(p)$ leads to a linear amount of growth in denominators of $p,$ it follows that
$$
C_{m}\in \Mat_{r}(p^{-(N+cm)}\Acris(R/p)).
$$
Now the claim is that the sequence $\lbrace C_{m} \rbrace_{m}$ converges in $\Mat_{r}(\Bcris(R/p))$ for the topology induced by the Banach-space topology on $\Bcris(R/p)$ introduced in the beginning of this section.
\\
To show this, one has to see that $\lbrace C_{m+1}-C_{m}\rbrace_{m}$ is a zero-sequence for the norm $
|b|=\inf\lbrace |\lambda|_{\mathbb{Q}_{p}}\colon b\cdot \lambda \in \Acris(R/p)\rbrace
$, so that $\lbrace C_{m} \rbrace_{m}$ is Cauchy and it will indeed converge, as $\Bcris(R/p)$ is complete and separated in the topology induced by the norm $|.|$. Now calculate that
$$
C_{1}-C_{0}=\rho(\delta(u_{0}))\mu^{\prime}(p)\varphi(C_{0})(\mu^{\prime}(p))^{-1}(\rho(U_{\cris}))^{-1}-C_{0}\in \Mat_{r}(K_{\cris}\otimes_{\mathbb{Z}_{p}}\mathbb{Q}_{p}).
$$
Indeed, after applying $\chi_{\mathbb{Q}}$ one gets the equation
$$
\rho(u_{0}\mu(p)F(g)\mu(p)^{-1}U^{-1})-\rho(g)=0,
$$
by the assumption that $g\in G(W_{\mathbb{Q}}(R/p))$ is a $G$-quasi-isogeny. Now define $N_{1}:=\min_{i,j}(\nu_{p}((C_{1}-C_{0})_{i,j})).$ Then one has that
$$
C_{1}-C_{0}\in \Mat_{r}(p^{N_{1}}K_{\cris}).
$$
It follows that
$$
C_{2}-C_{1}=\rho(\delta(u_{0}))\mu^{\prime}(p)\varphi(C_{1}-C_{0})(\mu^{\prime}(p))^{-1}(\rho(U_{\cris}))^{-1}
$$
is contained inside of
$$
\Mat_{r}(p^{-c}\varphi(p^{N_{1}}K_{\cris})).
$$
By induction, it follows that
$$
C_{m+1}-C_{m}=\rho(\delta(u_{0}))\mu^{\prime}(p)\varphi(C_{m}-C_{m-1})(\mu^{\prime}(p))^{-1}(\rho(U_{\cris}))^{-1}
$$
is contained inside of
$$
\Mat_{r}(p^{-c}\varphi(p^{-cm}\varphi^{m}(p^{N_{1}}K_{\cris})))\subseteq \Mat_{r}(p^{-(m+1)c}\varphi^{m+1}(p^{N_{1}}K_{\cris})),
$$
here again the $c$ depends on the choice of the representation $(\rho,T).$ 
But above it was already explained, that $p^{-cm}\varphi^{m}(K_{\cris})\subseteq p^{\frac{p^{m}-1}{p-1}-cm}\Acris(R/p),$ and using again that 
$$
p^{-mc+N_{1}}\cdot p^{\frac{p^{m}-1}{p-1}}=p^{\frac{p^{m}-1}{p-1}-cm+N_{1}}\rightarrow \infty,
$$
for $m\rightarrow \infty,$ one sees that the sequence $\lbrace C_{m+1}-C_{m}\rbrace_{m}$ indeed converges to zero for the $p$-adic topology.
\\
It remains to see that
$$
g_{\cris}(\rho,T)=\lim_{m\rightarrow \infty}C_{m}\in \GL(T\otimes_{\mathbb{Z}_{p}} \Bcris(R/p)),
$$
because the equation of a crystalline quasi-isogeny and $\chi_{\mathbb{Q}}(g_{\cris})=\rho(g)$ are true by construction. Namely, here the discussion of the topological pathologies of the ring $W_{\mathbb{Q}}(R/p)$ becomes relevant; this discussion implies that to show that $\chi_{\mathbb{Q}}(g_{\cris})=\rho(g),$ one has to convince onself that $\chi_{\mathbb{Q}}(C_{m})=\rho(g).$ This is done by induction. First, choose $C_{0}$ to be a lift of $\rho(g)$ under $\chi_{\mathbb{Q}},$ so that $\chi_{\mathbb{Q}}(C_{0})=\rho(g).$ Assume that one already knows that $\chi_{\mathbb{Q}}(C_{m})=\rho(g),$ then one gets
$$
\chi_{\mathbb{Q}}(C_{m+1})=\rho(u_{0}\mu(p)F(g)\mu(p)^{-1}U)=\rho(g),
$$
where it was used that $\chi\circ \varphi=F\circ \chi,$ $\chi(U_{\cris})=U$ and that $g$ is a $G$-quasi-isogeny between $u_{0}$ and $U.$
\\
Next, one has to show that the $p$-adic limit $\lim_{m\rightarrow \infty}C_{m}$ is indeed invertible. To see this, lift the matrix $\rho(g^{-1})\in \Mat_{r}(W_{\mathbb{Q}}(R/p))$ similiarly to a matrix $N^{\prime}\in Mat_{r}(\Bcris(R/p)).$ Define inductively, $D_{0}=N^{\prime}$ and 
$$
D_{m}=\rho(U_{\cris})\mu^{\prime}(p)\varphi(D_{m-1})\mu^{\prime}(p)^{-1}\rho(\delta(u_{0}))^{-1}.
$$
Then one sees by the same arguments as before, that
$$
\lim_{m\rightarrow \infty} D_{m}\in \Mat_{r}(\Bcris(R/p))
$$
exists. I claim that
$$
\id_{r}-CD=0.
$$
For this, one shows that $\lbrace \id_{r}-C_{m}D_{m}\rbrace_{m}$ is a zero-sequence. Inductively one sees that
\begin{align*}
\id_{r}-C_{m+1}D_{m+1} & =\id_{r}-\rho(\delta(u_{0}))\mu^{\prime}(p)\varphi(C_{m}D_{m})\mu^{\prime}(p)^{-1}\rho(\delta(u_{0}))^{-1} \\
& =\rho(\delta(u_{0}))\mu^{\prime}(p)(\varphi(\id_{r}-C_{m}D_{m}))\mu^{\prime}(p)^{-1}\rho(\delta(u_{0}))^{-1}
\end{align*}
is contained inside of
$$
\Mat_{r}(p^{-c(m+1)}\varphi^{m+1}(K_{\cris}\otimes_{\mathbb{Z}_{p}}\mathbb{Q}_{p})).
$$
By the same argument as before, one sees that $\lbrace \id_{r}-C_{m}D_{m}\rbrace_{m}$ indeed converges to zero. This finishes the proof.
\subsection{Extension to the general case}
In the previous section I just considered the case of a $p$-torsion free integral perfectoid ring $R$; this was needed to insure that the semi-perfect ring $R/p$ is quasi-regular. Here I want to explain how to extend the translation of the quasi-isogeny towards the case of a general integral perfectoid ring. This will use the same ideas as in section \ref{Section extension to the general case} to reduce separatly to the case of a perfect ring in characteristic $p$ and to a $p$-torsion free integral perfectoid ring.
\\
Let me first adress the translation in the case of a perfect ring: let $R$ be a perfect $k$-algebra, $\mathcal{P}_{0}=(\mathcal{Q}_{0},\alpha_{0})$ an adjoint nilpotent $\G$-$\mu$-Display over $\Spec(k),$ $\mathcal{P}=(\mathcal{Q},\alpha)$ a $\G$-$\mu$-Display over $R;$ recall that one can also consider $\mathcal{P}_{0}$ resp. $\mathcal{P}$ as $\G$-torsors over $\Spec(W(k))$ resp. $\Spec(W(R))$ with additional structure. One is given a $G$-quasi-isogeny
$$
\rho\colon \mathcal{P}_{0}\times_{\Spec(k)}\Spec(R)\dashrightarrow \mathcal{P}.
$$
\begin{Lemma}\label{Uebersetzung Quasi-isogenie perfekter Fall}
The datum of $\rho$ is equivalent to the datum of an isomorphism of $G$-torsors
$$
\rho^{\prime}\colon \mathcal{P}_{0}\times_{\Spec(W(k))}\Spec(W(R)[1/p])\simeq \mathcal{P}\times_{\Spec(W(R))}\Spec(W(R)[1/p]),
$$
that is compatible with Frobenius-structure.
\end{Lemma}
\begin{proof}
This follows from the fact that the association $R\rightarrow R^{\prime}\mapsto \underline{\Isom}_{G}(\mathcal{P}_{0, W(R^{\prime})[1/p]},\mathcal{P}_{W(R^{\prime})[1/p]})$ is a sheaf for the étale topology. For this the argument is as in Lemma \ref{Descent fuer Isokristalle} but easier, using that if $R\rightarrow R^{\prime}$ is a faithfully flat étale morphism of perfect rings, then $W(R)\rightarrow W(R^{\prime})$ is $p$-completely faithfully flat étale and that if $R^{\prime ^{\bullet}}$ is the Cech-nerve of the covering $R\rightarrow R^{\prime},$ one has that $W(R^{\prime \bullet})\simeq W(R^{\prime})^{\bullet},$ where this is the simplicial object with $W(R^{\prime})\widehat{\otimes}_{W(R)}...W(R^{\prime})$ in degree $n:$ now one can again directly apply Drinfeld's descent result for 'Banachian'-modules. 
\end{proof}
Now one can turn to the general case: let $R$ be a non-necessarily $p$-torsion free integral perfectoid $W(k)$-algebra, and fix for later use a perfectoid pseudo-uniformizer $\varpi\in R$ dividing $p$. Let again $R_{1}=R/R[\sqrt{pR}],$ $R_{2}=R/\sqrt{pR}$ and $R_{3}=R_{1}/\sqrt{pR_{1}}.$ Then $R_{1}$ is $p$-torsion free integral perfectoid, while $R_{2}$ and $R_{3}$ are perfect. One then has a crucial pullback and pushout square with the canonical maps
$$
\xymatrix{
R \ar[r] \ar[d] & R_{1} \ar[d] \\
R_{2} \ar[r] & R_{3},
}
$$
where the morphisms $R_{1}\rightarrow R_{3} \leftarrow R_{2}$ are both surjective. Then one the following pullback square
$$
\xymatrix{
\Acris(R) \ar[r] \ar[d] & \Acris(R_{1}) \ar[d] \\
\Acris(R_{2}) \ar[r] & \Acris(R_{3}),
}
$$
where $\Acris(R_{1})\rightarrow \Acris(R_{3}) \leftarrow \Acris(R_{2})$ is still surjective. Inverting $p,$ one gets similiarly that
$$
\Bcris(R)\simeq \Bcris(R_{1})\times_{\Bcris(R_{3})} \Bcris(R_{2})
$$
and the surjectivity is still satisfied.\footnote{Here one has to use that inverting $p$ commutes with a fiber product, which follows because there is a natural morphism from $\Bcris(R)$ into $\Bcris(R_{1})\times_{\Bcris(R_{3})} \Bcris(R_{2})$ and for the bijectivity, which is a question about abelian groups, one uses that filtered colimits commute with finite limits.} Let as before $\mathcal{P}_{0}$ be an adjoint nilpotent $\G$-$\mu$-Display and $\mathcal{P}$ a $\G$-$\mu$-Display over $R.$ Suppose one is given a $G$-quasi-isogeny\footnote{Note that as $\sqrt{pR}=\bigcup_{n}(\varpi^{1/p^{n}}),$ the image of $\varpi$ in $R_{1}$ and $R_{3}$ is zero and therefore one may apply the case discussed in lemma \ref{Uebersetzung Quasi-isogenie perfekter Fall} to the base change of $\rho.$}
$$
\rho\colon \mathcal{P}_{0}\times_{\Spec(k)}\Spec(R/\varpi)\dashrightarrow \mathcal{P}\times_{\Spec(R)}\Spec(R/\varpi).
$$

Then the $\G$-$\mu$-Display $\mathcal{P}$ is adjoint nilpotent and by Prop. \ref{Deszent von Acris nach Ainf} corresponds to a $\G$-$\mu$-window for the frame $\mathcal{F}_{\infintesimal}(R),$ which then gives rise to a $\G$-Breuil-Kisin-Fargues modulo $(\mathcal{P},\varphi_{\mathcal{P}})$ over $\Spec(\Ainf(R)),$ which is of type $\mu.$ Finally, let $\mathcal{P}_{\cris}=\mathcal{P}\times_{\Spec(\Ainf(R))}\Spec(\Acris(R))$ be the base-change, with Frobenius-structure $\varphi_{\mathcal{P}_{\cris}}.$ 
\begin{Lemma}\label{Uebersetzung der Qis allgemeiner Fall}\footnote{For peace of mind, I recall that $\Acris(R)=\Acris(R/p)=\Acris(R/\varpi),$ since $R^{\flat}=\lim_{\text{Frob}}R/pR=\lim_{\text{Frob}}R/\varpi R.$}
There is a uniquely determined isomorphism of $G$-torsors over $\Spec(\Bcris(R))$
$$
\rho_{\cris}\colon \mathcal{P}_{0}\times_{\Spec(W(k))}\Spec(\Bcris(R))\simeq \mathcal{P}_{\cris}\times_{\Spec(\Acris(R))}\Spec(\Bcris(R)),
$$
which is compatible with the Frobenius-structure and that lifts the given $G$-quasi-isogeny $$\rho\colon \mathcal{P}_{0}\times_{\Spec(k)}\Spec(R/p)\dashrightarrow \mathcal{P}\times_{\Spec(R)}\Spec(R/p).
$$
\end{Lemma}
\begin{proof}
Consider $X=\underline{\Isom}_{G}(\mathcal{P}_{0, \Bcris(R)},\mathcal{P}_{\cris \Bcris(R)});$ as $G$ is smooth and affine, this functor is representable by a flat and affine scheme over $\Spec(\Bcris(R)).$ Then one is in a situation, where one can use \cite[Prop. 2.2.15]{BouthierCesnavicius}, to deduce
\begin{equation}
X(\Bcris(R))\simeq X(\Bcris(R_{1}))\times_{X(\Bcris(R_{3}))} X(\Bcris(R_{2})).
\end{equation}
As explained before, one can use the given $G$-quasi-isogeny $\rho,$ to produce uniquely determined elements $x,y\in X(\Bcris(R_{1})) \times X(\Bcris(R_{2})),$ which have the property that mapping towards $\Bcris(R_{3}))$ they both agree, i.e. $z=(x,y)\in X(\Bcris(R)).$ This isomorphism $z$ is still compatible with Frobenius structures and gives then the desired lift.
\end{proof}
\section{Applications to the Bültel-Pappas moduli problem}\label{Applications to the Bueltel Pappas}
\subsection{Recollections on the Bültel-Pappas moduli problem}\label{Subsection Recollections on the Bueltel-Pappas moduli problem}
Let $(G,\lbrace \mu \rbrace, [b])$ be an unramified local Shimura-datum over $\mathbb{Q}_{p},$ (c.f. \cite[Def. 24.1.1.]{Berkeleylectures}, unramified means here just that the reductive group $G$ over $\mathbb{Q}_{p}$ is unramified, i.e. is quasi-split and split after an unramified extension) and $E$ the local Shimura-field; an unramified extension of $\mathbb{Q}_{p}.$ Write as usual $\breve{E}$ for the completion of the maximal unramified extension of $E$ and denote by $k$ its residue-field. Let $\G$ be a reductive group scheme over $\mathbb{Z}_{p},$ which is an integral model of $G.$ One may and do assume that one finds a representative $\mu\in \lbrace \mu \rbrace,$ which extends integrally, i.e. is given by a minuscule cocharchter
$$
\mu\colon \mathbb{G}_{m,\mathcal{O}_{E}}\rightarrow \G_{\mathcal{O}_{E}}.
$$
Here one use the assumption that one is working in the unramified case.
Then its $\G(\mathcal{O}_{E})$-conjugation class is well-defined (c.f. the discussion in \cite[section 4.1.1.]{BueltelPappas}).
Let $K=\G(\mathbb{Z}_{p})\subseteq G(\mathbb{Q}_{p})$ and $\Sh_{K}$ the local Shimura variety with maximal hyperspecial level structure $K$ associated to $(G,\lbrace \mu \rbrace, [b],K)$ by Scholze, \cite[Def. 24.1.3]{Berkeleylectures}.
\\
Next, I want to consider the moduli problem of adjoint nilpotent $\G$-$\mu$-displays up to isogeny; to have this well-defined I will add the following further assumptions to the present set-up: I assume that the $\sigma$-conjugacy $[b]\in B(G,\mu)$ has a representative $b=u_{0}\mu(p),$ for some $u_{0}\in \G(W(k))$ (in total one considers an integral unramified local Shimura-datum as in \cite[Def. 4.1.2]{BueltelPappas}) and then one requires that it satisfies the adjoint nilpotency condition, i.e. the isocrystal which is constructed from $b$ using the adjoint representation has all slopes $>-1.$ 
\\
Let then $\mathcal{M}^{\BP}=\mathcal{M}^{\BP}(\G,\lbrace \mu \rbrace, [b])$ be the moduli problem of adjoint nilpotent $\G$-$\mu$-displays up to isogeny, that was introduced by Bültel-Pappas in \cite[section 4.2.]{BueltelPappas}. Furthermore, denote by $\mathcal{P}_{u_{0}}$ the banal $\G$-$\mu$-display over $k$ constructed from $u_{0}.$ Recall the following conjecture put forward by Bültel-Pappas.
\begin{Conjecture}{(Bültel-Pappas \cite[Conjecture 4.2.1.]{BueltelPappas})}\label{Darstellbarkeitsvermutung BP}
The functor
$$
\mathcal{M}^{\BP}\colon \Nilp_{W}^{\text{op}}\rightarrow \Set,
$$
is representably by a formal scheme, which is locally formally of finite type and formally smooth over $\Spf(\mathcal{O}_{\breve{E}}).$
\end{Conjecture}
Concerning the special fiber, after the work of Bhatt-Scholze \cite{BhattScholzeWitt} on the projectivity of the Witt vector affine Grassmannian, one knows that the restriction of $\mathcal{M}^{\BP}$ towards perfect $k$-schemes is representable by a perfect scheme, locally perfectly of finite type. In the rest of this section I want to work in the direction of putting also a geometric structure on the generic fiber of this moduli problem. 
\\
To explain what is meant by this generic fiber, I borrow from Scholze-Weinstein \cite[section 2.2]{ScholzeWeinstein} a general procedure how to associate to a functor like $\mathcal{M}^{\BP}$ defined on $\Nilp_{\mathcal{O}_{\breve{E}}}$ a functor defined on complete affinoid Huber-pairs $\text{CAffd}_{\Spa(\breve{E},\mathcal{O}_{\breve{E}})}.$
\\
First define a pre-sheaf on $\text{CAffd}_{\Spa(\breve{E},\mathcal{O}_{\breve{E}})}$ as follows: if $(R,R^{+})\in \text{CAffd}_{\Spa(\breve{E},\mathcal{O}_{\breve{E}})},$ then let
$$
\mathcal{M}^{\BP}_{\eta, \text{pre}}(R,R^{+})=\underset{{R_{0}\subseteq R^{+}}} \colim \mathcal{M}^{\BP}(\Spf(R_{0})),
$$
where the colimit runs over all open and bounded $\mathcal{O}_{\breve{E}}$-subalgebras of $R^{+}.$ The point is here that open and bounded $\mathcal{O}_{\breve{E}}$-subalgebras of $R^{+}$ will have the $p$-adic topology and one can evaluate
$$
\mathcal{M}^{\BP}(\Spf(R_{0}))=\lim_{n}\mathcal{M}^{\BP}(\Spec(R_{0}/p^{n})).
$$ Now equip the oppposite category $\text{CAffd}^{\text{op}}_{\Spa(\breve{E},\mathcal{O}_{\breve{E}})}$ with the structure of a site, where one declares the coverings to be generated by coverings for the analytic topology on $\Spa(R,R^{+});$ more precisely, a family of morphisms
$$
(R,R^{+})\rightarrow (R_{i},R_{i}^{+})
$$
such that $(R_{i},R_{i}^{+})=(\mathcal{O}_{X}(U_{i}),\mathcal{O}_{X}^{+}(U_{i}))$ for a covering $\lbrace U_{i} \rbrace_{i\in I}$ of $X=\Spa_{\text{top}}(R,R^{+})$ is declared to a covering in $\text{CAffd}^{\text{op}}_{\Spa(\breve{E},\mathcal{O}_{\breve{E}})}$ (c.f. \cite[Def. 2.1.5]{ScholzeWeinstein}). Then by definition
$$
\mathcal{M}^{\BP}_{\eta}\colon \text{CAffd}^{\text{op}}_{\Spa(\breve{E},\mathcal{O}_{\breve{E}})}\rightarrow \Set
$$
is the analytic sheafification of the pre-sheaf $\mathcal{M}^{\BP}_{\eta, \text{pre}}.$ The Proposition 2.2.2. in loc.cit. makes sure that if the above conjecture, Conjecture \ref{Darstellbarkeitsvermutung BP}, would be satisfied, it would follow that $\mathcal{M}^{\BP}_{\eta}$ is representable by the usual Berthelot generic fiber of $\mathcal{M}^{\BP}.$ Now the aim of this section can be stated more precisely as showing directly that $\mathcal{M}^{\BP}_{\eta}$ is representable by the local Shimura-variety $\Sh_{K}.$
\subsection{Representability of the Diamond}\label{subsection Representability of the Diamond}
Recall that previously I considered the following sheaf for the analytic topology on complete affinoid adic spaces over $\Spa(\breve{E}):$
$$
\mathcal{M}^{\BP}_{\eta}\colon \text{CAffd}^{\text{op}}_{\Spa(\breve{E})}\rightarrow \Set.
$$
Now I will consider a diamantine version of it; let $\Perf_{k}$ be the category of affinoid perfectoid spaces over $k.$ Then $(\mathcal{M}^{\BP}_{\eta})^{\lozenge}\rightarrow \Spd(\breve{E})$ is the sheaf for the analytic topology on affinoid perfectoid spaces over $k:$
$$
(\mathcal{M}^{\BP}_{\eta})^{\lozenge}\colon \text{AffdPerf}^{\text{op}}_{k}\rightarrow \Set,
$$
which sends $S=\Spa(R,R^{+})\in \Perf_{k},$ towards equivalence classes of pairs $((S^{\sharp},\iota),x),$ where $S^{\sharp}$ is an affinoid perfectoid space over $\breve{E},$ $\iota\colon (S^{\sharp})^{\flat}\simeq S$ and $x\in \mathcal{M}^{\BP}_{\eta}(S^{\sharp}).$ Two points $((S^{\prime \sharp},\iota^{\prime}),x^{\prime})$ and $((S^{\sharp},\iota),x)$ are equivalent, if there is an isomorphism $f\colon S^{\sharp}\simeq S^{\prime \sharp},$ such that $\iota^{\prime}\circ f^{\sharp}=\iota$ and $f^{*}(x^{\prime})=x\in \mathcal{M}^{\BP}_{\eta}(S^{\sharp}).$
\begin{Remark}
If $\mathcal{M}^{\BP}$ would be representable by a formal scheme, locally formally of finite type over $\Spf(\mathcal{O}_{\breve{E}}),$ then $\mathcal{M}^{\BP}_{\eta}$ would describe the functor of points of its Berthelot generic fiber and $(\mathcal{M}^{\BP}_{\eta})^{\lozenge}$ the functor of points of the diamond associated to this generic fiber.
\end{Remark}
The basic aim of this sub-section is then to show the following
\begin{proposition}\label{Vergleich der Diamanten}
There is an isomorphism of analytic sheaves on $\text{AffdPerf}^{\text{op}}_{k}$
$$
\psi\colon (\mathcal{M}^{\BP}_{\eta})^{\lozenge}\simeq (\Sh_{K})^{\lozenge},
$$
over $\Spd(\breve{E}).$
\end{proposition}
\begin{Remark}
Since $(\Sh_{K})^{\lozenge}$ is already known to be a sheaf for the much finer $v$-topology on $\Perf_{k},$ (c.f. \cite{ScholzeDiamonds}, Prop. 11.9.) this in particular shows that also $(\mathcal{M}^{\BP}_{\eta})^{\lozenge}$ is one.
\end{Remark}
\begin{proof}
One first uses the previous results to obtain that a $S=\Spa(R,R^{+})$-valued point of $(\mathcal{M}^{\BP}_{\eta})^{\lozenge}$ can be concretely described as follows: one gives an untilt $S^{\sharp}=\Spa(R^{\sharp},R^{\sharp +})$ over $\breve{E},$ a covering by rational opens $\bigcup_{i\in I} U_{i}^{\sharp}=S^{\sharp},$ $U_{i}^{\sharp}=\Spa(R^{\sharp}_{i},R^{\sharp +}_{i}),$ a $\G$-Breuil-Kisin-Fargues-module $(\mathcal{M}_{i},\varphi_{\mathcal{M}_{i}})$ over $\Spec(R^{\sharp +}_{i})$ of type $\mu,$ and a crystalline $G$-quasi-isogeny
$$
\rho_{i}\colon \mathcal{P}_{u_{0}}\times_{\Spec(W(k))}\Spec(\Bcris(R^{\sharp +}_{i}/p))\simeq \mathcal{M}_{i} \times_{\Spec(\Ainf(R^{\sharp +}_{i}))} \Spec(\Bcris(R^{\sharp +}_{i}/p)).
$$
This data $((\mathcal{M}_{i},\varphi_{\mathcal{M}_{i}}),\rho_{i})$ is required to be isomorphic on overlaps $U_{i,j}^{\sharp}=U_{i}^{\sharp} \cap U_{j}^{\sharp}$ and to satisfy a cocycle condition on triple overlaps $U_{i,j,k}^{\sharp}=U_{i}^{\sharp} \cap U_{j}^{\sharp} \cap U_{k}^{\sharp}.$ 
\\
The precise statements that have been used are the following: the translation between $\G$-$\mu$-displays over integral perfectoid rings and Breuil-Kisin-Fargues modules, Corollary \ref{Korollar Vergleich Displays und BKF}\footnote{I use here that a $\G$-Breuil-Kisin-Fargues module equipped with a crystalline quasi-isogeny towards $\mathcal{P}_{u_{0}}$ lies in the essential image; for this one may reduce to the banal case and one has to see that the banal $\G$-$\mu$-window for the crystalline frame that corresponds to the trivial $\G$-BKF of type $\mu$ satisfies the adjoint nilpotency condition as in Definition \ref{Definition adjoint nilpotent fuer kristalline Windows}. But this follows from the existence of the crystalline $G$-quasi isogeny together with the running assumption on the slopes of $b.$}; together with the translation of the $G$-quasi-isogeny, Lemma \ref{Lemma Vergleich der Quasi-isogenie kristalline Qis, p torsionfreier Fall} (and that $R^{+ \sharp}_{i}$ is open bounded of course and since $\Spa(R^{\sharp},R^{+ \sharp})$ is an untilt over $\Spa(\breve{E})$ one may work with $p$ as the pseudo-uniformizer, so that the $R^{+ \sharp}_{i}$ will carry the $p$-adic topology; furthermore note that all $R^{\sharp +}_{i}$ will be $p$-torsion free as subrings of $R^{\sharp}_{i},$ where $p$ is invertible. Therefore the running assumptions in the translation statements between $\G$-$\mu$-displays and $\G$-Breuil-Kisin-Fargues modules of type $\mu$ are met). 
\\
To construct the map of analytic sheaves over $\Spd(\breve{E})$
$$
\psi\colon  (\mathcal{M}^{\BP}_{\eta})^{\lozenge}\simeq (\Sh_{K})^{\lozenge}
$$
one can restrict the Breuil-Kisin-Fargues modules, which are torsors over $\Spec(W(R_{i}^{+})),$ towards $$\Spa(W(R_{i}^{+}))-V([\varpi])=\mathcal{Y}_{[0,\infty)}(\Spa(R_{i},R_{i}^{+})),$$ to obtain a local shtuka over $U_{i}$ with a leg at $U_{i}^{\sharp}$ and singularities bounded by $\mu.$ Using the inclusions
$$
H^{0}(\mathcal{Y}_{[1/p^{p},1]},\mathcal{O})\subset \Bcris(R_{i}^{+}/p) \subset H^{0}(\mathcal{Y}_{[1/p^{p-1},1]},\mathcal{O}),
$$
one sees that the quasi-isogeny $\rho_{i}$ then translates as required to the trivialization of the local shtuka near infinity. Now analytic descent for local shtuka produces the desired point of $(\Sh_{K})^{\lozenge}.$
\subsubsection{Injectivity of the morphism $\psi$}
First, I want to explain why the morphism $\psi$ is injective; this will use the following statement:
\begin{Lemma} Let $A$ be an integral perfectoid $k$-algebra with perfectoid pseudo-uniformizer $\varpi\in A$ (in particular; it is assumed that this element is regular).
\begin{enumerate}
\item[(a):] There is an equivalence of categories between algebraic $\G$-torsors over $\Spec(W(A))-V(p,[\varpi])$ and analytic $\G^{\adic}$-torsors over $\Spa(W(A),W(A))-V(p,[\varpi]),$ where one equips $W(A)$ as usual with the $(p,[\varpi])$-adic topology.
\item[(b):] The restriction functor from $\G$-torsors over $\Spec(W(A))$ towards $\G$-torsors over $\Spec(W(A))-V(p,[\varpi])$ is fully faithful.
\end{enumerate}
\end{Lemma}
\begin{proof}
The statement (a) follows from a result of Kedlaya \cite[Thm. 3.8]{KedlayaAinf}, which tells us that there is an exact tensor equivalence between the category of algebraic vector bundles on $$\Spec(W(A))-V(p,[\varpi])$$ and analytic vector bundles on $\Spa(W(A))-V(p,[\varpi]).$ Now one may use the Tannakian formalism on both sides, to obtain the desired statement (see \cite[Thm. 19.5.1]{Berkeleylectures}, for the algebraic side and Thm. 19.5.2. for the analytic side).
\\
For statement (b), first observe that the restriction functor from vector bundles on $\Spec(W(A))$ towards vector bundles on $\Spec(W(A))-V(p,[\varpi])$ is fully faithful: By passing to inner-Hom's it is enough to see that
$$
H^{0}(\Spec(W(A))-V(p,[\varpi]),\mathcal{O})=W(A),
$$
which follows from the fact that $(p,[\varpi])$ form a regular sequence in $W(A);$ note that more generally this shows that for any vector bundle $\mathcal{E}$ on $\Spec(W(A))$ one has that
\begin{equation}\label{Ausdehnen von Schnitten ueber offene mit Komplement regulaere Sequenz}
H^{0}(\Spec(W(A)),\mathcal{E})=H^{0}(\Spec(W(A))-V(p,[\varpi]),\mathcal{E}\!\mid_{\Spec(W(A))-V(p,[\varpi])}).
\end{equation}
Now consider two $\G$-torsors $\mathcal{P},\mathcal{P}^{\prime}$ on $\Spec(W(A)).$ Then the functor of isomorphisms,
$$
Y=\underline{\Isom}_{\G}(\mathcal{P},\mathcal{P}^{\prime})
$$
is representable by an affine scheme over $\Spec(W(A)),$ which is of finite presentation. It is then enough to see that sections of $Y$ restricted to $\Spec(W(A))-V(p,[\varpi])$ extend uniquely to sections over $\Spec(W(A)).$ Taking a closed embedding $Y\hookrightarrow \mathbb{A}^{n}_{W(A)},$ one sees from (\ref{Ausdehnen von Schnitten ueber offene mit Komplement regulaere Sequenz}) that a $\Spec(W(A))-V(p,[\varpi])$-section of $Y$ extends uniquely to a $\Spec(W(A))$-section of $\mathbb{A}^{n}_{W(A)},$ which by the uniqueness has to factor over $Y.$\footnote{Namely, if $Y=V(f_{1},...,f_{k}),$ and if $s_{1},...,s_{n}\in W(A)$ is the extended section, then one has to see that $f_{j}(s_{1},...,s_{n})=0$ for all $1\leq j \leq k.$ But the zero section and the sections $f_{j}(s_{1},...s_{n})$ of $\mathbb{A}^{1}_{W(A)}\rightarrow \Spec(W(A))$ agree on $\Spec(W(A))-V(p,[\varpi]),$ so that by uniqueness again, one deduces that $f_{j}(s_{1},...,s_{n})=0.$}
\end{proof}
To conclude the verification of the injectivity of $\psi$, take two points $(\mathcal{M},\varphi_{\mathcal{M}},\rho)$ and $(\mathcal{M}^{\prime},\varphi_{\mathcal{M}^{\prime}},\rho^{\prime}),$ over some $S=\Spa(R,R^{+})\in \Perf_{k},$ where $(\mathcal{M}^{?},\varphi_{\mathcal{M}^{?}})$ are $\G$-BKFs over $R^{+ \sharp}$ of type $\mu$ and $\rho^{?}$ are crystalline $G$-quasi-isogenies, such that their images under $\psi$ agree. It follows in particular, that the $\G^{\text{adic}}$-torsors over $\Spa(W(R^{+}))-V(p,[\varpi]),$ which one constructs from glueing the shtuka over $\mathcal{Y}_{[0,\infty)}(S)$ with the trivial $\G^{\text{adic}}$-torsor over $\mathcal{Y}_{[0,\infty]}(S)$ constructed from $b$ along the trivializations $\iota$ resp. $\iota^{\prime}$ are isomorphic. By part (a) of the above lemma their algebraic counterparts are isomorphic and then by part (b) also $\mathcal{M}$ and $\mathcal{M}^{\prime}$ have to be isomorphic (which gives an isomorphism respecting the Frobenius-structure and the quasi-isogeny).\footnote{The fact that one carries along the datum of the quasi-isogeny saves one here from the objection that the restriction from Breuil-Kisin-Fargues modules towards shtuka is not fully faithfull.}
\subsubsection{Surjectivity of the morphism $\psi$}
The only interesting point is therefore the surjectivity, where the claim is basically that one may extend $\G$-torsors over $\Spec(W(R^{+}))-V(p,[\varpi])$ after passing to a covering by rational opens of $\Spa(R,R^{+})$ towards the whole spectrum.
\\
Start with an $S=\Spa(R,R^{+})$-valued point of $(\Sh_{K})^{\lozenge}.$ It is given by a triple $(S^{\sharp},((\mathcal{E},\varphi_{\mathcal{E}}),\iota))$ as before. The untilt $S^{\sharp}=\Spa(R^{\sharp},R^{+ \sharp})$ gives in particular rise to a degree one primitive element $\xi\in W(R^{+}).$
\\
Now one may proceed as before and use the trivialization near infinity, given by $\iota,$ to glue the shtuka $(\mathcal{E},\varphi_{\mathcal{E}})$ over $\mathcal{Y}_{[0,\infty)}(S)$ and the trivial torsor over $\mathcal{Y}_{[r,\infty]}(S),$ with Frobenius-equivariance given by $b_{0}=u_{0}\mu(p),$ to obtain a $\G^{\text{adic}}$-torsor over $\Spa(W(R^{+}))-V(p,[\varpi]),$ which corresponds then to a $\G$-torsor over $\Spec(W(R^{+}))-V(p,[\varpi]);$ I will denote it by $\mathcal{E}^{\text{approx}}.$ It is this $\G$-torsor, which one wants to extend.
\\
Let me first explain why for the proof of the surjectivity of the morphism $\psi,$ it is indeed enough to check that this $\G$-torsor extends to whole spectrum after passing to a rational covering of $\Spa(R,R^{+}).$ The observations are simply the following: that on the one hand the given Frobenius-structure on the $\G^{\text{adic}}$-shtuka $(\mathcal{E},\varphi_{\mathcal{E}})$ induces a Frobenius-structure
$$
\varphi_{\mathcal{E}^{\text{approx}}}\colon \varphi^{*}(\mathcal{E}^{\text{approx}})[1/\xi]\simeq \mathcal{E}^{\text{approx}}[1/\xi].
$$
Here I wrote $\varphi$ for the morphism induced on $\Spec(W(R^{+}))-V(p,[\varpi])$ via the Witt vector Frobenius $\varphi$ on $W(R^{+})$ and the notation $\mathcal{E}^{\text{approx}}[1/\xi]$ is a short hand notation for restricting the $\G$-torsor $\mathcal{E}^{\text{approx}}$ over $\Spec(W(R^{+}))-V(p,[\varpi])$ towards the open subscheme, which is the complement of $$V(\xi)\cap \Spec(W(R^{+}))-V(p,[\varpi]).$$ This behaviour was insured in the definition of a mixed-characteristic shtuka, because there one requires that the Frobenius-structure $\varphi_{\mathcal{E}}$ is meromorphic along the closed Cartier divisor $S^{\sharp}\hookrightarrow \mathcal{Y}(S)$ (c.f. \cite[Def. 23.1.1]{Berkeleylectures}). On the other hand, the trivialization
$$
\iota\colon \mathcal{E}\!\mid_{\mathcal{Y}_{[r,\infty)}(S)}\simeq \mathcal{E}_{b}\!\mid_{\mathcal{Y}_{[r,\infty)}(S)},
$$
for $r$ large enough, gives naturally rise to a morphism of locally ringed spaces 
$$
\mathcal{Y}_{[r,\infty)}(S)\rightarrow \mathcal{E}^{\text{approx}}\times_{\Spec(W(R^{+}))-V(p,[\varpi])}D(p),
$$
here $D(p)= \Spec(W(R^{+})[1/p])\subset \Spec(W(R^{+}))-V(p,[\varpi]).$
Since $\mathcal{E}^{\text{approx}}\rightarrow \Spec(W(R^{+}))-V(p,[\varpi])$ is a torsor under an affine group scheme, one deduces that this morphism of locally ringed spaces is equivalent to a morphism of rings of global sections. Using the Frobenius pullback trick as in \cite[Prop. 22.1.1.]{Berkeleylectures}, one can make sure that one can use the inclusions
$$
H^{0}(\mathcal{Y}_{[1/p^{p},1]},\mathcal{O})\subset \Bcris(R^{+}/p) \subset H^{0}(\mathcal{Y}_{[1/p^{p-1},1]},\mathcal{O}),
$$
to construct the desired section $\Spec(\Bcris(R^{+}/p))\rightarrow \mathcal{E}^{\text{approx}};$ using compatibility with Frobenius structures, this gives a crystalline $G$-quasi-isogeny.
\\
Now one turns to the problem of extending the $\G$-torsor $\mathcal{E}^{\text{approx}}.$ I will need some properties of henselian pairs, which I want to collect here: Let $A$ be a commutative ring and $I\subset A$ be an ideal. Recall (\cite[Tag 09XE]{stacks}) that the pair $(A,I)$ is called henselian, if $I$ is contained in the Jacobson radical of $A$ and whenever $f\in A[T]$ is a monic polynomial with factorization 
$$
\overline{f}=g_{0}\cdot h_{0},
$$
where $g_{0},h_{0}\in A/I[T]$ are monic polynomials generating the unit ideal in $A/I[T],$ then there exists a factorization $f=g\cdot h,$ with $g,h\in A[T]$ monic and $\overline{g}=g_{0}$ and $\overline{h}=h_{0}.$ See \cite[Tag 09XI]{stacks} for equivalent conditions. If the ring $A$ is $I$-adically complete, the pair $(A,I)$ is henselian (\cite[Tag 0ALJ]{stacks}). Henselian pairs are closed under filtered (!)\footnote{This is wrong without this condition; for example for co-products there is a counter-example due to Moret-Baily \cite[Tag 0FWU]{stacks}.} colimits (\cite[Tag 0FWT]{stacks}), direct limits (\cite[Tag 0EM6]{stacks}), under integral morphisms (e.g. surjections) in the sense that if $(A,I)$ is a henselian pair and $A\rightarrow A^{\prime}$ is integral, then $(A^{\prime},IA^{\prime})$ is henselian (\cite[Tag 09XK]{stacks}) and under nilpotent thickenings.
\begin{Lemma}\label{Lemma Bouthier Cesnavicius}
Let $x\in S=\Spa(R,R^{+})\in \Perf_{k}$ be a point, $\kappa(x)$ the completed residue field at $x$ and $\kappa(x)^{+}\subseteq \kappa(x)$ the corresponding valuation subring. Then one has that
\begin{enumerate}
\item[(a):] $$\underset{x\in U}\colim\text{ }H^{1}(\Spec(W(\mathcal{O}^{+}_{S}(U))),\G)\simeq H^{1}(\Spec(W(\kappa(x)^{+})),\G),$$
\item[(b):] $$\underset{x\in U}\colim\text{ }H^{1}(\Spec(W(\mathcal{O}^{+}_{S}(U)))-V(p,[\varpi]),\G)\simeq H^{1}(\Spec(W(\kappa(x)^{+})-V(p,[\varpi]),\G).$$
Here both colimits run over rational opens $U\subseteq S,$ which contain $x\in S$ and the isomorphism is induced via pullback towards $\Spec(W(\kappa(x)^{+}))$ resp. $\Spec(W(\kappa(x)^{+}))-V(p,[\varpi]).$
\end{enumerate}
\end{Lemma}
\begin{proof}{\textit{( of the Lemma)}}
Recall that $\kappa(x)=\widehat{\Frac(R/\text{supp}(x))}_{|.|_{x}}$ and that $\kappa(x)^{+}$ is the valuation ring corresponding to the valuation $x.$ If $\mathcal{O}^{+}_{S,x}=\colim_{x\in U}\mathcal{O}^{+}_{S}(U)$ is the un-completed stalk of $\mathcal{O}^{+}_{S}$ at $x\in S,$ it is true that
\begin{equation}\label{Vervollstaendigung der Plus-Komponente}
\widehat{(\mathcal{O}^{+}_{S,x})}_{\varpi}\simeq \kappa(x)^{+}.
\end{equation}
Now consider $A=\colim_{x\in U}W(\mathcal{O}^{+}_{S}(U)).$  The first observation is that if $A_{1}=\widehat{A}_{[\varpi]}$ and $A_{2}=\widehat{A_{1}}_{p},$ then one has that
$$
A_{2}\simeq W(\kappa(x)^{+}).
$$
Namely, since $W(\kappa(x)^{+})$ is $(p,[\varpi])$-adically complete and separated (thus also $p$-adic and $[\varpi]$-adic), the natural map $A\rightarrow W(\kappa(x)^{+})$ induces a morphism as above. Now observe first that $p$ is regular in $A_{2}$ and in $W(\kappa(x)^{+}):$ for $W(\kappa(x)^{+})$ this is clear, since $\kappa(x)^{+}$ is perfect in characteristic $p,$ and for $A_{2}$ it suffices to see that $p$ is regular in $A_{1}$ (since completing along a regular element keeps the property that this element was regular). Observe now that $[\varpi]\in A$ is regular. Indeed, by exactness of filtered colimits, it is enough to check that $[\varpi]\in W(\mathcal{O}^{+}_{S}(U))$ is regular. Then one uses the following small statement: \begin{Lemma}
Let $A$ be a ring, $f,g\in A$ non-zero elements, such that $A$ is $g$-adic and $g\in A$ is regular. If $\overline{f}\neq 0\in A/g$ is regular, also $f\in A$ is regular.
\end{Lemma}
\begin{proof}
Let $x\in A,$ such that $f\cdot x=0.$ One has to see that $x=0.$ Modulo $g$ this reads as $\overline{f}\cdot \overline{x}=0$ in $A/g.$ By assumption that $\overline{f}\neq 0$ is regular, this implies that $x\in gA,$ i.e. $x=g\cdot x^{\prime},$ for some $x^{\prime}\in A.$ It follows that $g\cdot (f\cdot x^{\prime})=f\cdot x=0$ and, since $g$ is regular in $A,$ that $f\cdot x^{\prime}=0.$ As before, one gets $x^{\prime}\in gA.$ Continuing in this fashion, one sees $x$ is infinitely $g$ divisible, which implies that $x=0.$
\end{proof}
Then also $[\varpi]\in A_{1}$ is regular and it suffices to show that $p$ is regular in $A/[\varpi],$ by the previous lemma. Observe that $\varpi$ is regular in $A/pA=\colim_{U}\mathcal{O}^{+}_{S}(U),$ since this element was by assumption a perfectoid pseudo-uniformizer in $(R,R^{+}),$ thus also in all $(\mathcal{O}_{S}(U),\mathcal{O}^{+}_{S}(U)).$ Now the torsion exchange lemma tells one that $p$ is regular in $A/[\varpi]A,$ as desired.\footnote{This lemma says the following: If $A$ is some commutative ring, $f,g\in A$ are regular, then $(A/f)[g]\simeq (A/g)[f].$ This can be seen by considering the first homology of the Koszul-complex $K^{\bullet}(A;f,g)\simeq K^{\bullet}(A;g,f).$} 
\\
The fact that $p$ is a non-zero divisor in $A_{2}$ will now be used as follows: since $A_{2}$ is by construction $p$-adic, it suffices to show that the morphism $A_{2}\rightarrow W(\kappa(x)^{+})$ is an isomorphism modulo $p.$ Just to be on the same page, here is the statement one uses: 
\begin{Lemma}
Let $\psi\colon A\rightarrow B$ be a ring homomorphism and let $f\in A$ such that $A$ is $f$-adic and $f$-torsion free and $B$ is $\psi(f)$-adic and $\psi(f)$-torsion free. If $\psi$ is an isomorphism modulo $f,$ then $\psi$ is an isomorphism.
\end{Lemma}
\begin{proof}
The morphism $\psi$ is surjective by Nakayama's lemma. Consider $\ker(\psi);$ this ideal is derived $g$-complete. Then it suffices to see that $\ker(\psi)\otimes_{A}A/g=0.$ If one knows that tensoring with $A/g$ the exact sequence
$$
\xymatrix{
0 \ar[r] & \ker(\psi) \ar[r] & A \ar[r] & B \ar[r] & 0
}
$$
one gets the exact sequence
$$
\xymatrix{
0 \ar[r] & \ker(\psi) \otimes_{A}A/g \ar[r] & A \otimes_{A} A/g \ar[r] & B \otimes_{A} A/g \ar[r] & 0,
}
$$
it would follow that $\ker(\psi) \otimes_{A}A/g=0,$ as desired. The obstruction to this sequence being exact is $\text{Tor}^{1}_{A}(B,A/g).$ Tensoring the exact sequence
$$
\xymatrix{
0 \ar[r] & A \ar[r]^{\cdot f} & A \ar[r] & A/f \ar[r] & 0
}
$$
with $B$ over $A,$ one sees that $\text{Tor}^{1}_{A}(B,A/g)$ identifies with the $g$-torsion in $B.$ By assumption it therefore vanishes.
\end{proof}
One is therefore left with showing that
$$
A_{2}/pA_{2}\simeq W(\kappa(x)^{+})/pW(\kappa(x)^{+})
$$
is an isomorphism.
For this it then suffices to see that 
$$
\widehat{A}_{[\varpi]}/p\widehat{A}_{[\varpi]}\simeq \widehat{(A/pA)}_{[\varpi]}
$$
by the formula (\ref{Vervollstaendigung der Plus-Komponente}) above.
To see this, one uses that $p\in A$ is regular, to obtain the following triangle
$$
\xymatrix{
A \ar[r]^{\cdot p} & A \ar[r] & A/pA \ar[r]^{+1} &.
}
$$
Applying the derived $[\varpi]$-adic completion, the following triangle
$$
\xymatrix{
\widehat{A}^{\mathbb{L}}_{[\varpi]} \ar[r]^{\cdot p} & \widehat{A}^{\mathbb{L}}_{[\varpi]} \ar[r] & \widehat{(A/pA)}^{\mathbb{L}}_{[\varpi]} \ar[r] &
}
$$
follows.
Since $A$ is $[\varpi]$-torsion free, one obtains $\widehat{A}^{\mathbb{L}}_{[\varpi]}\simeq \widehat{A}_{[\varpi]};$ similiarly $A/pA$ is $\varpi$-torsion free, so that one gets $\widehat{(A/pA)}^{\mathbb{L}}_{[\varpi]}\simeq \widehat{(A/pA)}_{[\varpi]}.$ Furthermore, recall that $p$ is regular in $\widehat{A}_{[\varpi]}.$ Comparing with the short exact sequence
$$
\xymatrix{
0 \ar[r] & \widehat{A}_{[\varpi]} \ar[r]^{\cdot p} & \widehat{A}_{[\varpi]} \ar[r] & \widehat{A}_{[\varpi]}/p\widehat{A}_{[\varpi]} \ar[r] & 0,
}
$$
one deduces the claim.
\\
Now let me first quickly explain why (a) is true: since $\Spec(A)=\lim_{x\in U}\Spec(W(\mathcal{O}^{+}_{S}(U)))$ and then by the behaviour of $\G$-torsors on inverse limits of qcqs schemes along affine maps, one has that
$$
\underset{x\in U}\colim H^{1}(\Spec(W(\mathcal{O}^{+}_{S}(U))), \G_{\Spec(W(\mathcal{O}^{+}_{S}(U)))})\simeq H^{1}(\Spec(A),\G_{\Spec(A)}).
$$
In fact, recall here that one is working with $\G$ smooth, so that one may take the non-abelian cohomology in the étale topology and then one can cite \cite[Thm. 2.1]{MargauxLimit}.
Then consider the following commutative diagram
$$
\xymatrix{
A \ar[r] \ar[d] & W(\mathcal{O}^{+}_{S,x}) \ar[r] \ar[d] & W(\kappa(x)^{+}) \ar[d] \\
\mathcal{O}^{+}_{S,x} \ar[r]^{\id} & \mathcal{O}^{+}_{S,x} \ar[r] & \kappa(x)^{+}.
}
$$
Here the morphism $A\rightarrow \mathcal{O}^{+}_{S,x}$ is induced by the colimit of the reduction modulo $p$ maps $$W(\mathcal{O}^{+}_{S}(U))\rightarrow \mathcal{O}^{+}_{S}(U).$$
Note that all downward arrows are surjections with henselian kernel: for $W(\kappa(x)^{+})\rightarrow \kappa(x)^{+}$ and $W(\mathcal{O}^{+}_{S,x})\rightarrow \mathcal{O}^{+}_{S,x}$ this is just the fact that the relevant rings of Witt vectors are $p$-adic (one takes here just the ring of Witt vectors of perfect rings in char $p$) and for the last surjection, just use the same observation and the fact that henselian surjections are closed under filtered colimits. 
\\
Furthermore, note that $\mathcal{O}^{+}_{S,x}$ is a filtered colimit of the $\varpi$-henselian rings $\mathcal{O}^{+}_{S}(U).$ Using invariance of $\G$-torsors under henselian pairs, as supplied by Bouthier-Česnavičius \cite[Thm 2.1.6]{BouthierCesnavicius}, for both $(\mathcal{O}^{+}_{S,x},(\varpi))$ and $(\kappa(x)^{+},(\varpi)),$
one obtains that
$$
H^{1}(\Spec(\mathcal{O}^{+}_{S,x}),\G)\simeq H^{1}(\Spec(\kappa(x)^{+}),\G).
$$
The above commutative diagram induces a corresponding commutative diagram on isomorphism classes of $\G$-torsors:
$$
\xymatrix{
H^{1}(\Spec(A),\G) \ar[d]^{\simeq} \ar[r] & H^{1}(\Spec(W(\mathcal{O}^{+}_{S,x})),\G) \ar[d]^{\simeq} \ar[r] & H^{1}(\Spec(W(\kappa(x)^{+})),\G) \ar[d]^{\simeq} \\
H^{1}(\Spec(\mathcal{O}^{+}_{S,x}),\G) \ar[r]^{\id} & H^{1}(\Spec(\mathcal{O}^{+}_{S,x}),\G) \ar[r]^{\simeq} &  H^{1}(\Spec(\kappa(x)^{+}),\G).
}
$$
Here I am allowed to write the isomorphism symbols in the above commutative diagram, because one may use invariance under henselian pairs again. This terminates the verification of part (a).
\\
Now let me turn to part (b). Here the key input will be yet another result of Bouthier-Česnavičius, which one applies twice. 
\\
Recall that I still denote $A=\colim_{x\in U}W(\mathcal{O}^{+}_{S}(U)).$ I explained before that $A$ is $[\varpi]$-torsion free. Then $(A,[\varpi],I=(1))$ is a bounded Gabber-Ramero triple, c.f. \cite[section 2.1.9]{BouthierCesnavicius}. Recall that they call a Gabber-Ramero triple $(A,t,I)$ henselian, if the pair $(A,tI)$ is henselian. Since all the rings $W(\mathcal{O}^{+}_{S}(U))$ are also $[\varpi]$-adic and henselian pairs are closed under filtered colimits, one sees that $(A,[\varpi],I=(1))$ is a henselian Gabber-Ramero triple. Denote by $(A_{1},[\varpi],I=(1))$ with $A_{1}=(\widehat{A})_{([\varpi])}$ the bounded Gabber-Ramero triple corresponding to the $[\varpi]$-adic completion of $A$ (recall that $[\varpi]$ is still regular in $A_{1}$); this is again a henselian Gabber-Ramero triple and the morphism of Gabber-Ramero triples $(A,[\varpi],I=(1))\rightarrow (A_{1},[\varpi],I=(1))$ satisfies the requirements of \cite{BouthierCesnavicius} Theorem 2.3.3. (c) (note that by assumption $\G$ is smooth over $\mathbb{Z}_{p}$ and affine, all their requirements for the group are met). 
\\
Then one has the inclusions
$$
\Spec(A[1/[\varpi]])\subseteq \Spec(A)-V(p,[\varpi]) \subseteq \Spec(A),
$$
and $\Spec(A_{1})-V(p,[\varpi])\rightarrow \Spec(A)-V(p,[\varpi])$ is the pullback $$\Spec(A)-V(p,[\varpi])\times_{\Spec(A)}\Spec(A_{1})\rightarrow \Spec(A)-V(p,[\varpi]),$$ so that one can apply their statement to get an isomorphism
$$
H^{1}(\Spec(A)-V(p,[\varpi]),\G)\simeq H^{1}(\Spec(A_{1})-V(p,[\varpi]),\G).
$$
Next, consider the Gabber-Ramero triple $(A_{1},p,I=(1)).$ For the boundedness, recall that it was already verified that $p$ is regular in $A_{1}.$ Next, observe that $(A_{1},p)$ is henselian. For this, it is enough (because henselian pairs are closed under limits and nilpotent thickenings) to see that $A_{1}/[\varpi]=A/[\varpi]$ is henselian along $p,$ which is ok because $A$ is henselian along $p$ and henselian pairs are preserved under quotients. Now, since also
$$
\Spec(A_{1}[1/p])\subseteq \Spec(A_{1})-V(p,[\varpi])\subseteq \Spec(A_{1}),
$$
one can apply \cite[Theorem 2.3.3.(c)]{BouthierCesnavicius} again - but this time applied to the morphism of bounded and henselian Gabber-Ramero triples $(A_{1},p,I=(1))\rightarrow (\widehat{A_{1}}_{(p)}, p, I=(1))$, to conclude that
$$
H^{1}(\Spec(A_{1})-V(p,[\varpi]),\G)\simeq H^{1}(\Spec(\widehat{A_{1}}_{(p)})-V(p,[\varpi]),\G).
$$
To get the statement of the sub-Lemma, I am proving right now, on uses the previous discussion comparing completions, which implied that
$$
\widehat{A_{1}}_{(p)}\simeq W(\kappa(x)^{+}).
$$
\end{proof}
Having this lemma in the pocket, one may finish the proof the main statement, I am hunting right now: recall $\mathcal{E}^{\text{approx}}$ the $\G$-torsor over $\Spec(W(R^{+}))-V(p,[\varpi])$ that was constructed previously via glueing. Let $x\in \Spa(R,R^{+})$ be a point, then consider the morphism on 'punctured spectra'
$$
\dot{f}_{x}\colon \Spec(W(\kappa(x)^{+}))-V(p,[\varpi])\rightarrow \Spec(W(R^{+}))-V(p,[\varpi])
$$
and the pullback $\dot{f}_{x}^{*}(\mathcal{E}^{\text{approx}}).$ 
\\
Then one can find a unique, up to isomorphism, $\G$-torsor $\mathcal{M}^{\prime}_{x}$ which is defined on $\Spec(W(\kappa(x)^{+})),$  extending $\dot{f}_{x}^{*}(\mathcal{E}^{\text{approx}}):$ here one uses on the one hand a result of Kedlaya, \cite[Thm. 2.7]{KedlayaAinf} (c.f. also with \cite[Prop. 14.2.6.]{Berkeleylectures}, which is verbatim the statement I am using here), and on the other hand a Lemma of Anschütz, \cite[Prop. 8.5]{JohannesExtending}.
\\
 By part (a) of the above lemma, one finds a $\G$-torsor $\mathcal{M}_{U},$ which is defined over $\Spec(W(\mathcal{O}^{+}_{S}(U))),$ such that its pullback along $\Spec(W(\kappa(x)^{+}))\rightarrow \Spec(W(\mathcal{O}^{+}_{S}(U)))$ is isomorphic to $\mathcal{M}^{\prime}_{x}.$ Now a priori, the pullback $\mathcal{E}^{\prime}_{U}$ of the given $\G$-torsor $\mathcal{E}^{\text{approx}}$ over $\Spec(W(R^{+}))-V(p,[\varpi])$ towards $\Spec(W(\mathcal{O}^{+}_{S}(U)))-V(p,[\varpi])$ may not be isomorphic to the restriction of $\mathcal{M}_{U}$ towards the 'punctured spectrum' $\Spec(W(\mathcal{O}^{+}_{S}(U)))-V(p,[\varpi]).$ But by construction both are isomorphic when further pulled back to $\Spec(W(\kappa(x)^{+}))-V(p,[\varpi]),$ so that by part (b) of the above lemma, one finds a rational open $x\in V\subseteq U,$ such that they are both isomorphic over $\Spec(W(\mathcal{O}^{+}_{S}(V)))-V(p,[\varpi]).$
 \\
Therefore, consider the $\G$-torsor $\mathcal{M}_{V}=f_{U,V}^{*}(\mathcal{M}_{U}),$ where $f_{U,V}\colon \Spec(W(\mathcal{O}^{+}_{S}(V)))\rightarrow \Spec(W(\mathcal{O}^{+}_{S}(U)));$ it extends the $\G$-torsor $\dot{f}_{V}^{*}(\mathcal{E}^{\text{approx}}).$ It follows that the Frobenius-structure and quasi-isogeny on $$\dot{f}_{V}^{*}(\mathcal{E}^{\text{approx}})$$ induces the same-structure on $\mathcal{M}_{V},$ so that one gets a $\G$-Breuil-Kisin-Fargues module $(\mathcal{M}_{V},\varphi_{\mathcal{M}_{V}})$ of type $\mu$ over $\Spec(\mathcal{O}^{+}_{S^{\sharp}}(V^{\sharp}))$ with a $G$-quasi-isogeny $\rho_{V}.$ 
\\
Running over all points $x\in \Spa(R,R^{+}),$ one obtains the desired covering by rational opens of $S.$ It remains to see that the tuples $((\mathcal{M}_{V},\varphi_{\mathcal{M}_{V}}),\rho_{V})$ are isomorphic on overlaps and satisfy the cocycle condition. Since one already knows at this point that $\psi$ is injective, one can deduce that the tuples $((\mathcal{M}_{V},\varphi_{\mathcal{M}_{V}}),\rho_{V})$ are isomorphic on overlaps and they satisfy the cocycle condition because the restriction from $\G$-Breuil-Kisin-Fargues modules over some integral perfectoid ring $A$ towards $\G$-shtuka is faithful (but not fully faithful in general! c.f. \cite[Proposition 2.2.6]{RapoPappasShtuka}. \footnote{The worry with fullness in general is that when one restricts $\varphi$-modules from $\mathcal{Y}_{[\rho,\infty]}$ towards $\mathcal{Y}_{[\rho,\infty)}$ the target category does not care about the '+'-component of the affinoid Huber pair, while the source category certainly does.}).
\end{proof}
\begin{Remark}\label{Input Bhatt zum Ausdehnen}\footnote{This remark was born after a remark given to me by Johannes Anschütz after I told him about the previous lemma.}
Here I want to give a slightly different argument for the extension after passing to an analytic covering of $\Spa(R,R^{+}),$ which , instead of using the results of Bouthier-Česnavičius, makes use of results of Bhatt \cite{BhattTannaka} ; I will stick to the case of $\G=\GL_{n}$ for simplicity.
\\
One looks at the following functor parametrizing all possible extensions of the vector bundle $\mathcal{N}$ on $\Spec(W(R^{+}))-V(p,[\varpi]):$
$$
Extn(\mathcal{N})\colon \text{Affd perfectoids}/\Spa(R,R^{+})^{\op}\rightarrow \text{Set},
$$ 
sending $g\colon \Spa(S,S^{+})\rightarrow \Spa(R,R^{+})$ towards pairs, up to equivalence,
$$
\lbrace \mathcal{M} \text{ rank }n\text{ vector bundle on }\Spec(W(S^{+})), \rho\colon \mathcal{M}[\frac{1}{(p,[\varpi])}]\simeq \dot{g}^{*}(\mathcal{N}) \rbrace.
$$
Here I denote by 
$$
\dot{g}\colon \Spec(W(S^{+}))-V(p,[\varpi])\rightarrow \Spec(W(R^{+}))-V(p,[\varpi])
$$
the induced morphism on 'punctured spectra' (well, $V(p,[\varpi])$ is not really a point here anymore...) and by $\mathcal{M}[\frac{1}{(p,[\varpi])}]$ the restriction of $\mathcal{M}$ to this 'punctured spectrum'.
Let $x\in \Spa(R,R^{+}).$ Then one can use again the result of Kedlaya \cite[Thm. 2.7]{KedlayaAinf}, which implies that one has a point $(\mathcal{M}_{x},\rho_{x}=\text{can})\in Extn(\mathcal{N})(\Spa(\kappa(x),\kappa(x)^{+})).$ The key claim is now that
$$
Extn(\mathcal{N})(\Spa(\kappa(x),\kappa(x)^{+}))=\underset{x\in U}\colim \text{ }Extn(\mathcal{N})(\Spa(\mathcal{O}_{S}(U),\mathcal{O}^{+}_{S}(U))),
$$
where the colimit runs over all rational opens $U\subseteq S$ containing $x.$
This would be sufficient to construct the desired extension after passing to a rational covering.
As
$$
\kappa(x)^{+}=\widehat{(\underset{x\in U}\colim\mathcal{O}^{+}_{S}(U))}_{(\varpi)},
$$
the previous claim is really about commutation with completed colimits and to handle this I will use Bhatt's results on generalizations of Beauville-Laszlo glueing
: if $\pi\colon Y\rightarrow X$ is a morphism of schemes, $Z\subset X$ is a constructible closed, $U=X-Z,$ $V=Y-\pi^{-1}(Z),$ such that
$$
Z\times_{X}^{L}Y\simeq Z,
$$
then 
$$
\Vect(X)\simeq \Vect(Y)\times_{\Vect(V)} \Vect(U).
$$
This is \cite{BhattTannaka}, Prop. 5.6. - taken together with loc.cit. Lemma 5.12. I want to apply this statement with $X=\Spec(A),Z=V(p,[\varpi]),Y=\Spec(\widehat{A}_{(p,[\varpi])}),$ where $A=\colim_{x\in U}W(\mathcal{O}^{+}_{S}(U));$ recall that $ \widehat{A}_{(p,[\varpi])}\simeq W(\kappa(x)^{+}).$\footnote{Indeed, I checked above that $\widehat{(\widehat{A}_{[\varpi]})}_{p} \simeq W(\kappa(x)^{+}),$ but note that this also implies that $\widehat{A}_{(p,[\varpi])}\simeq W(\kappa(x)^{+}).$ Namely, since $(p,[\varpi])$ is a regular sequence, one obtains that
$$
\widehat{A}_{(p,[\varpi])}\simeq R\lim_{n}K^{\bullet}_{A}(p^{n},[\varpi]^{n})=\widehat{A}^{\mathbb{L}}_{(p,[\varpi])}.
$$
But then, $\widehat{A}^{\mathbb{L}}_{(p,[\varpi])}\simeq \widehat{(\widehat{A}^{\mathbb{L}}_{[\varpi]})}^{\mathbb{L}}_{p},$ but one knows already that both completions on the right agree with their classical counterpart; finishing the small argument.
} The task is therefore to see that
$$
A/(p,[\varpi])\otimes_{A}^{L}\widehat{A}_{(p,[\varpi])}\simeq \widehat{A}_{(p,[\varpi])}/(p,[\varpi]),
$$
since this is just $A/(p,[\varpi]).$ For this, recall that $(p,[\varpi])$ was a regular sequence in both $A$ and $\widehat{A}_{(p,[\varpi])}.$ It follows that the Koszul-complex $K^{\bullet}_{A}(p,[\varpi])$ is a free resolution of $A/(p,[\varpi])$ and that $K^{\bullet}_{\widehat{A}_{(p,[\varpi])}}(p,[\varpi])$ is a free resolution of $\widehat{A}_{(p,[\varpi])}/(p,[\varpi]).$ Then calculate that
\begin{align*}
A/(p,[\varpi])\otimes_{A}^{L}\widehat{A}_{(p,[\varpi])} & \simeq K^{\bullet}_{A}(p,[\varpi])\otimes_{A} \widehat{A}_{(p,[\varpi])} \\
& \simeq  K^{\bullet}_{\widehat{A}_{(p,[\varpi])}}(p,[\varpi]) \\
& \simeq \widehat{A}_{(p,[\varpi])}/(p,[\varpi]) \\
& \simeq A/(p,[\varpi]),
\end{align*}
as desired. The hypothesis in Bhatt's result are therefore satisfied and one has that
$$
\Vect(\Spec(A))\simeq \Vect(\Spec(W(\kappa(x)^{+})))\times_{\Vect(\Spec(W(\kappa(x)^{+})-V(p,[\varpi]))} \Vect(\Spec(A)-V(p,[\varpi])).
$$
Since $\Vect(.)$ commutes with filtered colimits, one may deduce that
$$
Extn(\mathcal{N})(\Spa(\kappa(x),\kappa(x)^{+}))=\underset{x\in U}\colim \text{ } Extn(\mathcal{N})(\Spa(\mathcal{O}_{S}(U),\mathcal{O}^{+}_{S}(U))),
$$
and this finishes the proof.
\end{Remark}

\subsection{Towards the representability of the generic fiber: some conjectures}
Keep the set-up from the previous sections: $(\G,\lbrace \mu \rbrace, [b])$ is an integral unramified local Shimura-datum over $\mathbb{Q}_{p}$ and one still assumes that all slopes of the isocrystal associated to $\Ad(b)$ are $>-1.$ Associated to this data one had the Bültel-Pappas moduli problem of deformations of adjoint nilpotent $\G$-$\mu$-displays $\mathcal{M}^{\BP}$. Still, I have to assume unfortunately that $p>2.$
\\
As said before, it would follow from the representability of the functor $\mathcal{M}^{\BP}$ that the functor $(\mathcal{M}^{\BP})_{\eta}$ defined on $\Affd_{\Spa(\breve{E},\mathcal{O}_{\breve{E}})}$ from above is representable by a rigid analytic space over $\Spa(\breve{E}).$ Using the previous result, Proposition \ref{Vergleich der Diamanten}, I want to explain how one could go about proving the representability of the generic fiber $\mathcal{M}^{\BP}_{\eta};$ at least inside the category of smooth adic spaces over $\Spa(\breve{E}).$\footnote{The reason one has to restrict for now to smooth adic spaces is that I will use the fully faithfullness of the diamond functor; in fact semi-normality would be enough here.}
\\
Therefore, let $\SmAffd_{\Spa(\breve{E})}$ be the category of smooth and complete affinoid adic spaces over $\Spa(\breve{E});$ one may equip this category either with the analytic or étale topology. Now one can restrict both sheaves $\mathcal{M}^{\BP}_{\eta}$ and $\Sh_{K}$ towards the category $\SmAffd_{\Spa(\breve{E})}$ and the previous results will directly imply the existence of a comparison morphism of sheaves.
\\
Namely, if $\Spa(R,R^{\circ})\in \SmAffd_{\Spa(\breve{E})},$ with associated diamond $\Spd(R,R^{\circ})$ over $\Spd(\breve{E}),$ and if $$\Spa(R,R^{\circ})\rightarrow \mathcal{M}^{\BP}_{\eta}$$ is a $\Spa(R,R^{\circ})$-valued point, then one can apply the diamond functor to obtain a morphism of pro-étale sheaves over $\Spd(\breve{E}),$
$$
\Spd(R,R^{\circ})\rightarrow (\mathcal{M}^{\BP}_{\eta})^{\lozenge}.
$$
Then use the comparison of the diamonds, i.e. Prop. \ref{Vergleich der Diamanten}, to obtain a morphism of pro-étale sheaves over $\Spd(\breve{E})$
$$
\Spd(R,R^{\circ})\rightarrow (\Sh_{K})^{\lozenge}.
$$
By fully faithfulness of the diamond functor, \cite[Prop. 10.2.3.]{Berkeleylectures}, from smooth adic spaces over $\Spa(\breve{E})$ towards diamonds over $\Spd(\breve{E})$ (one has to remember the structure morphism!), one now obtains a morphism of smooth adic spaces over $\Spa(\breve{E}),$
$$
\Spa(R,R^{\circ})\rightarrow \Sh_{K}.
$$
This association is functorial in $\Spa(R,R^{\circ}),$ so that in total one constructed a comparison morphism of sheaves
$$
\Upsilon^{\text{sm}}\colon (\mathcal{M}^{\BP})_{\eta}\!\mid_{\SmAffd_{\Spa(\breve{E})}}\rightarrow \Sh_{K}\!\mid_{\SmAffd_{\Spa(\breve{E})}}.
$$
Then there is no way around making the following 
\begin{Conjecture}
The above morphism of sheaves $\Upsilon^{\text{sm}}$ is an isomorphism.
\end{Conjecture}
To attack this conjecture, one can break it up into two pieces: the injectivity and surjectivity. The injectivity should follow from the 'right' construction of the period morphism in terms of $\G$-$\mu$-displays.
\subsection{A conjecture on prismatic crystals}
In this subsection I want to state a conjecture which paves the way to showing that the previously constructed morphism $\Upsilon^{\text{sm}}$ is actually an epimorphism.
\\
Saying that $\Upsilon^{\text{sm}}$ is surjective boils down to the following: let $\mathcal{V}$ be the 'universal crystalline'  pro-étale $\underline{G(\mathbb{Q}_{p})}$-torsor over the admissible locus $\mathcal{F}l^{\text{adm}}(G,\lbrace \mu \rbrace, [b]),$ which falls out of the definition of the admissible locus as those modification of $G$-bundles on the relative Fargues-Fontaine curve, which are fiberwise trivial (\cite[Def. 3.1]{ChenFarguesShen}); plus the translation between fiberwise trivial $G$-torsors on the relative Fargues-Fontaine curve and pro-étale $\underline{G(\mathbb{Q}_{p})}$-torsors on the perfectoid bases one is working over (this pro-étale $\underline{G(\mathbb{Q}_{p})}$-torsors is constructed for example in the proof of \cite[Thm. 3.3]{ChenFarguesShen}). Now to give oneself a morphism
$$
f\colon \Spa(R,R^{\circ})\rightarrow \Sh_{K},
$$ 
where $\Spa(R,R^{\circ})$ is a smooth affinoid rigid-analytic space over $\Spa(\breve{E}),$
is equivalent to give oneself a reduction of the $\underline{G(\mathbb{Q}_{p})}$-torsor $f^{*}(\mathcal{V})$ towards a $\underline{\G(\mathbb{Z}_{p})}$-torsor. 
 \\
The surjectivity of $\Upsilon^{\text{sm}}$ means that one can extend this data - after admissible blow up of a chosen formal model of $\Spa(R,R^{\circ})$ - towards a $\G$-$\mu$-display, together with a quasi-isogeny. 
\\
One observes that this looks like a relative version of the recent result of Bhatt-Scholze \cite{BhattScholzePrismCrys} classifying lattices in crystalline Galois-representations in terms of absolute prismatic $F$-crystals.
\\
In this direction, I will make the following conjecture in the case of the most basic local Shimura-datum $(\GL_{n},\lbrace \mu_{n,d} \rbrace, [b])$ and in the setting, where the smooth, affinoid test adic space $\Spa(R,R^{\circ})$ admits a smooth $p$-adic formal model. 
\\
I have to introduce a minimum of set-up. Let $\Spf(A)$ be a $p$-adically smooth affine formal scheme over $\Spf(\mathcal{O}_{\breve{E}}).$ Let $A_{\text{qsyn}}$ denote the quasi-syntomic site, which, as $A$ is in particular quasi-syntomic, admits as a basis quasi-regular semiperfectoid $A$-algebras; on this basis one has the prismatic structure sheaf $\Prism_{\bullet},$ and also the sheaves $\Prism_{\bullet}[1/p], \Prism_{\bullet}[1/I], \widehat{\Prism_{\bullet}[1/I]}_{p},\Prism_{\bullet}\lbrace I/p \rbrace,$ see \cite{BhattScholzePrismCrys}, Construction 6.2. In the following, I will use $\Vect(A_{\text{qsyn}},?),$ where $$?\in \lbrace \Prism_{\bullet},\Prism_{\bullet}[1/p], \Prism_{\bullet}[1/I], \widehat{\Prism_{\bullet}[1/I]}_{p},\Prism_{\bullet}\lbrace I/p \rbrace \rbrace$$ as in loc.cit. Notation 2.1. Then one defines
$$
\Vect^{\varphi}(A_{\text{qsyn}},\Prism_{\bullet}[1/(p,I)])=\Vect^{\varphi}(A_{\text{qsyn}},\Prism_{\bullet}[1/p])\times_{\Vect^{\varphi}(A_{\text{qsyn}},\Prism_{\bullet}[1/p,1/I])}\Vect^{\varphi}(A_{\text{qsyn}},\Prism_{\bullet}[1/I]),
$$
and an object in this category has as a crystalline specialization: just apply $.\otimes_{\Prism_{\bullet}[1/p]} \Prism_{\bullet}\lbrace I/p \rbrace[1/p]$; and an étale specialization: just apply $.\otimes_{\Prism_{\bullet}} \widehat{\Prism_{\bullet}[1/I]}_{p}.$ Recall that Bhatt-Scholze check that $\mathbb{Z}_{p}$-local systems on the generic fiber of $\Spf(A)$ correspond to $\Vect^{\varphi}(A_{\text{qsyn}}, \widehat{\Prism_{\bullet}[1/I]}_{p}),$ see loc.cit. Cor. 3.8. It therefore makes sense to require that the étale specialization of an object $(\mathcal{N},\varphi_{\mathcal{N}})\in \Vect^{\varphi}(A_{\text{qsyn}},\Prism_{\bullet}[1/(p,I)])$ coincides with a given $\mathbb{Z}_{p}$-local system $\mathcal{T}.$
\\
I will use Koshikawa's log-generalization of prismatic cohomology, as developed in \cite{koshikawaLogPrism}, which allows one to give a definition of the absolute log-prismatic site. Here and in the following I will use the terminology of Hartl \cite[Def. 1.1 and Prop. 1.2.]{HartlSemistable} for what it means to be a semi-stable $p$-adic formal scheme. Let $\Spf(A^{\prime})$ be a small $p$-adic semi-stable affine formal scheme over $\Spf(\mathcal{O}_{\breve{E}}).$ Then one equips $\Spf(A^{\prime})$  with the canonical log-structure (i.e. given by the monoid $M=A^{\prime} \cap (A^{\prime}[1/p])^{*}$) to obtain a $p$-adic log formal scheme as in \cite{koshikawaLogPrism}. I will denote the monoid giving this log-structure by $M_{\text{ss}}.$ The category $\Spf(A^{\prime})_{\Prism_{\log}}$ is the category opposite to the category of all bounded log-prisms $(R,I,M_{R})$ as in loc.cit. Def. 3.3. together with a morphism of $p$-adic formal schemes
$$
f\colon \Spf(R/I)\rightarrow \Spf(A^{\prime}),
$$
and an exact closed immersion of log-formal schemes
$$
(\Spf(R/I,f^{*}(M_{\text{ss}}))\rightarrow (\Spf(R),M_{R})^{a}.
$$
As in loc. cit., after Def. 4.1., one has the structure sheaf $\mathcal{O}_{\Prism_{\text{log}}}$ on the category $(\Spf(A^{\prime}))_{\Prism_{\text{log}}},$ which allows one to make sense of the category of log-prismatic crystals in vector bundles, i.e. of $\Vect((\Spf(A^{\prime}))_{\Prism_{\log}}, \mathcal{O}_{\Prism_{\log}}).$ Similiarly as before, one can define $$\Vect(\Spf(A^{\prime}))_{\Prism_{\log}}, \mathcal{O}_{\Prism_{log}}[1/(p,I_{\Prism_{\log}})]).$$
Let me give $A$ the trivial log-structure, denoted by $M_{\triv}.$
Assume one is given a morphism of affine log-formal schemes
$$
g\colon (\Spf(A^{\prime}),M_{\text{ss}})\rightarrow (\Spf(A),M_{\triv}).
$$
One can pull back the object $$\mathcal{N}\in \Vect(\Spf(A)_{\Prism},\mathcal{O}_{\Prism}[1/(p,I_{\Prism})])$$ towards an object $$g^{*}(\mathcal{N})\in \Vect(\Spf(A^{\prime})_{\Prism_{\log}},\mathcal{O}_{\Prism_{\log}}[1/(p,I_{\Prism_{\log}})]).$$ Then it makes sense to require the existence of an $\mathcal{M}\in \Vect((\Spf(A^{\prime}))_{\Prism_{\log}}, \mathcal{O}_{\Prism_{\log}}),$ which extends $g^{*}(\mathcal{N}).$ 
\\
Now one goes back to the set-up considered in the beginning of this section: let $\Spf(A)$ be an affine smooth $p$-adic formal scheme over $\Spf(\mathcal{O}_{\breve{E}})$  and one is given a morphism
$$
f\colon \Spf(A)^{\text{ad}}_{\eta}\rightarrow \Sh_{K}.
$$
Recall that one denoted by $\mathcal{V}$ the 'universal' crystalline $\mathbb{Q}_{p}$-local system on $\mathcal{F}l^{\text{adm}}(\GL_{n},\lbrace \mu_{n,d} \rbrace, [b]),$ and that it was explained that the datum of the morphism $f$ above is equivalent to the datum of a pair $(\mathcal{T},\rho),$ where $\mathcal{T}$ is a $\mathbb{Z}_{p}$-local system on $\Spf(A)^{\text{ad}}_{\eta}$ and $\rho$ is an isomorphism of $\mathbb{Q}_{p}$-local systems
$$
\rho\colon \mathcal{T}\otimes_{\underline{\mathbb{Z}_{p}}}\underline{\mathbb{Q}_{p}}\simeq f^{*}(\mathcal{V}).
$$
Using the main results of \cite{GuoPrismCrystal}, on can see that  there exists a finite locally free prismatic $F$-crystal on $(A_{\text{qsyn}},\Prism_{\bullet}[1/(p,I)]),$ whose étale specialization is $\mathcal{T}$ and whose crystalline specialization is $(D,\varphi_{D})\otimes_{W(k)[1/p],\varphi}\Prism_{\bullet}\lbrace I/p \rbrace[1/p];$ here $(D,\varphi_{D})$ is some isocrystal representing the class $[b]\in B(\GL_{n}, \mu_{n,d}).$ 
\begin{Conjecture}\label{Vermutung prismatische Kristalle}
Let $\Spf(A)$ be an affine smooth $p$-adic formal scheme over $\Spf(\mathcal{O}_{\breve{E}})$ 
Let $\mathcal{N}\in \Vect(A_{\text{qsyn}},\Prism_{\bullet}[1/(p,I)])$ be a finite locally free crystal. Then there exists a composition of an rig-étale morphism, followed by an admissible blow-up, of $p$-adic formal schemes over $\Spf(\mathcal{O}_{\breve{E}})$ $$g\colon \Spf(A^{\prime})\rightarrow \Spf(A),$$ where $\Spf(A^{\prime})$ is a semi-stable $p$-adic formal $\Spf(\mathcal{O}_{\breve{E}}),$ and a finite locally free log-prismatic for $\mathcal{M}\in \Vect(\Spf(A^{\prime})_{\Prism_{\log}},\mathcal{O}_{\Prism_{\log}}),$ which extends $\mathcal{N}.$
\end{Conjecture}
\begin{Remark}
Let me admit right away that a conjecture in this direction for the local Shimura-datum $$(\GL_{n},\lbrace \mu_{n,d} \rbrace, [b])$$ is probably not so interesting, since in this case one can translate everything back to $p$-divisible groups and use the representability of the corresponding raw Rapoport-Zink space; needless to say the idea is here to find a direct argument, which does not make use of the representability and which also generalizes to a situation where one considers $\G$-torsors - in fact, I am rather optimistic that once the above linear case is understood, one can adapt to the more general situation.
\end{Remark}
Finally let me point out, that in the perfectoid setting one could see the arguments in the proof of Prop. \ref{Vergleich der Diamanten} as basically verifying the above conjecture.
\subsection{Formal models of the tubes}\label{subsection Formal models of the tubes}
As a not too suprising application of the previous results, I want to explain how to check the 'local' representability statement conjectured by Pappas-Rapoport \cite[Conjecture 3.3.4]{RapoPappasShtuka} in the easier non-ramified case. I will use the notation from loc.cit. section 3.3.1. 
\\
Let $(\G,\lbrace \mu \rbrace, [b])$ be an integral non-ramified local Shimura-datum over $\mathbb{Z}_{p}$ and $p\geq 3.$ In this case, one may take for the local model simply the Flag-variety $\G/P_{\mu}$ and then consider the Scholze $v$-sheaf $\mathcal{M}^{\text{int}}_{\text{Scholze}}$ as in \cite[Def. 25.1.1.]{Berkeleylectures}. Using the main results of the PhD-thesis of Gleason \cite{GleasonPhD}, for a point $x\in X_{\mu}(b)(k),$ one may consider the formal completion
$$
\widehat{\mathcal{M}^{\text{int}}_{\text{Scholze}}}_{/_{x}},
$$
(c.f. \cite{RapoPappasShtuka} section (3.3.4.)). This should be a qcqs small $v$-sheaf over $\Spd(\mathcal{O}_{\breve{E}}),$ but I did not find this exact statement in the literature.
\begin{proposition}\label{Formale Modelle fuer die Tubes}
The conjecture \cite[Conjecture 3.3.4]{RapoPappasShtuka} is satisfied.
\end{proposition}
\begin{proof}
Since the perfect affine Deligne-Lustzig variety $X_{\mu}(b)$ is isomorphic to the perfection of the reduced locus of the Bültel-Pappas moduli problem, one can consider $x\in X_{\mu}(b)$ as the datum of a pair $(\mathcal{P}_{x},\rho)\in \mathcal{M}^{\text{BP}}(k).$ 
\\
Now consider the following deformation problem: Let $\text{Art}_{k}$ be the category of augmented local Artin $\mathcal{O}_{\breve{E}}$-algebras (i.e. local Artin-algebras $(A,\mathfrak{m}),$ with an isomorphism $A/\mathfrak{m}\simeq k$). Let
$$
\text{Def}_{x}\colon \text{Art}_{k}\rightarrow \Set
$$
be the functor classifying deformations of $\mathcal{P}_{x}.$ By \cite[section 3.5.9.]{BueltelPappas} this functor is pro-representable by a complete noetherian local and regular ring (in fact, a power series ring over $\mathcal{O}_{\breve{E}}$) $A_{\text{univ}}.$ Since 
$$
W_{\mathbb{Q}}(A_{\text{univ}}/\mathfrak{m}_{\text{univ}}^{n})=W_{\mathbb{Q}}(k),
$$  
one may lift the quasi-isogeny $\rho,$ to get a compatible system of objects $(\mathcal{P}_{n},\rho_{n})\in \mathcal{M}^{\text{BP}}(A_{\text{univ}}/\mathfrak{m}_{\text{univ}}^{n}),$ i.e. one obtains a morphism
$$
\Spf(A_{\text{univ}})\rightarrow \mathcal{M}^{\text{BP}}.
$$
Passing to associated $v$-sheaves on $\text{Perf}_{\overline{k}_{E}},$ and using the extension of the lifting result for crystalline quasi-isogenies as provided by Lemma \ref{Uebersetzung der Qis allgemeiner Fall}, one constructs easily a morphism of $v$-sheaves over $\text{Perf}_{\overline{k}_{E}},$
$$
(\mathcal{M}^{\text{BP}})_{v}\rightarrow \mathcal{M}^{\text{int}}_{\text{Scholze}}.
$$
The composition 
$$
\xymatrix{
(\Spf(A_{\text{univ}}))_{v} \ar[r] & (\mathcal{M}^{\text{BP}})_{v} \ar[r] & \mathcal{M}^{\text{int}}_{\text{Scholze}}
}
$$
factors over $\widehat{\mathcal{M}^{\text{int}}_{\text{Scholze}}}_{/_{x}}\rightarrow \mathcal{M}^{\text{int}}_{\text{Scholze}}.$ The induced morphism
$$
(\Spf(A_{\text{univ}}))_{v}  \rightarrow \widehat{\mathcal{M}^{\text{int}}_{\text{Scholze}}}_{/_{x}}
$$
is qcqs: $(\Spf(A_{\text{univ}}))_{v}\rightarrow \Spf(\mathcal{O}_{\breve{E}})_{v}$ is qcqs and $\widehat{\mathcal{M}^{\text{int}}_{\text{Scholze}}}_{/_{x}}\rightarrow \Spf(\mathcal{O}_{\breve{E}})_{v}$ is qs (as a composite of an open immersion and since the Scholze $v$-sheaf is quasi-separated over $\Spf(\mathcal{O}_{\breve{E}})_{v}$ by a result of Gleason \cite[Prop. 2.25]{GleasonPhD}), so that by e.g. \cite[Prop. 1.8. (iii)]{SGA6} this induced morphism is in fact qcqs. It therefore suffices to check it is an isomorphism on geometric points. This then follows from the results in section \ref{Section Connection to local mixed-characteristic shtukas} in characteristic $0$ and is direct in characteristic $p.$
\end{proof}

\newpage
\bibliography{/Users/sebastianbartling/Documents/Bibtex/mybib}{}
\bibliographystyle{alpha}

\end{document}